\documentclass[11pt, oneside]{book}
\usepackage[ansinew]{inputenc}
\usepackage{amsfonts}
\usepackage{amssymb,amsmath}
\usepackage{bbm}
\usepackage[french, english]{babel}
\usepackage{latexsym}
\usepackage{natbib}

\usepackage{hyperref}
\usepackage[numbered]{bookmark}
\usepackage{stmaryrd}

\usepackage[T1]{fontenc}

\usepackage{amsfonts,amssymb, amsthm,amsmath, amscd, dsfont}
\usepackage{mathrsfs}
\usepackage{bbm}
\usepackage{graphicx}

\newcommand {\debeq}	{\begin{eqnarray*}}
\newcommand {\fineq}	{\end{eqnarray*}}

\newcommand	{\intgen}	
{\int_0^\infty}

\newcommand	{\dtf}	{\bar{d}_f}

\newcommand{\tr}{\mathbbm{t}}
\newcommand{\trT}{\tr^{\{T\}}}
\newcommand{\dottrT}{\dot\tr^{\{T\}}}

\newcommand{\gw}{\mathcal{B}}

\newcommand{\mc}{\mathcal{M}}
\newcommand{\xc}{\mathcal{X}}
\newcommand{\hc}{\mathcal{H}}
\newcommand{\tc}{\mathcal{T}}
\newcommand{\vc}{\mathcal{V}}
\newcommand{\ec}{\mathcal{E}}
\newcommand{\yc}{\mathcal{Y}}
\newcommand{\kc}{\mathcal{K}}

\newcommand{\R}{\mathbb{R}}

\newcommand{\Z}{\mathbb{Z}}

\newcommand{\N}{\mathbb{N}}

\newcommand	{\E}{\mathbb{E}}
\newcommand{\U}{\mathbb{U}}

\newcommand{\Dc}{\mathscr{D}}
\newcommand{\Ec}{\mathscr{E}}

\newcommand{\Jc}{\mathscr{J}}
\newcommand{\Lc}{\mathscr{L}}

\newcommand{\Rc}{\mathscr{R}}

\newcommand{\Tc}{\mathscr{T}}

\newcommand{\Uc}{\mathscr{U}}

\newcommand{\Tn}{\Tc_n}
\newcommand{\Tnl}{\Tc_n^\ell}
\newcommand{\Rn}{\Rc_n}
\newcommand{\Rnl}{\Rc_n^\ell}

\newcommand{\URnl}{\texttt{\rm Unif}_{\Rnl}}
\newcommand{\UTnl}{\texttt{\rm Unif}_{\Tnl}}

\newcommand{\Pnurt}{P_n^{\texttt{\rm urt}}}
\newcommand{\Pnerm}{P_n^{\texttt{\rm erm}}}
\newcommand{\Pnpda}{P_n^{\texttt{\rm pda}}}

\newcommand{\Ppda} [1] {P_{#1}^{\texttt{\rm pda}}}

\newcommand{\qnerm}{q_n^{\texttt{\rm erm}}}
\newcommand{\qnpda}{q_n^{\texttt{\rm pda}}}

\def \Ã©{\'e}
\def \Ã¨{\`{e}}
\def \Ã¹{\`{u}}
\def \Ã®{\^{\i}}
\def \Ã {\`{a}}
\def \Ã§{\c{c}}

\newcommand	{\RR}{\mathbb{R}}
\newcommand	{\PP}{\mathbb{P}}
\newcommand	{\EE}{\mathbb{E}}

\newtheorem	{thm}		{Theorem}[section]
\newtheorem	{dfn}		[thm]{Definition}
\newtheorem	{lem} 	[thm]	{Lemma}
\newtheorem	{prop}	[thm]{Proposition}

\newtheorem     {rem}       [thm]    {Remark}

\newtheorem     {expl}        [thm]   {Example}
\newtheorem     {exo}        [thm]   {Exercise}

\newcommand	{\indic}	[1] {{\mathbbm{1}}_{#1}}
\newcommand	{\indicbis}	[1] {\indic{\{#1\}}}

\usepackage{tikz-cd}

\definecolor{couleur}{rgb}{0.4,0.8,0.7}

\begin{document}
\thispagestyle{empty}
\large
\noindent

\begin{center}
\textsc{XIX Escola Brasileira de Probabilidade}

\vskip 0.3cm

\textsc{3--8 de Agosto de 2015}

\vskip 0.3cm

\textsc{São Sebastião-SP, Brasil} 

\end{center}
\vskip 2cm

\centerline{\LARGE \bf %Some 
Probabilistic Models for the (sub)Tree(s) of Life
}

\vskip 1cm
\centerline{\Large \bf  Amaury Lambert\footnote{UPMC Univ Paris 06. Please email your comments to \texttt{amaury.lambert@upmc.fr} and check updates at
\texttt{http://www.lpma-paris.fr/pageperso/amaury/}
}}

\vskip 2cm
\paragraph{Abstract:}

The goal of these lectures is to review some mathematical aspects of random tree models used in evolutionary biology to model species trees. 

We start with stochastic models of tree shapes (finite trees without edge lengths), culminating in the $\beta$-family of Aldous' branching models. 

We next introduce real trees (trees as metric spaces) and show how to study them through their contour, provided they are properly measured and ordered. 

We then focus on the reduced tree, or coalescent tree, which is the tree spanned by species alive at the same fixed time. We show how reduced trees, like any compact ultrametric space, can be represented in a simple way via the so-called comb metric. Beautiful examples of random combs include the Kingman coalescent and coalescent point processes. 

We end up displaying some recent biological applications of coalescent point processes to the inference of species diversification, to conservation biology and to epidemiology.


\tableofcontents

\chapter{Tree shapes}

Standard references on the topic of this chapter include: \citet{BGBook, StaBook, SSBook}.

\section{Definitions}

In these lecture notes, we call \emph{tree shape} any rooted, finite tree (no edge lengths, no plane embedding). A finite tree  $\tau$ is an acyclic graph $(\vc, \ec)$ where $\vc=\vc(\tau)$ is a finite set whose elements are called the \emph{vertices} or \emph{nodes} of $\tau$, and $\ec =\ec (\tau)$ is a subset of $\vc\times \vc$ whose elements are (not ordered and) called the \emph{edges} of $\tau$. The root of $\tau$ is a distinguished vertex of $\tau$.\\
\\ 
We will also use the following terminology.
\begin{itemize}
\item
Degree. The degree of a node $u$ of the tree shape $\tau$ is the number of its neighbors, where a neighbor of $u$ is an element $v$ of $\tau$ such that $(u,v)$ is an edge of $\tau$. 
%the number of connected components of $\tau \setminus\{u\}$. 
We will always assume that the root has degree 1. 

\item
Partial order. If $u$ and $v$ are two nodes of $\tau$, we say that $v$ is descending from $u$ and we write $u\preceq v$, if $v$ is not in the same connected component of $\tau\setminus \{u\}$ as the root. The relation $\preceq$ is a (partial) order on the vertex set of $\tau$.
  
\item Tips/leaves. All nodes with degree 1 but the root, are called \emph{tips} or \emph{leaves}. All other nodes (including the root) are called \emph{internal nodes} (or internal vertices). 

\item
A tree shape is said \emph{binary} when each of its internal nodes (but the root) has degree 3.
\end{itemize}
In the next chapter, we will introduce a framework not needed at this stage, known as Ulam--Harris--Neveu coding, in which finite trees can be embedded into the set $\Uc$ of finite words, where a word $v$ descends from a word $u$ if $u$ is a prefix of $v$.

We also mention that in combinatorial phylogenetics, a tree whose tips are labelled by some finite set $X$ is usually identified as a so-called \emph{$X$-hierarchy} (\citealt{SSBook}). An $X$-hierarchy is a collection $\hc$ of subsets of $X$ containing all singletons and such that for any $A,B\in\hc$, $A\cap B\in\{\varnothing, A, B\}$. 
\begin{exo}
Display the labelled tree coded by the $\{1,2,3,4,5\}$-hierarchy $\hc=\{\{1\}, \{1,4\}, \{1,4,5\}, \{2\}, \{2,3\}, \{3\}, \{4\}, \{5\}\}$.
\end{exo}
%((23)(14)5))

\section{Combinatorics warm up}

\subsection{Counting trees}

\begin{dfn}
We let $\Tn$ denote the set of all binary tree shapes with $n$ tips, and $\Tnl$ the set of binary tree shapes with $n$ tips labelled by $\{1,\ldots, n\}$. The elements of $\Tnl$ are called \emph{cladograms} or \emph{labelled tree shapes} (with $n$ tips). 
\end{dfn}
\begin{dfn}
For any (labelled or not) tree shape $\tau$, we call \emph{radial order}  any total order $\sqsubset$ on the internal nodes of $\tau$ respecting the genealogical order. In other words, $\sqsubset$ is a radial order if for any internal nodes $u$ and $v$ of $\tau$,
$$
u\preceq v\Rightarrow u\sqsubset v,
$$
where we recall that $u\preceq v$ means that $v$ is descending from $u$. 
\end{dfn}
\noindent
Note that the minimal element in a radial order is always the root.

\begin{expl}
The archetypal example of radial order is the order in which splits occur through continuous time in a tree produced by a birth--death process. Note that this is specific to continuous time, since in discrete time there are always several nodes with the same generation (i.e., the same graph distance to the root). 
\end{expl}
\begin{dfn}
We let $\Rn$ denote the set of all binary tree shapes with $n$ tips endowed with a radial order. The elements of $\Rn$ are called \emph{ranked tree shapes} (with $n$ tips). 

Also, we let $\Rnl$ denote the set of ranked tree shapes with $n$ tips labelled by $\{1,\ldots, n\}$. The elements of $\Rnl$ are called \emph{labelled, ranked tree shapes} (with $n$ tips). 
\end{dfn}
\noindent
There is a canonical surjection $\ell:\Tnl\rightarrow \Tn$ mapping each labelled tree shape to the same tree shape without labels. We use the same notation for the surjection $\ell:\Rnl\rightarrow \Rn$. 

Similarly, there is a canonical surjection $r:\Rn\rightarrow \Tn$ mapping each ranked tree shape to the same tree shape without radial order. We use the same notation for the surjection $r:\Rnl\rightarrow \Tnl$. 

The mappings $\ell$ and $r$ are sometimes called the \emph{forgetful maps}, because they consist in `forgetting' the labels or the ranks, respectively.
It is obvious that $\ell\circ r = r\circ\ell$, so that the following diagram commutes.
\[
\begin{tikzcd}
\Huge
\Rnl \arrow{r}{r} \arrow[swap]{d}{\ell} & \Tnl \arrow{d}{\ell} \\
\Rn  \arrow{r}{r} & \Tn
\end{tikzcd}
\]
We now give explicit expressions for the cardinal numbers of $\Tnl$ and $\Rnl$.
\begin{prop}[\citealt{Mur84}] 
\label{prop:cardinals}
For each $n\ge 2$, set $t_n:=\#\Tnl$ and $r_n:=\#\Rnl$. Then
$$
t_n= (2n-3)!!
:=(2n-3)(2n-5)\cdots(3)(1)
$$
$$
r_n=\frac{n!\,(n-1)!}{2^{n-1} }
$$
\end{prop}
\begin{rem} No explicit expression is known for $\#\Tn$ (sometimes called Wedderburn-Etherington number) or $\#\Rn$.
\end{rem}
\begin{proof}
We reason by induction.
Let $p_n:\Tc_{n+1}^\ell\rightarrow \Tnl$ denote the projection which maps each $\tau$ with $n+1$ labelled tips to the tree spanned by the tips of $\tau$ carrying labels in $\{1,\ldots,n\}$. 
Note that 
$$
t_{n+1} = \sum_{\tau\in \Tnl}\#p_n^{-1}(\{\tau\}).
$$
Now let us compute $\#p_n^{-1}(\{\tau\})$.
It is immediate that each tree shape with $n$ labelled tips $\tau\in \Tnl$ has $n$ external edges, $n-2$ internal edges and a root edge. This gives $2n-1$ distinct locations where to grow a new external edge carrying the label $n+1$, which means that $\#p_n^{-1}(\{\tau\})=2n-1$. In conclusion,  $t_{n+1} = (2n-1) t_n$ and the first result follows, since $t_2=1$.

For ranked tree shapes, we can similarly consider the projection (still denoted) $p_n:\Rc_{n+1}^\ell\rightarrow \Rnl$ and use the similar equality
$$
r_{n+1} = \sum_{\tau\in \Rnl}\#p_n^{-1}(\{\tau\}).
$$
Now let us compute $\#p_n^{-1}(\{\tau\})$.
Each ranked tree shape with $n$ labelled tips $\tau\in \Rnl$ has $n-1$ ordered internal nodes. Growing a new edge with label $n+1$ requires to add a new internal node. There are $n$ distinct locations where to insert it in the radial order, say between the $(k-1)$-st internal node and the $k$-th internal node for $1\le k\le n$. For a given $k$, there are $k$ distinct edges where to grow the new edge, which means  $\#p_n^{-1}(\{\tau\}) = 1+\cdots+n=n(n+1)/2$. In conclusion, $r_{n+1} =   n(n+1) r_n/2$ and the second result follows, since $r_2=1$.
 \end{proof}

\begin{exo}
Check that $t(x):=\sum_{n\ge 1} t_n\displaystyle \frac{x^n}{n!}$ is well defined for $|x|<\frac12$.
By combinatorial arguments, prove that for $n\ge 2$
$$
t_n=\frac 1 2 \sum_{i=1}^{n-1} {n\choose i} t_i\, t_{n-i}
$$
and deduce that $t(x) = x +\frac12 t\left(x\right)^2$.
Conclude that $t(x)=1-\sqrt{1-2x}$.
\end{exo}

\begin{exo}
A tree shape is said \emph{oriented} when each internal vertex has a \emph{left} and a \emph{right} descending subtree. Each orientation gives a tree a unique plane embedding. 
Prove that the number of ranked oriented trees with $n$ tips is $(n-1)!$ and that the number of ranked, oriented trees with $n$ labelled tips is $n!\, (n-1)!$. Also if $o$ denotes the map forgetting orientation, check that for any $\tau\in\Rnl$, $o^{-1}(\{\tau\}) = 2^{n-1}$. This gives another proof of the formula $r_n= \frac{n!\,(n-1)!}{2^{n-1} }$.
\end{exo}

\subsection{Counting rankings and labellings}

\begin{dfn}
For any $\tau\in \Tn$, an internal node $u$ of $\tau$ is said \emph{symmetric} if the two subtrees descending from $u$ (i.e., the two connected components of $\tau\setminus \{u\}$ not containing the root) are identical. 
A particular case of symmetric node is when $u$ subtends (i.e., is the most recent common ancestor of) a \emph{cherry}, that is $u$ only subtends (two) tips.

We denote by $s(\tau)$ the number of symmetric nodes of $\tau$, and by $c(\tau)$ the number of cherries of $\tau$. 

%Also for any $\tau\in \Rn$, $s(\tau):=c\circ r(\tau)$ denotes the number of cherries of $r(\tau)$.
\end{dfn}
\noindent
Assume we were to extend the notion of symmetric node to labelled (resp. ranked) tree shapes, in the sense that the two descending subtrees of a symmetric node should have not only the same shape but also the same tip labels (resp. the same internal node ranks). Then in a labelled or ranked tree shape, only cherries would be symmetric. 
This explains the following convention. 
\begin{dfn}
For any $\tau\in \Tnl\cup\Rn\cup\Rnl$, we denote invariably by $s(\tau)$ or $c(\tau)$ the number of cherries of $\tau$.
\end{dfn}

%why for $\tau\in \Rn$, we abuse notation by using $s(\tau)=c(\tau):=c\circ r(\tau)$, which is justified by the fact that cherries are the only symmetric nodes in ranked tree shapes.

%For tree shapes $\tau$ that are labelled, we will adopt the convention $s(\tau)=c(\tau):=c\circ\ell(\tau)$.

\begin{exo}
Prove that for any unlabelled tree shape $\tau$ with $n$ tips  ($\tau\in\Tn\cup \Rn$), the number of distinct labellings of $\tau$ is
\begin{equation}
\label{eqn:l-1}
\#\ell^{-1}(\{\tau\}) = 2^{-s(\tau)}\, n!
\end{equation}
where we recall that if $\tau$ is ranked ($\tau\in \Rn$), $s(\tau)$ is the number of cherries of $\tau$ (see previous discussion).
\end{exo}

\begin{dfn}
For any tree shape $\tau$ (labelled or not, ranked or not), for any vertex $v\in \vc(\tau)$, if $\tau'$ denotes the subtree descending from $v$, we denote invariably by $\lambda(v)$ or by $\lambda(\tau')$ the number of leaves subtended by $v$, which is also the number of leaves of $\tau'$.
\end{dfn}
\noindent
Note that $\lambda(v)=1$ iff $v$ is a leaf.\\
\\
Eq \eqref{eqn:l-1} gives an explicit expression for the number of distinct labellings of a given tree shape $\tau$. The next statement gives the number of distinct rankings of any given (either labelled or not labelled) tree shape. 
\begin{prop}[\citealt{KnuBook}]
\label{prop:r-1}
For any unranked tree shape $\tau$ with $n$ tips ($\tau\in\Tn\cup\Tnl$), the number of distinct rankings of $\tau$ is
\begin{equation}
\label{eqn:r-1}
\#r^{-1}(\{\tau\}) =2^{c(\tau)-s(\tau)} \,\frac{(n-1)!}{\prod_{v\in\mathring\vc(\tau)}(\lambda(v)-1)},
\end{equation}
where $\mathring\vc(\tau)$ denotes the set of internal vertices of $\tau$. Recall that when $\tau$ is labelled ($\tau\in \Tnl$), $s(\tau)=c(\tau)$.
\end{prop}
\begin{proof}
First note that the total number of internal vertices of $\tau$ is $\#\mathring\vc(\tau)=\lambda(\tau)-1$. 

Now assume that $\tau\in \Tnl$ and let $\tau'$ and $\tau''$ denote the two labelled subtrees descending from the root of $\tau$ (say for example that $\tau'$ is the subtree containing the tip with label 1). Let $k$ denote the number of internal vertices of $\tau'$. Assuming that the $k$ internal nodes of $\tau'$ are ordered and that the $n-k-2$ internal nodes of $\tau''$ are ordered, the number of ways of ordering the internal nodes of $\tau'$ and $\tau''$ with respect to each other is ${n-2}\choose k$. In conclusion,
$$
\#r^{-1}(\{\tau\}) = {{n-2}\choose k} \#r^{-1}(\{\tau'\})\, \#r^{-1}(\{\tau''\}),
$$
which also reads
$$
\frac{\#r^{-1}(\{\tau\})}{(n-1)!} = \frac{1}{n-1} \,\frac{\#r^{-1}(\{\tau'\})}{k!}\, \frac{\#r^{-1}(\{\tau''\})}{(n-2-k)!}.
$$
An immediate induction yields
$$
\frac{\#r^{-1}(\{\tau\})}{(n-1)!}=\prod_{v\in\mathring\vc(\tau)}\frac{1}{\lambda(v)-1},
$$
which is the expected result for labelled tree shapes.

Now assume that $\tau\in\Tn$. The following proof relies on the equality 
$$
\#\ell^{-1}(r^{-1}(\{\tau\})) = \#r^{-1}(\ell^{-1}(\{\tau\})),
$$
due to the fact that $\ell$ and $r$ commute.

First, let $\tilde\tau\in\Rn$ such that $r(\tilde\tau)=\tau$. 
Recall from Eq \eqref{eqn:l-1} that the number of distinct labellings of $\tilde\tau$ is $2^{-s(\tilde\tau)}\,n!= 2^{-c(\tau)}\,n!$, which yields
$$
\#\ell^{-1}(r^{-1}(\{\tau\}))= \#r^{-1}(\{\tau\})\,2^{-c(\tau)}\,n! 
$$
Second, let $\bar\tau\in\Tnl$ such that $\ell(\bar\tau)=\tau$. From what precedes, we know that   
$$
\#r^{-1}(\{\bar\tau\})=(n-1)!\prod_{v\in\mathring\vc(\bar\tau)}\frac{1}{\lambda(v)-1} =(n-1)!\prod_{v\in\mathring\vc(\tau)}\frac{1}{\lambda(v)-1},
$$
which yields
$$
\#r^{-1}(\ell^{-1}(\{\tau\}))=  \#\ell^{-1}(\{\tau\})(n-1)!\prod_{v\in\mathring\vc(\tau)}\frac{1}{\lambda(v)-1} = 2^{-s(\tau)}\,n!\,(n-1)!\prod_{v\in\mathring\vc(\tau)}\frac{1}{\lambda(v)-1},
$$
thanks again to Eq \eqref{eqn:l-1}.
Equalling the expression for $\#\ell^{-1}(r^{-1}(\{\tau\}))$ and the expression for $\#r^{-1}(\ell^{-1}(\{\tau\}))$ provides the final result. 
 \end{proof}

\section{Random tree shapes}

\subsection{Uniform distributions}

\begin{dfn}
Let $\UTnl$ and $\URnl$ denote  the uniform distributions on $\Tnl$ and $\Rnl$ respectively.

We will adopt the following notation.
The distributions  $\Pnpda$ and $\Pnerm$ are the probabilities on $\Tnl$ defined as
$$
\Pnpda:=\UTnl\quad \mbox{ and } \quad \Pnerm = \URnl\circ r^{-1}.
$$
The distribution $\Pnurt$ is the probability on $\Rn$ defined as
$$
\Pnurt:=\URnl\circ \ell^{-1}
$$
\end{dfn}
\noindent
Note that $\Pnurt\circ r^{-1} =\Pnerm\circ \ell^{-1}$ is the push forward of $\URnl$ by $\ell\circ r$.
\begin{rem}
The preceding denominations come from the terminology used in the phylogenetics literature (\citealt{Ald96, Ald01, BF06, Bro94, LS13, SSBook}), where these three TLAs\footnote{Three-Letter Acronym} have the following meanings.
\begin{itemize}
\item PDA stands for `proportional to distinguishable arrangements',
\item ERM stands for `equal rates Markov',
\item URT stands for `uniform on ranked (labelled) trees'.
\end{itemize}
\end{rem}
\noindent
Thanks to Proposition \ref{prop:cardinals}, and to Equations \eqref{eqn:l-1} and \eqref{eqn:r-1}, here are the probabilities of any given tree $\tau$ under various of the previously defined distributions.
\noindent
For any $\tau\in \Tnl$,
$$
\Pnpda(\tau)= \frac{1}{t_n} = \frac{1}{(2n-3)!!}
$$ 
and
$$
\Pnerm (\tau) = \frac{\#r^{-1}(\{\tau\})}{r_n}=\frac{2^{n-1}}{n!}\prod_{v\in\mathring\vc(\tau)}\frac{1}{\lambda(v) -1} 
$$
For any $\tau\in\Rn$,
$$
\Pnurt(\tau) = \frac{\#\ell^{-1}(\{\tau\})}{r_n} = \frac{2^{n-1-c(\tau)}}{(n-1)!}
$$
For any $\tau\in \Tn$,
$$
\Pnpda\circ\ell^{-1}(\tau)= \frac{\#\ell^{-1}(\{\tau\})}{t_n} =\frac{n!\,2^{-s(\tau)}}{t_n}= \frac{2^{n-1-s(\tau)}}{c_{n-1}},
$$ 
where $c_k:=\displaystyle \frac{1}{k+1} {2k \choose k}$ is the $k$-th Catalan number, and
$$
\Pnerm\circ\ell^{-1} (\tau) =\#\ell^{-1}(\{\tau\})\frac{2^{n-1}}{n!}\prod_{v\in\mathring\vc(\tau)}\frac{1}{\lambda(v) -1}=\frac{2^{n-1-s(\tau)}}{\prod_{v\in\mathring\vc(\tau)}(\lambda(v) -1)}
$$

\begin{figure}[!ht]
\includegraphics[width=\textwidth]{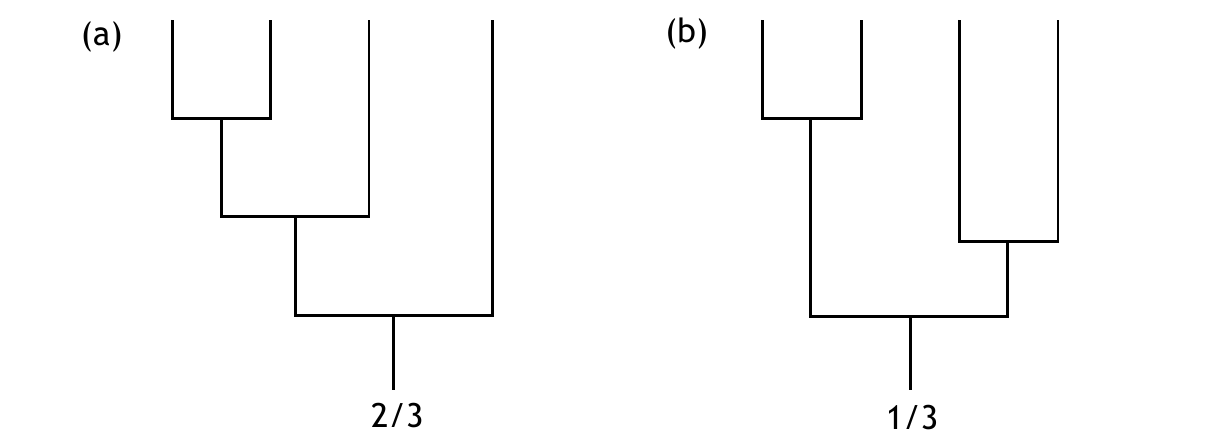}
\caption{Probabilities of trees shapes with $n=4$ tips under $\Pnurt$.
(a) %$(((12)3)4)$ has probability $2/3$
the caterpillar tree with 4 tips has only one conforming ranked shape, with probability $2/3$;
(b) % $((12)(34))$ has probability $1/3$;
the symmetric tree with $4$ tips also has one conforming ranked shape, with probability $1/3$. 
}
\label{fig:fourtips}
\end{figure}

\begin{figure}[!ht]
\includegraphics[width=\textwidth]{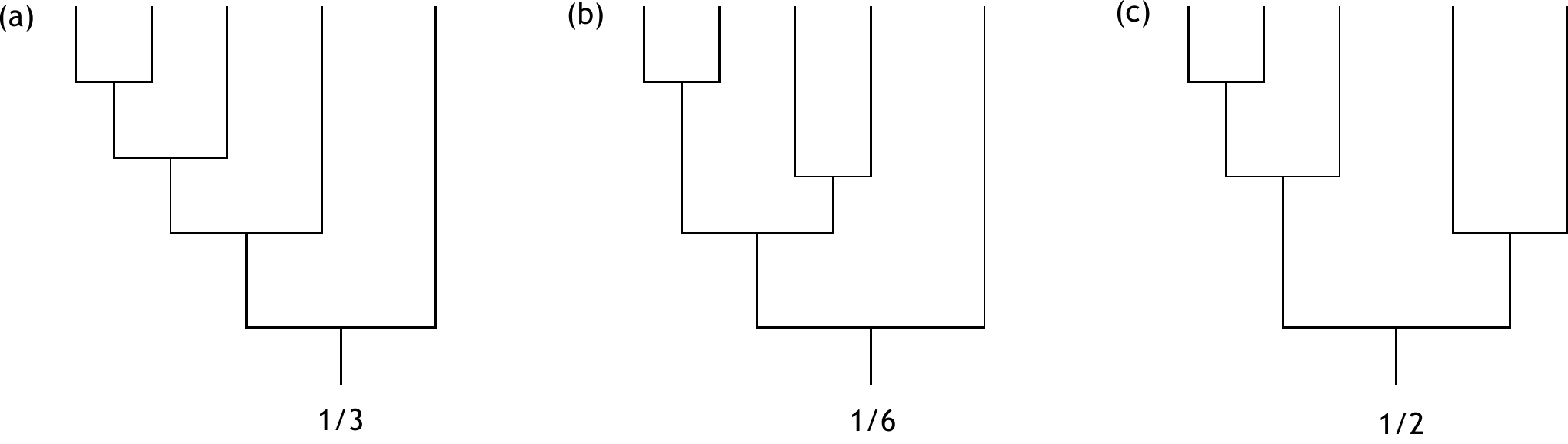}
\caption{Probabilities of trees shapes with $n=5$ tips under $\Pnurt$.
(a) %$((((12)3)4)5)$ has probability $1/3$
is the caterpillar tree with 5 tips, which has only one conforming ranked shape, with probability $1/3$;
(b) % $(((12)(34))5)$ has probability $1/6$;
has one conforming ranked shape with probability $1/6$;
(c) %$(((12)3)(45))$ has probability $1/2$
has 3 conforming ranked shapes, each with probability $1/6$.
}
\label{fig:fivetips}
\end{figure}

\begin{figure}[!ht]
\includegraphics[width=\textwidth]{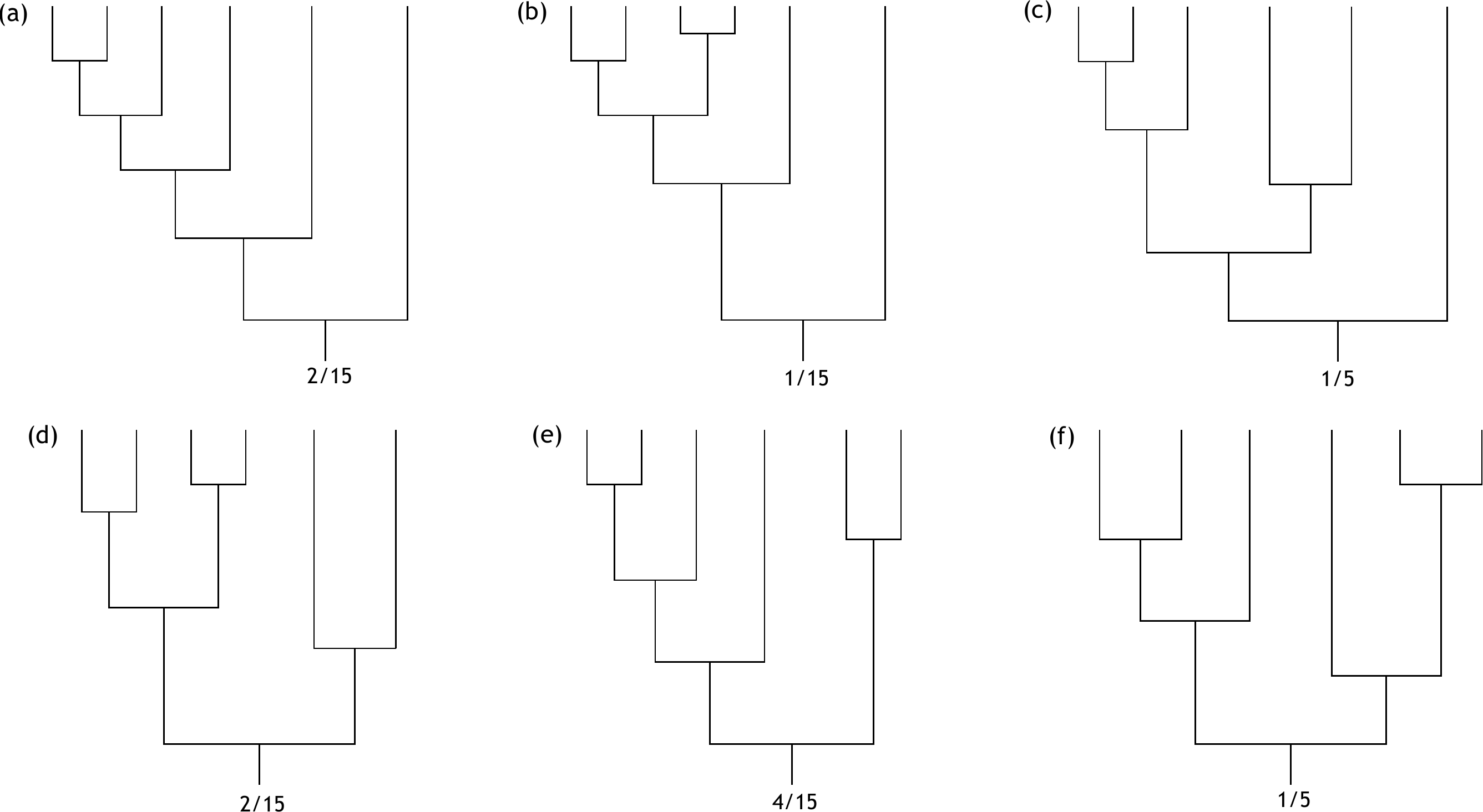}
\caption{Probabilities of trees shapes with $n=6$ tips under $\Pnurt$.
(a) %$(((((12)3)4)5)6)$
is the caterpillar tree with 6 tips, which has only one conforming ranked shape, with probability $2/15$;
(b) %$((((12)(34))5)6)$
 has one conforming ranked shape with probability $1/15$;
(c) %$((((12)3)(45))6)$ has probability $1/5$
has 3 conforming ranked shapes, each with probability $1/15$; 
(d) % $(((12)(34))(56))$ has probability $2/15$ 
has 4 conforming ranked shapes, each with probability $1/30$; 
(e) %$((((12)3)4)(56))$ has probability $4/15$ 
has 4 conforming ranked shapes, each with probability $1/15$;
 (f) %$(((12)3)((45)6))$ has probability $1/5$
 has 3 conforming ranked shapes, each with probability $1/15$.
}
\label{fig:sixtips}
\end{figure}

\begin{exo}
We will see that the tree generated by a Yule (pure-birth) process stopped upon reaching $n$ particles, has the law $\Pnurt$ of the uniform labelled, ranked tree shape after ignoring labels. Check the computations shown for the Yule tree with $n$ tips ($n=4,5,6$) in Figures \ref{fig:fourtips}, \ref{fig:fivetips} and \ref{fig:sixtips}.
\end{exo}

\subsection{The law of your favourite random tree}

\subsubsection{The Bienaymé--Galton--Watson tree shape}

 Let $\gw$ stand for the law of a binary (unlabelled) Galton--Watson tree with parameter $p\in(0,1)$, where $p$ is the probability of begetting two offspring, and $\gw_n$ the law of the same tree conditioned to have $n$ leaves. Set $\sigma_n$ the probability of having $n$ leaves under $\gw$
$$
\sigma_n:=\gw(\lambda = n),
$$ 
so that for each $\tau\in\Tn$,
$$
\gw_n(\tau) = \frac{\gw(\tau)}{\sigma_n}.%\, 1_{\tau=n}.
$$
Now I will leave the reader convince herself (by drawing an example, say the two different binary trees with 4 leaves; or more rigorously, proving that the total number of plane orientations of an unlabelled, unranked tree shape $\tau$ is $2^{n-1-s(\tau)}$) 
 that 
\begin{equation}
\label{eqn:gw}
\gw(\tau) = 2^{n-1-s(\tau)}\left[
\prod_{v\in\mathring\vc(\tau)}p \right]
\left[
\prod_{u\text{ tip of }\tau} (1-p)\right] = 2^{n-1-s(\tau)} \,p^{n-1}\,(1-p)^n .
\end{equation}
Now thanks to Eq \eqref{eqn:l-1}, we know that
$$
t_n=\sum_{\tau\in\Tn}\ell^{-1}(\{\tau\}) = \sum_{\tau\in\Tn}2^{-s(\tau)}\,n!
$$
so that
\begin{equation}
\label{eqn:gw-nbtips}
\sigma_n=\sum_{\tau\in\Tn}2^{n-1-s(\tau)} \,p^{n-1}\,(1-p)^n = \frac{2^{n-1}\,t_n}{n!} \,p^{n-1}\,(1-p)^n.
\end{equation}
In conclusion, for each $\tau\in\Tn$,
$$
\gw_n(\tau) = \frac{\gw(\tau)}{\sigma_n} = \frac{n!\,2^{-s(\tau)}}{t_n}=\frac{\#\ell^{-1}(\{\tau\})}{t_n}=\Pnpda\circ\ell^{-1}(\tau),
$$
which can be recorded in the following statement
\begin{prop}
\label{prop:gw}
For each integer $n\ge 2$, $\gw_n=\Pnpda\circ\ell^{-1}$.
\end{prop}   
\noindent
If the reader is not convinced that \eqref{eqn:gw} does hold, another proof will be given in the context of Markov branching models page \pageref{BGW-back}.

Notice that $\gw_n$ does not depend on $p$. Also note for the record that the probability of a given labelled tree shape under the Galton--Watson model with uniform labelling is $\frac{2^{n-1}}{n!} p^{n-1}(1-p)^n$ (where $n$ is its number of tips).

\subsubsection{The Yule tree shape}

The reason why $\URnl\circ r^{-1}$ is denoted $\Pnerm$ is that it is the law of the (uniformly labelled) tree shape given by the genealogy of a population where all particles split independently and at the same rate $b$, called birth rate, into two new particles (more details to come in Chapter \ref{chap:reduced}). The process counting the size of the population is a Markov process jumping from $k$ to $k+1$ at rate $bk$ and is usually called pure birth process, or Yule (sometimes Yule--Furry) process (see \citealt{Lam08} for an introduction to stochastic models of population dynamics and genealogies). \\
\\
More specifically let $\yc_n$ denote the probability on $\Rn$ defined as the law of the \emph{ranked} tree shape generated by a pure-birth process started at 1 and stopped upon reaching $n$, where the radial order is the chronological order of node splits. It is clear that this probability does not depend on the birth rate $b$.
\begin{prop}
For each integer $n\ge 2$, $\yc_n=\Pnurt$. In particular, the law $\yc_n\circ r^{-1}$ of a Yule tree whose node ranks are ignored is  $\Pnerm\circ \ell^{-1}$.
\end{prop}   
\begin{proof}
Let us prove the proposition by induction on $n$ (the proposition obviously holds for $n=2$). Let $n\ge 2$ and assume the proposition holds for this $n$.
Now let $\tau$ be a ranked tree shape with $n+1$ tips. Let $v$ denote the maximal interior node in the radial order and let $v'$ denote the maximal interior node in the genealogical order of the path from the root to $v$. Define $\hat\tau$ as the ranked tree shape with $n$ tips obtained from $\tau$ by collapsing the cherry subtended by $v$ into a single terminal edge. 
By definition of the Yule process, if $v'$ subtends a cherry in $\hat\tau$, then splitting any of the two tips of this cherry of $\hat\tau$ into a new cherry yields $\tau$. Otherwise, only one tip of $\hat\tau$ can be split to yield $\tau$. This can be expressed as
$$
\yc_{n+1}(\tau) = \frac{2^{1_{\lambda(v')=3}}}{n}\yc_{n}(\hat\tau)
$$
By the induction hypothesis,
$$
\yc_{n+1}(\hat\tau) = \Pnurt(\hat\tau) = \frac{2^{n-1-c(\hat\tau)}}{(n-1)!},
$$
so that
$$
\yc_{n+1}(\tau) = \frac{2^{1_{\lambda(v')=3}}\, 2^{n-1-c(\hat\tau)}}{n!}.
$$
But now check that $c(\tau) = c(\hat\tau)+1_{\lambda(v')\not=3}$, which yields the result.
\end{proof}

\subsubsection{The Kingman tree shape}
Conversely, consider a population where each pair of particles independently merges at the same rate $c$, called coalescence rate (or competition rate) into one single new particle. The process counting the size of the population is a Markov pure-death process jumping from $k$ to $k-1$ at rate $ck(k-1)/2$.\\ 
\\
More specifically, let $\kc_n$ denote the probability on $\Rnl$ of the labelled, ranked tree shape generated by this process started from $n$ labelled particles and naturally stopped when it reaches 1, usually called Kingman $n$-coalescent tree (\citealt{Kin82}). Obviously, $\kc_n$ does not depend on $c$.
\begin{prop}
For each integer $n\ge 2$, $\kc_n=\URnl$. In particular, the law $\kc_n\circ \ell^{-1}$ of a Kingman tree whose labels are ignored is the law $\URnl\circ \ell^{-1}=\yc_n$ of a Yule tree.
\end{prop}   
\begin{proof}
Let us prove the proposition by induction on $n$ (the proposition obviously holds for $n=2$). Let $n\ge 2$ and assume the proposition holds for this $n$. Let $\tau$ be a ranked tree shape with $n+1$ labelled tips. Let $i$ and $j$ be the labels of the cherry subtended by the maximal interior node in the radial order, and as in the previous proof, let $\hat\tau$ be the ranked tree shape with $n$ labelled tips obtained from $\tau$ by collapsing this cherry into one single terminal edge and relabelling the new tip. By induction $\kc_n(\hat\tau)= 1/r_n$, so that 
$$
\kc_{n+1}(\tau) = \frac{2}{n(n+1)} \kc_n(\hat\tau) =  \frac{2}{n(n+1)} \frac{1}{r_n} = \frac{1}{r_{n+1}},
$$
which terminates the proof.
\end{proof}   

\begin{exo}
Explain why the probability distribution $\Pnerm$ puts more weight on balanced trees %(they have more distinct rankings) 
than $\Pnpda$. If $a_n$ denotes the most imbalanced tree with $n$ tips, known as the caterpillar tree (see panel (a) in each of Figures \ref{fig:fourtips}--\ref{fig:sixtips}), prove that
$$
\Pnpda\circ\ell^{-1}(a_n)=\frac{2^{n-2}}{c_{n-1}}\quad\mbox{ and }\quad \Pnerm\circ\ell^{-1}(a_n)=\frac{2^{n-2}}{(n-1)!}.
$$ 
\end{exo}
\begin{table}[!ht]
\begin{center}
\large
%\begin{tabularx}{\textwidth}{ p{3cm} | X | X | X | X | X }
\begin{tabular}{c | c | c | c | c | c}
$n$& $\ 3\ $ &$\ 4\ $&$\ 5\ $&$\ 6\ $&$\ 7\ $\\
\hline
 & & & & &\\
$\Pnpda\circ\ell^{-1}(a_n)$&$1$&$\frac45$&$\frac47$&$\frac8{21}$&$\frac8{33}$\\
 & & & & &\\
\hline
 & & & & &\\
$\Pnerm\circ\ell^{-1}(a_n)$&$1$&$\frac23$&$\frac13$&$\frac2{15}$&$\frac2{45}$\\
 & & & & &
\end{tabular}
\end{center}
\caption{Numerical values of the probability of the caterpillar tree under $\Pnerm$ and $\Pnpda$. Observe the larger weight put on this tree under $\Pnpda$ than under $\Pnerm$.}
\end{table}
\noindent
One can easily see that the probability of the caterpillar tree with $n$ tips, under either of these two distributions, vanishes as $n\to\infty$. Nevertheless, the probability that one of the two subtrees incident to the root subtends exactly one tip (i.e., a long external edge) converges to $\frac14$ under $\Pnpda$, as will be seen in Exercise \ref{exo:qnpda(i)}.\\
\\
It is generally observed that none of the random tree shapes generated by $\Pnerm$ or by $\Pnpda$ statistically give a good fit to empirical species trees. More specifically, real phylogenies are less balanced than random trees under $\Pnerm$ but more balanced than random trees under $\Pnpda$. In the following section, we introduce more general models of random tree shapes, as well as a one-parameter family of tree shape distributions interpolating in particular $\Pnerm$ and $\Pnpda$.

\section{Markov branching models}

\subsection{Definitions and interval splitting}

In what follows, for each $\tau\in \Tn\cup\Tnl$, we will write $\tau =\tau'\oplus\tau''$ to denote the fact that $\tau'$ and $\tau''$ are the two  subtrees descending from the root of $\tau$. It will be convenient to assume that $\tau'$ is chosen uniformly at random among the two possible subtrees (except in the case when $\tau$ is not labelled and the root is symmetric, since then $\tau'=\tau''$). Note that it would also have been possible to select $\tau'$ deterministically as the subtree containing the tip with label 1 say, or the subtree with the smaller number of tips.

\begin{dfn}[\citealt{Ald96}]
\label{dfn:Mbm}
A family of distributions $(P_n)$ on $\Tnl$ is a Markov branching model if there is a family of laws $(q_n)$ on  $\{1,\ldots, n-1\}$ such that $q_n(i) = q_n(n-i)$ for all $i$, and for any $\tau\in \Tnl$,
$$
P_n(\tau)= \frac{2q_n(i)}{{n\choose i}} \, P_i(\tau')\,P_{n-i}(\tau''),
$$
where $\tau=\tau'\oplus\tau''$ and $i=\lambda(\tau')$ (when $n=3, $ and $i=2$, $P_2(\tau')=1=P_1(\tau'')$).
\end{dfn}
\noindent
Note the abuse of notation, since $\tau'$ is not in general labelled by $\{1,\ldots,i\}$. It is implicit that $P_n$ is invariant by permutations of labels ($P_n$ is said exchangeable or equivariant) and defined independently from the chosen label set.
\begin{rem}
If $\tau'$ was chosen to be the subtree containing the tip with label 1 (instead of being chosen at random), the previous display would become
$$
P_n(\tau)= \frac{q_n(i)}{{n-1\choose i-1}} \, P_i(\tau')\,P_{n-i}(\tau'').
$$
A third possibility would be to choose $\tau'$ as the subtree with the smaller number of tips. The details are left to the reader (see also \citealt{Ald01}).
\end{rem}
\noindent
To generate a labelled tree shape with law $P_n$, proceed recursively. 
\begin{enumerate}
\item
Draw a random variable $K_n\in\{1,\ldots,n-1\}$ with law $q_n$.
\item
Conditional on $K_n=i$, select a subset $I$ of $\{1,\ldots,n\}$ with cardinality $i$, uniformly at random.
\item 
Create two edges joining the root of $\tau$ to $\tau'$ and $\tau''$ respectively, where $\tau'$ and $\tau''$ are two independent tree shapes labelled respectively by $I$ and its complement, with respective laws $P_i$ and $P_{n-i}$. 
\end{enumerate}
\begin{exo}
\label{exo:Mbml-1}
If $\tau\in \Tn$, recall that $\tau=\tau'\oplus\tau''$ still makes sense, and check that $s(\tau)=s(\tau')+s(\tau'')$ except when $\tau'=\tau''$, where $s(\tau)=1+s(\tau')+s(\tau'')$. Deduce that
$$
P_n\circ \ell^{-1}(\tau)= 2^{1_{\tau'\not=\tau''}}q_n(i) \, P_i\circ \ell^{-1}(\tau')\,P_{n-i}\circ \ell^{-1}(\tau''),
$$
where again $i=\lambda(\tau')$. Explain why $(P_n\circ \ell^{-1})$ can also be seen as a Markov branching model (on unlabelled tree shapes).
\end{exo}
\noindent
Now we introduce a specific way of designing Markov branching models. 
Let $f$ be a non-negative function on $(0,1)$ such that $f(x)=f(1-x)$ and $\int_0^{1} x(1-x)f(x)\,dx<\infty$. Set 
$$
\alpha_n:= \sum_{i=1}^{n-1} {n\choose i}\int_0^1x^i(1-x)^{n-i} f(x)\, dx = \int_0^1 \big(1-x^n-(1-x)^n\big)\,f(x)\,dx,
$$
which is finite by assumption. Then we define 
$$
q_n^f(i):= \alpha_n^{-1} \, {n\choose i} \int_0^1x^i(1-x)^{n-i} f(x)\, dx,
$$
as well as $(P_n^f)$ the associated Markov branching model.

\noindent
In the case when $f$ is integrable, one can generate a labelled tree shape with law $P_n^f$ by throwing uniformly $n$ points in $(0,1)$ and performing a recursive interval splitting procedure, down until each interval contains at most two of the initial points. Let  $U_1,\ldots, U_n$ be i.i.d. random variables uniform in $(0,1)$. Assume that $\int_0^1 f(x)\, dx=1$.
\begin{enumerate}
\item
Draw an independent r.v. $X$ with density $f$ in $(0,1)$. 
\item 
Let $I$ be the subset of $\{1,\ldots,n\}$ defined by: $i\in I\Leftrightarrow U_i<X$. Then conditional on $X=x$, proceed as in the general case, by putting the labels of $I$ in $\tau'$ and those of its complement in $\tau''$.
\item
Apply recursively the same procedure to the intervals $(0,x)$ and $(x,1)$ independently. 
\end{enumerate}
It is intuitive from this description that the more $f$ puts weight close to the boundaries of $(0,1)$, the more the associated random tree shape is imbalanced.

\subsection{ERM, PDA and Aldous' $\beta$-splitting model}

\begin{thm}[\citealt{Har71, Slo90, Bro94}]
Both $(\Pnerm)$ and $(\Pnpda)$ are Markov branching models with 
$$
\qnpda(i)= \frac{1}{2} {n \choose i} \frac{t_i\, t_{n-i}}{t_n}\quad\mbox{ and }\quad  \qnerm(i)= \frac{1}{n-1} 
$$
\end{thm}
\begin{proof}
We know that for any $\tau\in\Tnl$ written $\tau=\tau'\oplus\tau''$, with $i=\lambda(\tau')$, $\Pnpda(\tau) = \frac{1}{t_n}$, $\Ppda{i}(\tau') = \frac{1}{t_i}$ and $\Ppda{n-i}(\tau'') = \frac{1}{t_{n-i}}$, so that 
$$
\Pnpda(\tau) =  \frac{1}{t_n} =  \frac{t_i\,t_{n-i}}{t_n}\, \Ppda{i}(\tau')\,\Ppda{n-i}(\tau''),
$$
which agrees with the characterization of Markov branching models choosing $\qnpda$ as in the theorem.

We now focus on the distribution of $\lambda(\tau')=:J_n$ under $\Pnerm$. 
Recall that $\Pnerm\circ\ell^{-1}$ is the law of the unranked and unlabelled Yule tree. Consider two particles, a red one and a black one. Give these two particles independent Yule descendances and stop the process when the total number of particles equals $n$. By the strong Markov property of the Yule process (applied at the time it reaches 2), $J_n$ has the same law as the number of red (say) particles when the process stops. 
By the strong Markov property applied at the first time when there are $n-1$ particles, it is easy to see that
$$
\PP(J_n = i) = \frac{i-1}{n-1}\,\PP(J_{n-1}=i-1) +  \frac{n-1-i}{n-1}\,\PP(J_{n-1}=i), 
$$
An immediate induction shows that $\PP(J_n = i) =\displaystyle \frac 1{n-1}$.

Now recall that $\Pnerm$ is also the law of the unranked Kingman coalescent tree. Fixing a subset $I$ of $\{1,\dots,n\}$ with cardinality $i$ and conditioning the Kingman coalescent tree to have on the one hand all lineages initially labelled by $I$ and on the other hand all lineages initially labelled by the complement of $I$, coalesce within each other before coalescing between each other, yields two independent Kingman coalescent trees (forbidding a restricted subset of the pairwise exponential clocks to ring does not alter the independence of the other clocks). Now these two unranked, labelled subtrees follow respectively $P_i^{\texttt{\rm erm}}$ and $P_{n-i}^{\texttt{\rm erm}}$, which yields the result.
\end{proof}
\begin{exo}
\label{exo:qnpda(i)}
Show that $\lim_{n\to\infty} \qnpda (i) = c_{i-1} 4^{-i}$, which `contrasts sharply with the flat distribution $\qnerm$' (\citealt{SSBook}). 
\end{exo}

\begin{table}
\begin{center}
\large
\begin{tabular}{c|c|c|c|c}
$\ i\ $&$\ 1\ $&$\ 2\ $&$\ 3\ $&$\ 4\ $\\
\hline
 & & & & \\
$\lim_n\qnpda(i)$&$\frac14$&$\frac1{16}$&$\frac1{32}$&$\frac5{128}$\\
 & & & & 
\end{tabular}
\end{center}
\caption{The limiting value, as $n\to\infty$, of a basal split $i$ vs $n-i$, under $\Pnpda$.}
\end{table}

\subsubsection{Back to the Galton--Watson tree}
\label{BGW-back}
Here, we want to give a proof of Proposition \ref{prop:gw} via Markov branching models, without using  \eqref{eqn:gw}.
Let $\tau\in\Tn$ and write $\tau=\tau'\oplus\tau''$, $i=\lambda(\tau')$. Then by the branching property,
$$
\gw(\tau) = p\,2^{1_{\tau'\not=\tau''}}\,\gw(\tau')\,\gw(\tau''),
$$
which becomes
$$
\gw_n(\tau) = p\,2^{1_{\tau'\not=\tau''}}\,\frac{\sigma_i\,\sigma_{n-i}}{\sigma_n}\,\gw_i(\tau')\,\gw_{n-i}(\tau'').
$$
Now recalling Exercise \ref{exo:Mbml-1}, this shows that $\gw_n=P_n\circ\ell^{-1}$, where $(P_n)$ is the Markov branching model associated with 
$$
q_n(i)=p\,\frac{\sigma_i\,\sigma_{n-i}}{\sigma_n}
$$
Then it remains to show that
\begin{equation}
\label{eqn:sigman}
\sigma_n= \frac{2^{n-1}\,t_n}{n!} \,p^{n-1}\,(1-p)^n,
\end{equation}
for this will ensure that
$$
q_n(i)= \frac12 { n \choose i} \,\frac{t_i\,t_{n-i}}{t_n},
$$
which indeed is the splitting probability of the PDA model.
Actually, \eqref{eqn:sigman} was already obtained as Eq \eqref{eqn:gw-nbtips}, but this equation was derived from \eqref{eqn:gw}, so we have to prove \eqref{eqn:sigman} by other means. This can actually be done in multiple ways, using for example the Lukasiewicz path associated to the tree and Dwass identity (see for example \citealt{PitBook}).

\subsubsection{Aldous' $\beta$-splitting model}

The $\beta$-splitting model of \citet{Ald96, Ald01} is a one-parameter family of interval splitting branching models. Specifically, for any $\beta\in(-2,+\infty)$, consider
$$
f_\beta(x)= x^\beta (1-x)^\beta\qquad x\in(0,1).
$$
The law $q_n^{f_\beta}$ associated with $f_\beta$ will be denoted $q_n^\beta$.
\begin{exo}
Prove that 
\begin{equation}
\label{eqn:qnbeta}
q_n^\beta (i)= \frac{1}{a_n(\beta)}\, \frac{\Gamma(\beta+i+1)\, \Gamma(\beta+n-i+1)}{\Gamma(i+1)\,\Gamma(n-i+1)} ,
\end{equation}
where $\Gamma$ is the usual Gamma function $\Gamma(x)=\intgen t^{x-1}\, e^{-t}\,dt$, for $x>0$,
and
$$
a_n(\beta) :=\frac{\Gamma(2\beta + n+2)}{\Gamma(n+1)}\,\int_0^1 \big(1-x^n-(1-x)^n\big)\,x^\beta(1-x)^\beta\,dx.
$$
\end{exo}
\noindent
This family has some interesting special cases. As noticed earlier, the balance of the tree increases with $\beta$. As $\beta \to\infty$, the intervals are split deterministically in their middle, while as $\beta\to -2$, the splitting procedure converges to pure erosion, that is $P_n^\beta$ puts weight converging to 1 on the caterpillar tree. The three other cases of interest are $\beta=0$, $\beta =-3/2$ and $\beta=-1$.

\begin{itemize}
\item $\beta=0$. Thanks to Eq \eqref{eqn:qnbeta}, $q_n^0(i)$ does not depend on $i$, which implies $q_n^0(i)=\displaystyle\frac 1{n-1}=\qnerm(i)$, so that
$$
\beta=0\ \Longrightarrow \ P_n^\beta =\Pnerm.
$$
\item $\beta=-3/2$. We are going to prove that
$$
\beta=-3/2\ \Longrightarrow \ P_n^\beta =\Pnpda.
$$
First notice that we can write
$$
t_n= \frac{(2n-3)!}{2^{n-2}\,(n-2)!} = \frac{\Gamma(2n-2)}{2^{n-2}\,\Gamma(n-1)} = 2^{n-1}\,\frac{\Gamma\big(n-\frac 1 2\big)}{\Gamma\big(\frac 1 2\big)},
$$
where we used the identities 
%\begin{equation}
%\label{eqn:gamma-id}
$\Gamma(x)\,\Gamma\big(x+\frac 1 2\big)=2^{1-2x}\,\sqrt\pi \,\Gamma(2x)$
%\end{equation}
 and $\Gamma\big(\frac 1 2\big)=\sqrt\pi$. Now tedious calculations (note the missing `4' in \citealt{Ald96}, equation following Eq 5) show that
$$
a_n\left(-\frac 32\right) = \frac{4\,\Gamma\big(n-\frac 1 2\big)\,\Gamma\big(\frac 1 2\big)}{\Gamma(n+1)},
$$
so that
\begin{multline*}
q_n^{-\frac 3 2}(i) = \frac{\Gamma(n+1)}{4\,\Gamma\big(n-\frac 1 2\big)\,\Gamma\big(\frac 1 2\big)}\,\frac{\Gamma\big(i-\frac 1 2\big)\, \Gamma\big(n-i-\frac 1 2\big)}{\Gamma(i+1)\,\Gamma(n-i+1)} \\
= \frac{1}{2}\, \frac{\Gamma(n+1)}{\Gamma(i+1)\,\Gamma(n-i+1)} \,\frac{t_i\,t_{n-i}}{t_n} = \qnpda(i) .
\end{multline*}
\item $\beta=-1$. This model is sometimes called AB model for `Aldous branching'. First check thanks to \eqref{eqn:qnbeta} that
$$
q_n^{-1} (i)= \frac{1}{a_n(-1)}\, \frac{\Gamma(i)\, \Gamma(n-i)}{\Gamma(i+1)\,\Gamma(n-i+1)} = \frac{1}{a_n(-1)}\, \frac{1}{i(n-i)}.
$$
Then
$$
a_n(-1) = \sum_{i=1}^{n-1}\frac{1}{i(n-i)} = \frac{1}{n} \sum_{i=1}^{n-1}\left(\frac{1}{i}+\frac{1}{n-i}\right) = \frac{2h_{n-1}}{n},
$$
with the usual notation $h_n$ for the harmonic series
$$
h_n:=1+\frac12 +\cdots +\frac1n.
$$
In the end, we get 
$$
q_n^{-1} (i)= \frac{n}{2h_{n-1}}\, \frac{1}{i(n-i)}
$$
\end{itemize}
The remarkable feature of the $\beta$-splitting family is that it interpolates between maximally imbalanced (caterpillar) trees and random (maximally) balanced trees passing through $\Pnpda$ and $\Pnerm$.
Table \ref{table-median} is taken from \citet[Table 3]{Ald01} and provides the median size of the smaller daughter clade at the basal split under $P_n^\beta$.

\begin{table}
\begin{center}
\large
\begin{tabular}{c|c|c|c|c|c}
$\beta$&$\ -2\ $&$\ -\frac32\ $&$\ -1\ $&$\ 0\ $&$\ \infty\ $\\
\hline
 & & & & & \\
smaller clade size&$1$&$1.5$&$\sqrt n$&$\frac n4$&$\frac n2$\\
 & & & & & 
\end{tabular}
\end{center}
\caption{Median value of the smaller clade size at the basal split as $n\to\infty$, under $P_n^\beta$ (taken from \citealt{Ald01}).}
\label{table-median}
\end{table}

\begin{rem} The reason why the name of Aldous is tied to the special case $\beta=-1$ is due to the empirical observation that the trees generated by $P_n^{-1}$ give the best fit to real phylogenies. In \citet{Ald01}, this was shown by a visual fit to a linear dependence with slope 1/2 in the log-log scale of the size of the median split vs $n$ (see previous table). In \cite{BF06}, the MLE of $\beta$ for species trees is remarkably centered around $-1$. The biological reason for this pattern is still very much debated (\citealt{HHS15, MLM15}).
\end{rem}
\noindent
For the record, we finally provide a closed-form expression for the probability of a given labelled tree shape $\tau\in\Tnl$ under $P_n^\beta$.
\begin{prop}
For any $\beta>-2$ and $\tau\in\Tnl$,
\begin{equation}
\label{eqn:explicitPnbeta}
P_n^\beta(\tau) =
\frac{ \Gamma(\beta+2)^n\,2^{n-1}}{\Gamma(\beta + n +1)}\,\prod_{v\in\mathring\vc(\tau)}\frac{\Gamma(\beta+\lambda(v)+1)}{\Gamma(\lambda(v)+1)\,a_{\lambda(v)}(\beta)}  
\end{equation}
\end{prop}

\begin{proof}
Thanks to Definition \ref{dfn:Mbm} and Eq \eqref{eqn:qnbeta}, writing $\tau=\tau'\oplus\tau''$ and $i=\lambda(\tau')$, we get by an immediate induction
\begin{eqnarray*}
P_n^\beta(\tau)&=&\frac{2q_n(i)}{{n\choose i}} \, P_i^\beta(\tau')\,P_{n-i}^\beta(\tau'')\\
 &=&  \frac{2}{a_n(\beta)\,n!} \,\Gamma(\beta+\lambda(\tau')+1)\, \Gamma(\beta+\lambda(\tau'')+1)\, P_i^\beta(\tau')\,P_{n-i}^\beta(\tau'')
\\ 
&=& \prod_{v\in\mathring\vc(\tau)}\frac{2}{a_{\lambda(v)}(\beta)\,\lambda(v)!} \,\Gamma(\beta+\lambda(v_1)+1)\,\Gamma(\beta+\lambda(v_2)+1),
\end{eqnarray*}
where for each internal node $v$, we have denoted by $v_1$ and $v_2$ its two offspring vertices. Therefore,
\begin{eqnarray*}
P_n^\beta(\tau)
&=&\Gamma(\beta + n +1)^{-1}\left[
\prod_{v\in\mathring\vc(\tau)}\frac{2\Gamma(\beta+\lambda(v)+1)}{a_{\lambda(v)}(\beta)\,\lambda(v)!} \right]
\left[
\prod_{u\text{ tip of }\tau} \Gamma(\beta+\lambda(u)+1)
\right] ,
\end{eqnarray*}
which yields \eqref{eqn:explicitPnbeta}.
\end{proof}

\begin{exo}
By giving the value 0 or $-3/2$ to $\beta$, recover from \eqref{eqn:explicitPnbeta} the explicit expressions 
$$
\Pnerm (\tau) =\frac{2^{n-1}}{n!}\prod_{v\in\mathring\vc(\tau)}\frac{1}{\lambda(v) -1}\quad \mbox{ and } \quad \Pnpda(\tau)=  \frac{1}{t_n}  
$$
For $\beta=0$ you will need to remember that $a_n(0)= n-1$, and for $\beta=-3/2$ that
$$
t_n=2^{n-1}\,\frac{\Gamma\big(n-\frac 1 2\big)}{\Gamma\big(\frac 1 2\big)} \quad \mbox{ and }\quad a_n\left(-\frac 32\right) = \frac{4\,\Gamma\big(n-\frac 1 2\big)\,\Gamma\big(\frac 1 2\big)}{\Gamma(n+1)}.
$$
\end{exo}

\subsection{Sampling consistency}

Following \cite{Ald96}, we will say that the Markov branching model $(P_n)$ is \emph{sampling consistent} if for each $n\ge 2$, the random tree shape  $\hat\tau$ obtained from the tree $\tau$ with law $P_{n+1}$, after removing its edge subtending the label $n+1$, has law $P_n$.
\begin{exo}
Prove that a Markov branching model where $q_4(2,2)=1$ cannot be sampling consistent. %For any tree shape with 5 tips, there is a choice of labelling (choices bien où mettre le 5) such that the tree spanned by 1234 is not symmetric
\end{exo}
\noindent
It is obvious that the interval splitting branching models $(P_n^f)$ are sampling consistent. %They are also exchangeable, i.e. invariant by permutations of labels. 
The converse statement is given in the following theorem.
\begin{thm}
\label{thm:HMPW}
Let $(P_n)$ be a sampling consistent Markov branching model. Then there is a measure $\mu$ on $[0,1]$ invariant by $x\mapsto 1-x$ such that $\int_{(0,1)}x(1-x)\,\mu(dx)<\infty$ and
\begin{equation}
\label{eqn:qnfinal}
q_n(i)= \alpha_n^{-1}\,\left\{{n\choose i} \,\int_{(0,1)}\mu(dx)\,x^i\,(1-x)^{n-i} + n\mu(\{0\})1_{i=1}+ n\mu(\{1\})1_{i=n-1} \right\} 
\end{equation} 
where
$$
\alpha_n = \int_{(0,1)}\mu(dx)\,\big(1-x^n-(1-x)^n\big)+ n\mu(\{0,1\}).
$$
\end{thm}
\begin{rem}
If $\mu$ has a density $f$ w.r.t. Lebesgue measure, then we are left with the interval splitting models of Aldous. The terms due to atoms at 0 and 1 correspond to single labels being taken away from the rest, a phenomenon called \emph{erosion}. The theorem states that the only Markov branching models that are sampling-consistent combine interval splitting with erosion. In the case of pure erosion (i.e., when $\mu$ charges only $\{0,1\}$), all trees are caterpillar trees a.s.  
\end{rem}
\begin{rem}
A more general version of the previous statement (i.e., not restricted to binary trees) is shown in \cite{HMPW} by identifying the splitting rules with the transitions of a general fragmentation process (\citealt{BerBook2}). This also allows the authors to study scaling limits of these random tree shapes. In particular, when $\beta\in(-2,-1)$, the trees with $n$ tips generated by the $\beta$-splitting branching model, converge as $n\to\infty$ when their edges are given properly scaled lengths, to some closed set called real tree (weakly in the Gromov--Hausdorff topology, see next chapter). 
\end{rem}

\begin{proof}[Proof of Theorem \ref{thm:HMPW}]
Assume that $(P_n)$ form a Markov branching model, by definition exchangeable, and by assumption sampling consistent. 
Let $1\le i\le n\le m$. Consider the tree generated by $P_m$ and let $\sigma_{n,m}$ be the most recent common ancestor (mrca, i.e., the ancestor with maximal distance to the root) of all its tips carrying a label in $\{1,\ldots,n\}$, let $\tau_{n,m}$ denote the descending subtree of $\sigma_{n,m}$, and write $\tau_{n,m}=\tau_{n,m}'\oplus \tau_{n,m}''$ (each subtree being equally likely to be chosen as $\tau_{n,m}'$). Now let $J_{n,m}$ (resp. $J_{n,m}'$, $J_{n,m}''$) denote the set of labels carried by the tips of $\tau_{n,m}$ (resp. $\tau_{n,m}'$, $\tau_{n,m}''$). By construction, $\{1,\ldots, n\}$ is entirely contained in $J_{n,m}$, and intersects both $J_{n,m}'$ and $J_{n,m}''$, which form a partition of $J_{n,m}$. 

Next, by sampling consistency, the triple $(J_{n,m+1}, J_{n,m+1}', J_{n,m+1}'')$ restricted to $\{1,\ldots,m\}$ has the same law as $(J_{n,m}, J_{n,m}', J_{n,m}'')$. By Kolmogorov's extension theorem, all these triples can be coupled on a same probability space, i.e., there exists a random triple $(J_n, J_n', J_n'')$ of random subsets of $\N$, such that $J_n$ contains $\{1,\ldots,n\}$, $(J_n', J_n'')$ form a partition of $J_n$, both intersect $\{1,\ldots, n\}$, and the restriction of $(J_n, J_n', J_n'')$ to $\{1,\ldots,m\}$ has the same law as $(J_{n,m}, J_{n,m}', J_{n,m}'')$.

From here on, we denote $\{1,\ldots,n\}$ by $[n]$.
We define $I_n$ as the set $(J_n\setminus [n])-n$, that is
$$
I_n:=s_n(J_n\cap \{n+1,n+2,\ldots \}),
$$ 
where $s_n(x) = x-n$. We define similarly $I_n'$ and $I_n''$. Now the triple $(I_n, I_n', I_n'')$ is exchangeable, so by de Finetti's theorem, they all have an asymptotic frequency, say $Y_n, Y_n'$ and $Y_n''$ respectively. Note that conditional on $Y_n=y$, the random variables $(\indicbis{k\in I_n})_k$ are i.i.d. Bernoulli random variables with parameter $y$. In particular, $I_n$ is empty iff $Y_n=0$ and if $Y_n\not=0$, $I_n$ is infinite with positive asymptotic frequency $Y_n$. 

Let us prove that $\PP(I_n=\varnothing)=0$ so that $Y_n>0$ a.s. Set $p_{n,m}:=\PP(J_{n,m}=[n])$. In particular, $p_{2,m}$ is the probability that in the tree $T_m$, the tips labelled 1 and 2 form a cherry. Then by exchangeability,
$$
1\ge \PP(\text{ the tip labelled 1 belongs to a cherry in }T_m)=(m-1)\,p_{2,m},
$$
so that $p_{2,m}$ vanishes as $m\to\infty$. This shows that 
$$
\PP(I_2=\varnothing) = \PP(J_{2}=[2]) = \lim_{m\to\infty} p_{2,m} = 0.
$$
Following the preceding discussion, $I_2$ is a.s. infinite. Now it can be seen that $\tau_{2,m}$ is a subtree of $\tau_{n,m}$, so that $J_{2,m}\subseteq J_{n,m}$. As a consequence, $\#J_2$ is stochastically smaller than $\# J_n$, which shows that $J_n$ (and hence $I_n$) is infinite a.s. We have proved that $Y_n>0$ a.s. 

Now let $X_n'$ (resp. $X_n''$) be the asymptotic fraction of $I_n'$ (resp. $I_n''$) in $I_n$, that is
$$
X_n':=\lim_{k\to\infty}\frac{\#I'_n\cap[k]}{\#I_n\cap[k]}=\frac{Y_n'}{Y_n}\quad\mbox{ and }\quad X_n'':=\lim_{k\to\infty}\frac{\#I''_n\cap[k]}{\#I_n\cap[k]}=\frac{Y_n''}{Y_n}
$$
Note that $X_n'+X_n''=1$. Since $X_n'$ and $X_n''$ have the same law, we can record that $X'_n$ and $1-X'_n$ have the same law. 
The Markov property of the branching model and the exchangeability imply that $X_n'$ is independent of $I_n$ and in particular of $Y_n$. In addition, conditional on $J_n$ and $X_n'$, the indicator variables $(\indicbis{a_i\in I_n'})$, where $a_i$ is the $i$-th element of $I_n$, are independent copies of a Bernoulli r.v. with success parameter $X_n'$.

From now on, we write $K_n':=J_n'\cap [n]$ and $K_n''=J_n''\cap [n]$, so that $K_n'\cup K_n'' = [n]$.
Let $A$ be a fixed subset of $[n]$ such that $1\in A$ but $2\not\in A$. Now observe that on the event $K_n'=A$, the triple $(J_2,J_2', J_2'')$ is either equal to $(J_n, J_n', J_n'')$ or to $(J_n,J_n'', J_n')$  with probabilities equal to $1/2$. So for any $x\in(0,1)$, writing $A^c$ for $[n]\setminus A$,
$$
\PP(X_n'\in dx, K_n'=A) = \frac 12 \, \PP(X_2'\in dx, J_2'\cap [n]=A, J_2''\cap [n]=A^c).  
$$
Now recall that $J_2$ is independent of $X_2'$ and conditional on $J_2$ and $X_2'=x$, the indicator variables $(\indicbis{a_i\in I_2'})$, where $a_i$ is the $i$-th element of $J_2$, are independent copies of a Bernoulli r.v. with success parameter $x$. So we get
\begin{multline*}
\PP(X_2'\in dx, J_2'\cap [n]=A, J_2''\cap [n]=A^c)\\ = \PP([n]\subset J_2)\, \PP(X_2'\in dx)\, \PP(J_2'\cap [n]=A, J_2''\cap [n]=A^c|[n]\subset J_2, X_2'=x)\\ = \PP([n]\subset J_2)\,\PP(X_2'\in dx)\,x^{i-1}\,(1-x)^{n-i-1} ,
\end{multline*}
where we let $i$ denote the cardinality of $A$. We can rewrite the next-before-last equality as
$$
\PP(X_n'\in dx, K_n'=A) = \frac 12 \, \PP([n]\subset J_2)\,\PP(X_2'\in dx)\,x^{i-1}\,(1-x)^{n-i-1} . 
$$
Summing all these equalities over all possible $A$'s with cardinality $1\le i\le n-1$, gives 
$$
\PP(X_n'\in dx, \#K_n'=i, 1\in K_n', 2\not\in K_n') = \frac 12 \,  {{n-2}\choose {i-1}}\, \PP([n]\subset J_2)\,\PP(X_2'\in dx)\,x^{i-1}\,(1-x)^{n-i-1} .  
$$
By exchangeability again, the left-hand-side equals
$$
\PP(X_n'\in dx, \#K_n'=i, 1\in K_n', 2\not\in K_n') = \PP(X_n'\in dx, \#K_n'=i) \,\frac{i(n-i)}{n(n-1)},
$$
so that
$$
\PP(X_n'\in dx, \#K_n'=i) = {{n}\choose {i}}\, \PP([n]\subset J_2)\,\frac{\PP(X_2'\in dx)}{2x(1-x)}\,x^{i}\,(1-x)^{n-i} .  
$$
Now let us treat the case when $X_n'=0$, that is $I_n'$ is empty and $J_n'$ is reduced to a singleton. By exactly the same reasoning as above,
$$
\PP(X_n'=0, J_n'=\{1\}) = \frac12\,\PP([n]\subset J_2)\,\PP(X_2'=0), 
$$
so again by exchangeability
$$
\PP(X_n'=0, \#K_n'=1) = \frac n2\,\PP([n]\subset J_2)\,\PP(X_2'=0).
$$
Reasoning symmetrically with $X_n''$, we finally get for all $x\in[0,1]$ and $i\in[n-1]$, 
$$
\PP(X_n'\in dx, \#K_n'=i) = \PP([n]\subset J_2) \,{{n}\choose {i}}\,\left(x^{i}\,(1-x)^{n-i}\,1_{x\in(0,1)}+ 1_{(x,i)=(0,1)\text{ or }(1,n-1)}\right)\,\mu(dx),  
$$
where $\mu$ is the positive measure on $[0,1]$ defined by
$$
\mu(dx) := \frac{1}{2x(1-x)}\,\PP(X_2'\in dx)1_{x\in(0,1)} + \frac 12\PP(X_2'=0)\,\delta_0(dx)+\frac 12\PP(X_2'=1)\,\delta_1(dx) .
$$
 Summing on $i\in[n-1]$ yields
$$
\PP(X_n'\in dx) = \PP([n]\subset J_2) \,\left(\big(1-x^n-(1-x)^n\big)\,1_{x\in(0,1)}+  n\,1_{x\in\{0,1\}}\right)  \,\mu(dx). 
$$
Integrating w.r.t. $x$, we have
$$
\PP([n]\subset J_2) = \left(\int_{x\in(0,1)}\big(1-x^n-(1-x)^n\big)\,\mu(dx)+  n\mu(\{0,1\})\right)^{-1} =:\frac 1{\alpha_n} ,
$$
so that 
$$
\PP(X_n'\in dx, \#K_n'=i) = \alpha_n^{-1}\,\left({{n}\choose {i}}\,x^{i}\,(1-x)^{n-i}\,1_{x\in(0,1)}+ 1_{(x,i)=(0,1)\text{ or }(1,n-1)}\right)\,\mu(dx).  
$$
Integrating w.r.t. $x$ gives the result, because $q_n(i)= \PP(\#K_n'=i)$.
\end{proof}

\chapter{Real Trees}

Textbooks and surveys available on the topic of this chapter include: \citet{DLG02, EvaBook, LeG05}.

\section{Preliminaries}

\subsection{Scaling limits}
As seen in the previous chapter, it is tempting to investigate the limiting behaviour of some marginals of our random tree shapes with $n$ tips, as $n\to\infty$. Typically interesting marginals include the maximal leaf height (the generation at which the population becomes extinct), the maximal width (the maximal population size), the coalescence time (number of generations back to the most common recent 	ancestor of a given subpopulation). There are also higher dimensional, natural marginals like the leaf-height process (the sequence of heights of tip $i$, $i=1,\ldots,n$, for some plane embedding of the tree), or the width process (the process counting the number of individuals at each successive generation).\\
\\
If for the same proper rescaling several of these marginals converge in distribution, it is relevant to ask whether the trees themselves converge in some sense to some continuous object. Such a limit theorem would have several important implications.

First, if our marginals of interest can be obtained from the tree by a continuous mapping, then by the continuous mapping theorem, they should converge to the image of the limiting object by the same mapping (in practice however, it can be easier to prove the convergence directly than to prove the mapping's continuity...). 

Second, some difficult computations in the finite case can be smoothened out in the limiting case, just as solving differential equations can be simpler than solving difference equations. The limit theorem already provides the scaling, now computations can provide the constant in front of the scaling. For example the maximal distance achieved by a random walker in $n$ time units scales like $\sqrt{n}$ thanks to Donsker's theorem, and when rescaled by $\sqrt{n}$, it converges in distribution to the maximum of the reflected Brownian motion on $[0,1]$.

Third, exactly as in the case of Donsker's theorem, we could hope that the limit theorem is an invariance principle, in the sense that the law of the limiting object is the same for a wide class of converging random sequences. In practice, this has the very important consequence that the patterns predicted by the model do not depend on the details of the model.\\
\\
It is beyond the scope of these notes to give more details %display the entire zoology of
about limit theorems for random trees, see  \cite{LGM12} or \cite{H16}.
In this chapter, we want to directly pounce to the continuous objects, only mentioning in passing how they arise as limits of discrete objects. We end this section recalling some well-known definitions in this area. We will then  introduce the general framework of real trees and explain how they can be usefully coded by a real function called the (jumping) contour process.

\subsection{Local time} % and the Ray--Knight theorem}
\label{subsec:local time}
If $A$ is a closed subset of $[0,\infty)$, a \emph{local time} associated to $A$ is a nondecreasing mapping $L:[0,\infty)\to[0,\infty)$ such that $L(0)=0$ and whose points of increase coincide with $A$. If $A$ is discrete, $L$ can be defined simply, for example as the counting process $L_t=\#[0,t]\cap A$. If $A$ is not discrete, the counting process will blow up at the first accumulation point of $A$, so a different strategy is needed.\\ 
\\
Assume that $A$ is perfect, i.e., it has empty interior and no isolated point. For any compact interval, say $[0,M]$, we can construct a continuous mapping $L:[0,M]\to [0,1]$ such that $L(0)=0$, $L(M)=1$ and for any $0\le s<t\le M$
$$
L_t>L_s \Leftrightarrow (s,t)\cap A\not=\varnothing.
$$
The construction is recursive, exactly as for the Cantor--Lebesgue function, also called `devil's staircase'. The reader who already knows this construction may skip the next paragraph. 

First recall that the open set $B:=(0,M)\setminus A$ can be written as a countable union of open intervals, say $(I_n)_{n\ge 1}$. Note that for any $\varepsilon>0$, there can be only finitely many of these intervals which have length larger than $\varepsilon$. Therefore, we can assume that the intervals $(I_n)$ are ranked by decreasing order of their lengths (in case of equality, in their order of appearance, say). Finally for each $n\ge 1$, write $I_n=(g_n,d_n)$. 

We are going to construct recursively, for each $n\ge 1$, a continuous mapping $L^n:[0,M]\to [0,1]$ which is piecewise affine and constant exactly on $\cup_{k=1}^n I_k$. First, $L^1$ is the function equal to $1/2$ on $I_1$, affine on $[0,g_1]$ and on $[d_1,1]$, such that $L^1(0)=0$ and $L^1(M)=1$. Now assume that we are given a continuous function $L^n:[0,M]\to [0,1]$ which is piecewise affine and constant exactly on $\cup_{k=1}^n I_k$. Writing $d_0=0$ and $g_0=1$, there is a unique pair $0\le k,j\le n$ such that $d_j<g_{n+1}<d_{n+1}<g_k$ minimizing $g_k-d_j$. Then we can define $L^{n+1}$ as the continuous function equal to $L^n$ outside $(d_k,g_j)$, constant to $\frac12(L^n(g_j) + L^n(d_k))$ on $[g_{n+1}, d_{n+1}]$, affine on $[d_j, g_{n+1}]$ and on $[d_{n+1}, g_k]$. It is easy to see that $L^{n+1}$ satisfies the announced properties. 

Now for any $p\in\N$, let $\Dc_p$ denote the set of dyadic numbers of $(0,1)$ whose dyadic expansion has length smaller than $p$, i.e., $\Dc_p=\{x\in (0,1): \exists (x_1,\ldots,x_p) \in\{0,1\}:x=\sum_{k=1}^px_k\,2^{-k}\}$. Also for $n\in \N$, let $\Jc_n$ denote the set of values taken by $L^n$ on its constancy intervals. Because $A$ has no isolated point, for each $p\ge 1$ there is an integer $N$ such that for all $n\ge N$, $\Dc_p\subset\Jc_n$, so that for any $n'\ge N$, $\| L^n-L^{n'}\|\le 2^{-p}$, where $\|\cdot\|$ is the supremum norm on $[0,1]$. This shows that $(L^n)$ is a Cauchy sequence for the supremum norm, and so converges uniformly on $[0,1]$ to a continuous function $L$. It is not difficult to see, using the stationarity of $(L^n)$ on $B$, that $L$ increases exactly on $A$.
\\
%For each $n\ge 1$ set $S_n:=\sum_{k=1}^n |I_k|$ and $a_n:=1/(M-S_n)$. Then it is easily seen that there is a unique mapping $L^n:[0,M]\to [0,1]$ which is constant exactly on $\cup_{k=1}^n I_k$ and is affine elsewhere with the same slope $a_n$. Denote by $Q_{n,k}$ the value taken by $L^n$ on the interval $I_k$. Set $\Gc_k:=\{j\ge 1:g_j<g_k\}$ the set of indices of intervals located on the left of $I_k$. Then an elementary computation yields $Q_{n,k}=a_n \left(g_k - \sum_{j=1}^n |I_j|\,1_{j\in\Gc_k}\right)$. Now since \textbf{$A$ has zero Lebesgue measure} $g_k = \sum_{j\ge1} |I_j|\,1_{j\in\Gc_k}$ and $\lim_nS_n= M$, so that
%$$
%Q_{n,k}=\frac{\sum_{j>n} |I_j|\,1_{j\in\Gc_k}}{\sum_{j>n} |I_j|} .
%$$
\\
In probability theory, local times are most often used to `count' the visits to a point or set by a stochastic process, e.g. the visit times of zero by Brownian motion. In this case, there are alternative ways of constructing the local time which ensure that it is at the same time adapted and unique up to a multiplicative constant. This contrasts with the recursive construction given above, where the local time is measurable, but is certainly neither adapted nor unique.

In the case of the standard Brownian motion $B$ (but also of many other Markov processes or semi-martingales) there exist several possible such constructions.
For example, if $N_\varepsilon(t)$ denotes the number of positive excursions of $B$ with height larger than $\varepsilon$, then a.s. for all $t$, $2\varepsilon \,N_\varepsilon (t)$ converges as $\varepsilon\downarrow 0$ to $L_t^0$, which is a local time at 0 for $B$. Also,
$$
L^0_t= \lim_{\varepsilon\downarrow 0} \frac{1}{2\varepsilon}\int_0^t\indic{|B_s|< \varepsilon}\,ds,
$$
where the limit again holds a.s. for all $t$. Even more interestingly, the local times of $B$ at all levels (not only 0) can be simultaneously constructed in a consistent manner. Namely, let $\nu_t$ denote the so-called occupation measure of $B$, that is for any non-negative Borel function $f$
$$
\int_\R f\,d\nu_t = \int_0^t f(B_s) \,ds.
$$ 
Then a.s. for all $t$, $\nu_t$ has a density $(L_t^a; a\in\R)$ w.r.t. Lebesgue measure. In addition, the doubly indexed process $(L_t^a; a\in\R)$ is bicontinuous and for each $a$, $(L_t^a;t\ge 0)$ is a local time for $B$ at level $a$.\\
\\
Not the least utility of an adapted local time $L$ (at 0 say) for a stochastic process $X$, is that it provides a way of indexing the excursions of $X$ away from 0 in their order of appearance. Indeed, if $J$ denotes the inverse $J$ of $L$, then to each jump $\Delta J_s$ of $J$ corresponds an excursion of $X$ away from 0 with length $\Delta J_s$, say $e_s$. In particular, if $X$ is a strong Markov process, $((s,\Delta J_s);s\ge 0)$ are the atoms of a Poisson point process in $[0,\infty)^2$. Furthermore, $((s, e_s);s\ge 0)$ are the atoms of a Poisson point process in $[0,\infty)\times \Ec$, where $\Ec$ is the space of paths with finite lifetime $V$ visiting $0$ at most at 0 and at $V$. The intensity measure of this Poisson point process is called \emph{Itô's excursion measure} and is the analogue to the common probability distribution of excursions when the visit times of 0 by $X$ form a discrete set.

\section{Definitions and examples}

\subsection{The real tree}

\begin{dfn} A real tree, or $\R$- tree,  is a complete metric space $(\tr,d)$ satisfying
\begin{itemize}
\item[(A)] Uniqueness of geodesics. For any $x,y\in \tr$, there is a unique isometric map $\phi_{x,y}:[0,d(x,y)]\to\tr$ such that $\phi_{x,y}(0)= 0$ and $\phi_{x,y}(d(x,y))=y$. 

The geodesic $\phi_{x,y}([0, d(x,y)])$, also called \emph{arc}, is denoted $\llbracket x,y\rrbracket$. 
\item[(B)] No loop. For any continuous, injective map $\psi:[0,1]\to\tr$, $\psi([0,1])= \llbracket \psi(0),\psi(1)\rrbracket$.
\end{itemize}
The root of an $\R$-tree $\tr$ is a distinguished element of $\tr$ denoted $\rho$.
\end{dfn}

\begin{thm}[Four points condition] The metric space $(\tr,d)$ is a real tree if it is complete, path-connected and satisfies for any $x_1, x_2, x_3, x_4\in\tr$
$$
d(x_1,x_2)+d(x_3,x_4)\le \max\left[d(x_1,x_3)+d(x_2,x_4), d(x_1,x_4)+d(x_2,x_3)\right]
$$
\end{thm}
\noindent
For references on real trees and the paternity of the last theorem, see \cite{DMT96} and \citet[p.2]{Duq06}.

\begin{dfn}
For any $x\in\tr$, the \emph{multiplicity}, or \emph{degree} of $x$ denotes the number of connected components of $\tr\setminus\{x\}$.
\begin{itemize}
\item $m(x)=1$ : $x$ is called a \emph{leaf}
\item $m(x)=2$ : $x$ is an \emph{internal vertex}
\item $m(x)\ge 3$ : $x$ is a \emph{branching point}.
\end{itemize}
The set of leaves of $\tr$ is denoted $\text{\rm Lf}(\tr)$ and the set of branching points $\text{\rm Br}(\tr)$.
The \emph{skeleton} of $\tr$ is $\text{\rm Sk}(\tr):=\tr\setminus\text{\rm Lf}(\tr)$.
\end{dfn}
\begin{exo} Prove that for any sequence $(x_n)$ dense in the $\R$-tree $\tr$,
$$
\text{\rm Sk}(\tr) = \bigcup_n\ \llbracket \rho, x_n\llbracket.
$$
\end{exo}
\noindent\textbf{From now on, we will assume that $\tr$ denotes a \emph{binary} $\R$-tree, that is, $m(x)\le 3$ for all $x\in\tr$.} We will also assume that $m(\rho)=1$.
We will further need the following notation and terminology.
\begin{itemize}
\item Mrca. For any $x,y\in\tr$ the \emph{most recent common ancestor} (in short mrca) of $x$ and $y$, denoted $x\wedge y$, is the unique $z\in \tr$ such that $\llbracket\rho,x\rrbracket \cap \llbracket\rho,y\rrbracket= \llbracket\rho,z\rrbracket$.
\item Partial order. For any $x,y\in\tr$, $y$ is said to \emph{descend} from $x$, and then $x$ is called an \emph{ancestor} of $y$ if $x\in\llbracket\rho,y\rrbracket$, and this is denoted $x\preceq y$. 
\item Lebesgue measure. Whenever $\tr$ is locally compact, there is a unique measure $\ell$ on the Borel $\sigma$-field of $\tr$, called Lebesgue measure or \emph{length measure}, such that for any $x,y\in\tr$, $\ell(\llbracket x,y\rrbracket) = d(x,y)$ (see Section 4.3.5 in \citealt{EvaBook}).
\item Orientation. For any $x\in\text{\rm Br}(\tr)$, $\tr\setminus\{x\}$ has 3 connected components: the one containing $\rho$ and two others, which are assumed to be labelled as the \emph{left subtree} $L_x$ and the \emph{right subtree} $R_x$.  
\end{itemize}

\subsection{First constructive example: Connecting segments}
\label{subsec:segments}

Fix some infinite-dimensional, complete vector space $\mathcal X$ and let $(\gamma_n)$ be a linearly independent, countable family of $\mathcal X$. Then construct a real tree by recursively connecting segments as follows.
\begin{enumerate}
\item Start with a segment colinear to $\gamma_1$;
\item Given a set consisting of the first $n$ segments properly connected, draw a uniform point (according to Lebesgue measure) in this set and glue one of the extremities of a segment colinear to $\gamma_{n+1}$ to this point.
\item When all segments are connected, take the closure $\tr$ of the resulting set.   
\end{enumerate}
The (random) set $\tr$ is obviously path connected and satisfies the four points condition. So it is a real tree iff it is complete, which holds iff the sequence $(\|\gamma_n\|)$ vanishes (see \citealt[Lemma 4.33]{EvaBook}). In this case $\tr$ is even compact.

\subsection{Second constructive example: Chronological trees}
%We will use the notation $\Z_+ = \N\cup\{0\}$.
We start with a well-known coding of discrete, rooted (plane) trees, sometimes denoted UHN, for Ulam--Harris--Neveu.
\begin{dfn}
A (rooted) \emph{discrete tree} $\tc$ is a subset of $\Uc:=\bigcup_{n\in \Z_+}\N^n$ (finite words), with the convention $\N^0 = \{\varnothing\}$, whose elements are called \emph{vertices}, satisfying
\begin{itemize}
\item[(i)] $\varnothing\in\tc$
\item[(ii)] if $v=uj\in\tc$, then $u\in\tc$
\item[(iii)] for any $u\in\tc$, there is $K_u\in\Z_+\cup\{+\infty\}$ such that
$$
uj\in \tc\Leftrightarrow 1\le j < K_u
$$
\end{itemize}
The vertex $\varnothing$ is called the root of $\tc$.
\end{dfn}
Let us give some further terminology and notation.
\begin{itemize}
\item Edge. An \emph{edge} is any (non-ordered) pair $\{u,uj\}$ such that $u\in\tc$, $uj\in\tc$. 
\item Partial order. The vertex $v$ is said to \emph{descend} from $u$, and then $u$ is called an \emph{ancestor} of $v$ if there is a finite word $w$ such that $v=uw$, and this is denoted $u\preceq v$.
\item Generation. The number of letters in the word $u$ is called its \emph{length}, or \emph{generation}, and denoted $|u|$. 
\item Ancestor. The ancestor of $u$ at generation $k$ is denoted $u|k$.
\item Mrca. The \emph{most common recent ancestor} (mrca) of $u$  and $v$ is 
$$
\displaystyle u\wedge v:=\arg\max_{w\in\tc}\{|w|:w\preceq u, w\preceq v\}.
$$
\item Distance. The graph distance $d$ in $\tc$ can be written as 
$$
d(u,v)=|u| +|v| -2|u\wedge v|\qquad u,v\in\tc.
$$  
\item Boundary. We denote by $\partial \tc$ the \emph{boundary} of $\tc$ defined as
$$
\partial \tc:=\{u\in\N^{\N} : \forall n\in\N,\ u|n\in \tc\}.
$$
\end{itemize}

\begin{dfn}
Let $\tc$ be a discrete, rooted tree. Assume that each vertex $u\in \tc$ is endowed with a date of birth $\alpha(u)\in[0,+\infty)$, a date of death $\omega(u)\in (\alpha(u),+\infty]$ and a lifetime duration $\zeta(u):=\omega(u)-\alpha(u)$, satisfying
\begin{itemize}
\item[(i)] $\alpha(\varnothing)=0$
\item[(ii)] for any $u\in\tc$, for any $j\in\N$,
$$
uj\in\tc\Rightarrow \alpha(u)<\alpha(uj)\le\omega(u).
$$
\item[(iii)] for any $u\in\partial \tc$,
$$
\alpha(u):=\lim_{n\uparrow\infty}\uparrow\alpha(u|n)<\infty\Rightarrow \lim_{n\to\infty}\zeta(u|n)=0,
$$
%so that in particular, $\alpha(u) = \omega(u)\Leftrightarrow |u| =\infty$.
\end{itemize}
Then we can define the \emph{chronological tree} $\tr$ as the subset of $\U:=\Uc\times[0,\infty]$ defined by
$$
\tr:=\bigcup_{u\in\tc}\{u\}\times(\alpha(u),\omega(u)]\cup \bigcup_{u\in\partial\tc:\alpha(u)<\infty}\{u\}\times\{\alpha(u)\}
$$
The chronological tree is naturally rooted at $\rho:=(\varnothing,0)$.
If in addition 
$$
\alpha(ui)\not=\alpha(uj) \qquad u\in\tc, i\not=j,
$$
then the tree $\tc$ is said \emph{binary}.
\end{dfn}
We will further need the following notation and terminology. We use (here only) the notation $p_2:\U\to [0,\infty]$ for the canonical projection $p_2((u,t))=t$.
\begin{itemize}
%\item Projections. Define $p_1$ and $p_2$ as the canonical projections $p_1:\U\to \Uc$ and $p_2:\U\to [0,\infty]$.
\item Partial order. The point $y=(v,t)\in \tr$ is said to \emph{descend} from $x=(u,s)\in\tr$, and then $x$ is called an \emph{ancestor} of $y$, if either $u=v$ and $s\le t$, or $u\not=v$, $u\preceq v$ and $s\le \alpha(v)$. This is denoted $x\preceq y$.
\item Mrca. The \emph{most common recent ancestor} of $x\in\tr$  and $y\in\tr$ is 
$$
\displaystyle x\wedge y:=\arg\max_{z\in\tr}\{p_2(z):z\preceq x, z\preceq y\}.
$$
\item Distance. We still denote by $d$ the graph distance in $\tr$, defined by
$$
d(x,y) = p_2(x)+p_2(y)-2p_2(x\wedge y)\qquad x,y\in\tr.
$$
\end{itemize}
\begin{exo} Prove the following statement.
\end{exo}

\begin{thm} The metric space $(\tr, d)$ is a locally compact $\R$-tree. 
\end{thm}
\begin{exo} Check that the notions of partial order and mrca in the chronological tree $\tr$ coincide with the corresponding notions in $\R$-trees. Characterize the leaves and branching points of $\tr$. %for any $u\in\tc$ and $a,b \in\R$ such that $0\le a\le b \le \zeta(u)$, $\ell (\{u\}\times [a,b]) = b-a$.
\end{exo}
\begin{rem}
For chronological trees represented in the plane with vertical edges, there is a natural orientation stemming from the rule that `daughters sprout to the right of their mother' (see Fig \ref{fig:smalltree} and see \citealt{Lam10} for a rigorous definition). We will always assume that chronological trees are endowed with this specific orientation. 
\end{rem}
 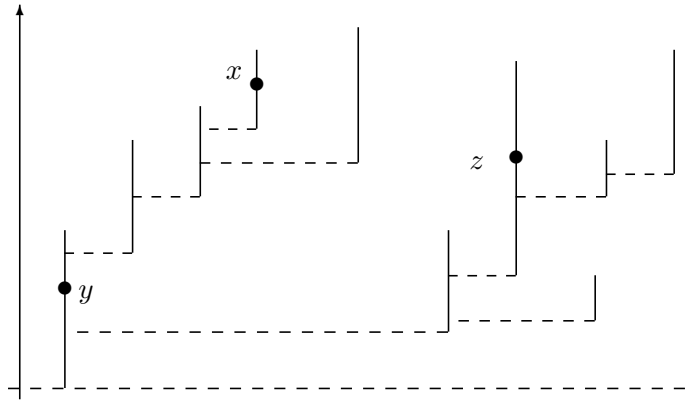
\begin{figure}[ht]
\unitlength 1.5mm % = 4.55pt
\linethickness{0.6pt}
%\ifx\plotpoint\undefined\newsavebox{\plotpoint}\fi % GNUPLOT compatibility
\begin{picture}(65,41)(-10,0)
\put(6,1){\vector(0,1){35}}
%\dashline{1}(5,2)(65,2)
\put(4.96,1.96){\line(1,0){.984}}
\put(6.92,1.96){\line(1,0){.984}}
\put(8.89,1.96){\line(1,0){.984}}
\put(10.86,1.96){\line(1,0){.984}}
\put(12.82,1.96){\line(1,0){.984}}
\put(14.79,1.96){\line(1,0){.984}}
\put(16.76,1.96){\line(1,0){.984}}
\put(18.73,1.96){\line(1,0){.984}}
\put(20.69,1.96){\line(1,0){.984}}
\put(22.66,1.96){\line(1,0){.984}}
\put(24.63,1.96){\line(1,0){.984}}
\put(26.6,1.96){\line(1,0){.984}}
\put(28.56,1.96){\line(1,0){.984}}
\put(30.53,1.96){\line(1,0){.984}}
\put(32.5,1.96){\line(1,0){.984}}
\put(34.46,1.96){\line(1,0){.984}}
\put(36.43,1.96){\line(1,0){.984}}
\put(38.4,1.96){\line(1,0){.984}}
\put(40.37,1.96){\line(1,0){.984}}
\put(42.33,1.96){\line(1,0){.984}}
\put(44.3,1.96){\line(1,0){.984}}
\put(46.27,1.96){\line(1,0){.984}}
\put(48.23,1.96){\line(1,0){.984}}
\put(50.2,1.96){\line(1,0){.984}}
\put(52.17,1.96){\line(1,0){.984}}
\put(54.14,1.96){\line(1,0){.984}}
\put(56.1,1.96){\line(1,0){.984}}
\put(58.07,1.96){\line(1,0){.984}}
\put(60.04,1.96){\line(1,0){.984}}
\put(62.01,1.96){\line(1,0){.984}}
\put(63.97,1.96){\line(1,0){.984}}
%\end
\put(10,16){\line(0,-1){14}}
\put(16,24){\line(0,-1){10}}
\put(22,19){\line(0,1){8}}
\put(27,25){\line(0,1){7}}
\put(36,22){\line(0,1){12}}
\put(44,7){\line(0,1){9}}
\put(50,12){\line(0,1){19}}
%\dashline{1}(16,14)(10,14)
\put(15.96,13.96){\line(-1,0){.857}}
\put(14.24,13.96){\line(-1,0){.857}}
\put(12.53,13.96){\line(-1,0){.857}}
\put(10.81,13.96){\line(-1,0){.857}}
%\end
%\dashline{1}(22,19)(16,19)
\put(21.96,18.96){\line(-1,0){.857}}
\put(20.24,18.96){\line(-1,0){.857}}
\put(18.53,18.96){\line(-1,0){.857}}
\put(16.81,18.96){\line(-1,0){.857}}
%\end
%\dashline{1}(27,25)(22,25)
\put(26.96,24.96){\line(-1,0){.833}}
\put(25.29,24.96){\line(-1,0){.833}}
\put(23.62,24.96){\line(-1,0){.833}}
%\end
%\dashline{1}(36,22)(22,22)
\put(35.96,21.96){\line(-1,0){.933}}
\put(34.09,21.96){\line(-1,0){.933}}
\put(32.22,21.96){\line(-1,0){.933}}
\put(30.36,21.96){\line(-1,0){.933}}
\put(28.49,21.96){\line(-1,0){.933}}
\put(26.62,21.96){\line(-1,0){.933}}
\put(24.76,21.96){\line(-1,0){.933}}
\put(22.89,21.96){\line(-1,0){.933}}
%\end
%\dashline{1}(50,12)(44,12)
\put(49.96,11.96){\line(-1,0){.857}}
\put(48.24,11.96){\line(-1,0){.857}}
\put(46.53,11.96){\line(-1,0){.857}}
\put(44.81,11.96){\line(-1,0){.857}}
%\end
%\dashline{1}(44,7)(10.13,7)
\put(43.96,6.96){\line(-1,0){.996}}
\put(41.96,6.96){\line(-1,0){.996}}
\put(39.97,6.96){\line(-1,0){.996}}
\put(37.98,6.96){\line(-1,0){.996}}
\put(35.99,6.96){\line(-1,0){.996}}
\put(33.99,6.96){\line(-1,0){.996}}
\put(32,6.96){\line(-1,0){.996}}
\put(30.01,6.96){\line(-1,0){.996}}
\put(28.02,6.96){\line(-1,0){.996}}
\put(26.02,6.96){\line(-1,0){.996}}
\put(24.03,6.96){\line(-1,0){.996}}
\put(22.04,6.96){\line(-1,0){.996}}
\put(20.05,6.96){\line(-1,0){.996}}
\put(18.06,6.96){\line(-1,0){.996}}
\put(16.06,6.96){\line(-1,0){.996}}
\put(14.07,6.96){\line(-1,0){.996}}
\put(12.08,6.96){\line(-1,0){.996}}
%\end
\put(58,19){\line(0,1){5}}
\put(64,21){\line(0,1){11}}
\put(57,8){\line(0,1){4}}
%\dashline{1}(57,8)(44,8)
\put(56.96,7.96){\line(-1,0){.929}}
\put(55.1,7.96){\line(-1,0){.929}}
\put(53.24,7.96){\line(-1,0){.929}}
\put(51.38,7.96){\line(-1,0){.929}}
\put(49.53,7.96){\line(-1,0){.929}}
\put(47.67,7.96){\line(-1,0){.929}}
\put(45.81,7.96){\line(-1,0){.929}}
%\end
%\dashline{1}(58,19)(50,19)
\put(57.96,18.96){\line(-1,0){.889}}
\put(56.18,18.96){\line(-1,0){.889}}
\put(54.4,18.96){\line(-1,0){.889}}
\put(52.62,18.96){\line(-1,0){.889}}
\put(50.84,18.96){\line(-1,0){.889}}
%\end
%\dashline{1}(64,21)(58,21)
\put(63.96,20.96){\line(-1,0){.857}}
\put(62.24,20.96){\line(-1,0){.857}}
\put(60.53,20.96){\line(-1,0){.857}}
\put(58.81,20.96){\line(-1,0){.857}}
%\end
\put(25,30){\makebox(0,0)[cc]{$x$}}
\put(27,29){\circle*{1.12}}
\put(10,10.88){\circle*{1.12}}
\put(50,22.5){\circle*{1.12}}
\put(11.88,10.25){\makebox(0,0)[cc]{$y$}}
\put(46.5,22){\makebox(0,0)[cc]{$z$}}
\end{picture}
\caption{A binary chronological tree, where edges are vertical, time flows upward, dashed lines represent filiation and daughters sprout to the right of their mother, conferring a natural orientation to the tree. The three points $x, y, z$ satisfy  $y\preceq x$ and $x\le y \le z$.
%The \emph{heights} (generations in the discrete tree) of points $x, y, z$ are respectively 3, 0, 2.
}
\label{fig:smalltree}
\end{figure}

\section{The contour process}

\subsection{From the $\R$-tree to its contour process}

Let $\tr$ be a binary, oriented $\R$-tree.
\begin{dfn}
\label{dfn:ord}
The relation $\le$ on $\tr$ is defined as follows. For any $x,y\in\tr$, 
$$
x\preceq y \Rightarrow y\le x,
$$
otherwise $x\wedge y \in \text{\rm Br}(\tr)$ and
$$
\left\{
\begin{array}{rcl}
x\in L_{x\wedge y}&\Rightarrow &x\le y\\
x\in R_{x\wedge y}&\Rightarrow &y\le x.
\end{array}
\right.
$$
\end{dfn}
\begin{exo}
Prove that $\le$ is a total order on $\tr$, that $\rho=\max \tr$, and find $\min\tr$ on the example shown in Figure \ref{fig:smalltree}.
\end{exo}
\begin{exo}
Prove that for any $x\in\tr$,
$$
\pi(x):=\{y\in\tr:y\le x\}
$$
is a Borel set.
\end{exo}
\noindent
Now we assume that we are given a finite measure $\mu$ on the Borel $\sigma$-field of $\tr$, called \emph{mass measure}, satisfying
\begin{itemize} 
\item[\textbf{Mes 1}] for any $x\le y\in\tr$,
$$
x\not=y \Rightarrow \mu(\pi(x))<\mu(\pi(y)).
$$
\item[\textbf{Mes 2}] $\mu$ is diffuse (no atom).
\end{itemize}
\begin{rem}
Whenever $\ell(\tr)<\infty$, the length measure $\ell$ is a natural example of finite measure which satisfies both  {\rm \textbf{Mes 1}} and {\rm \textbf{Mes 2}}. 
\end{rem}
Now define $\varphi:(\tr,\le, \mu)\rightarrow ([0,\mu(\tr)],\le, \text{\rm Leb})$ by
$$
\varphi(x):=\mu(\pi(x))\qquad x\in\tr,
$$
which always makes sense, since $\pi(x)$ is a Borel set of $\tr$. It is clear that $\varphi$ is one-to-one (by \textbf{Mes 1}), preserves the order and the measure. But it is not clear whether it is onto.
\begin{exo}
Display an example of a real tree that has no minimal element, and so for which $0\not\in \varphi(\tr)$.
\end{exo}
\begin{lem}
The set $D:=\varphi(\tr)$ is dense in $[0,\mu(\tr)]$.
\end{lem}
\begin{proof}
%We prove that $0\in \overline{D}$, the proof being similar for other real numbers. Assume that there is a $\le$-minimal element in $\tr$, say $\sigma$. Then $\varphi(\sigma)=\mu(\pi(\sigma))=\mu(\{\sigma\})=0$ by \textbf{Mes 2}.  If there is no minimal element in $\tr$, s
%Set $i:=\inf\{\varphi(x);x\in\tr\}$. Then there exists a sequence $(x_n)$ of $\tr$, decreasing for $\le$, such that $\lim_n\downarrow\varphi(x_n)=i$. Now since $(x_n)$ is decreasing, $\pi(x_n)$ is also decreasing, denote by $L$ its limit. If there were two distinct elements in $L$, say $y_1<y_2$, we would have $0\le\varphi(x_1)<\varphi(x_2)\le i$ by  \textbf{Mes 1}, which would contradict the definition of $i$. So there is at most one element in $L$ and $i=\lim_n\downarrow\varphi(x_n) = \lim_n\downarrow \mu(\pi(x_n))= \mu(L)=0$. 

Let $t\in(0,\mu(\tr))$. Set $G_t:=\{x\in\tr: \varphi(x)< t\}$ and $D_t:=\{y\in\tr: \varphi(y)\ge t\}$. Also set $s_t:=\sup\{\varphi(x):x\in G_t\}$ and $i_t:=\inf\{\varphi(y):y\in D_t\}$, so that in particular $s_t\le t\le i_t$.

First notice that for any $x\in G_t$, $\pi(x)\subset G_t$, so that $G_t$ is necessarily of the form $\pi(x)$ or $\pi(x)\setminus \{x\}$, which yields $\mu(G_t)=s_t$.

Now by definition of $i_t$, there is some $\le$-decreasing sequence $(y_n)$ of elements of $D_t$ such that $\lim_n\downarrow\varphi(y_n) = i_t$. Since $(y_n)$ is decreasing, the sequence $\pi(y_n)$ is also decreasing, let $L$ denote its limit. 
If there were two elements in $L\setminus G_t$, say $z_1<z_2$, we would have 
$t\le \varphi(z_1)<\varphi(z_2)\le i_t$ by \textbf{Mes 1}. Now this contradicts the definition of $i_t$, since $L\setminus G_t\subset D_t$, so $L\setminus G_t$ contains at most one element. By \textbf{Mes 2}, this shows that $\mu(L)= \mu(G_t)$. Now recall that $\mu(G_t)=s_t$, so that
$$
i_t =\lim_n\downarrow\varphi(y_n) = \lim_n\downarrow \mu(\pi(y_n))= \mu(L)=\mu(G_t)= s_t,
$$
which shows that $i_t=s_t=t$. 
\end{proof}
\noindent
In light of the previous lemma, we can define $\phi: [0,\mu(\tr)] \rightarrow \tr$ as
$$
\phi (t) :=\lim_{ s\downarrow t,\ s\in D} \varphi^{-1}(s),
$$
which we call the \emph{exploration process}. The existence and the uniqueness of this limit come from the fact that all monotonic sequences of $\tr$ do converge (see \citealt{LUB15}).
Of course, $\phi$ does not preserve the order any longer.

\begin{figure}[ht]
\centering
\unitlength 1.5mm % = 8.536pt
\linethickness{0.2pt}
\input{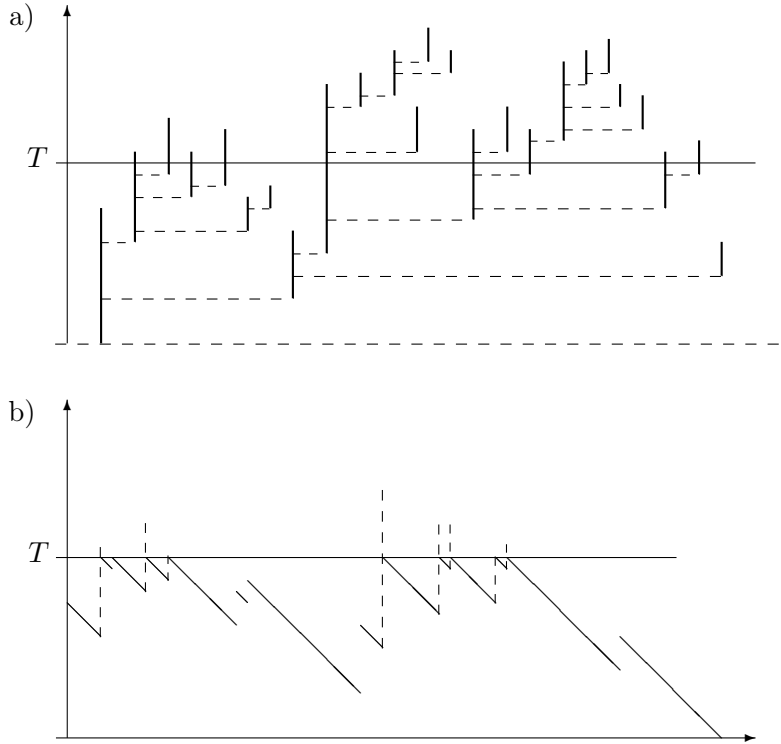}
\caption{ A chronological tree (a) and (b) the jumping contour process of its truncation below $T$, where $\mu$ is chosen equal to $\ell$.
}
\label{fig:jccp}
\end{figure}

\begin{thm}[\citealt{LUB15}]
The exploration process is the only càdlàg extension to $\varphi^{-1}$. 
The mapping $h:[0,\mu(\tr)]\to[0,\infty)$ defined by $h(s):=d(\rho,\phi(s))$ is called the \emph{jumping contour process} of $\tr$. It is càdlàg and has no negative jumps.
\end{thm}
\noindent
Fig \ref{fig:jccp} shows an example of a real tree and of the jumping contour process of its truncation below $T$, when $\mu$ is chosen equal to $\ell$. Fig \ref{fig:jcp} shows how to recover a chronological tree from its contour.

\begin{figure}[!ht]
\includegraphics[width=\textwidth]{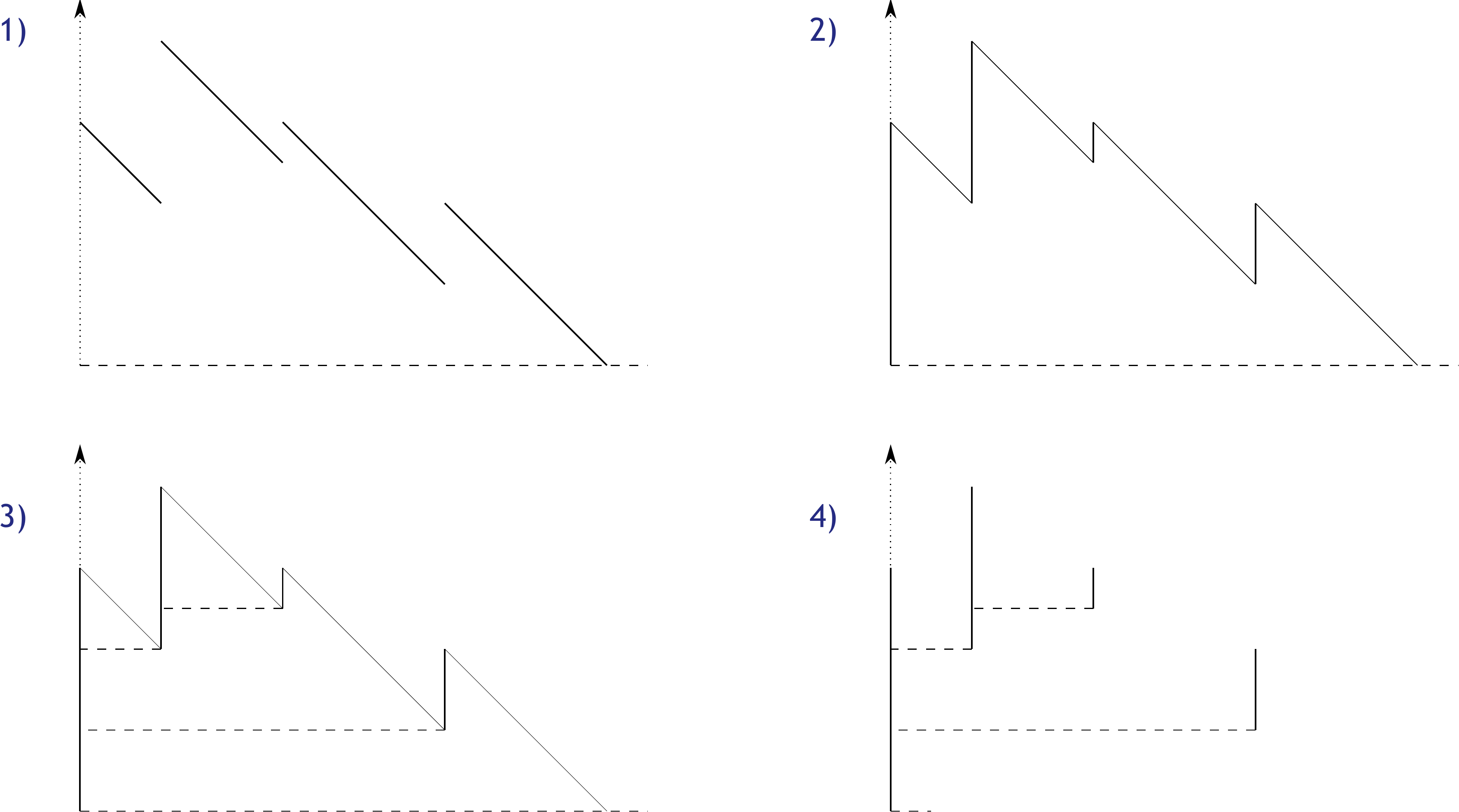}
\caption{The jumping contour process of a chronological tree with finite length, where $\mu$ is chosen equal to $\ell$: how to recover the tree from the contour. 1) Start with a càdlàg map with compact support; 2) Draw vertical solid lines in the place of jumps; 3) Report horizontal dashed lines from each edge bottom left to the rightmost solid point; 4) erase diagonal lines.}
\label{fig:jcp}
\end{figure}

\subsection{From the contour to the tree}

Let $h:[0,\infty)\to[0,\infty)$ be càdlàg with no negative jumps and compact support. Set $\sigma_h:=\sup\{t>0: h(t)\not=0\}$, as well as
$$
m_h(s,t):=\inf_{[s\wedge t, s\vee t]}h\qquad s,t\ge 0,
$$
and
$$
d_h(s,t):=h(s)+h(t) -2m_h(s,t).
$$
It is clear that $d_h$ is a pseudo-distance on $[0,\infty)$. Further let $\sim_h$ denote the equivalence relation on $[0,\infty)$
$$
s\sim_h t \Leftrightarrow d_h(s,t)=0 \Leftrightarrow h(s)=h(t)=m_h(s,t).
$$
\begin{thm}
Denote by $\tr_h$ the quotient space $[0,\sigma_h]|_{\sim_h}$. Then $(\tr_h,d_h)$ is a compact $\R$-tree. 
\end{thm}
\begin{exo}
Prove the last statement using the four points condition.
\end{exo}
\noindent
From now on, let $p_h:[0,\sigma_h]\to \tr_h$ map any element of $[0,\sigma_h]$ to its equivalence class relative to $\sim_h$. We can also endow $\tr_h$ with a total order and a mass measure, as follows. 
\begin{itemize}
\item Total order. We define $\le_h$ as the order of first visits, that is for any $x,y\in\tr_h$,
$$
x\le_h y \Leftrightarrow \inf p_h^{-1}(\{x\})\le \inf p_h^{-1}(\{y\}).
$$
\item Mass measure. The measure $\mu_h$ is defined as the push forward of Lebesgue measure by $p_h$.
\end{itemize}
\begin{thm}[\citealt{Duq06, LUB15}]
Let $(\tr, d)$ be a compact, binary $\R$-tree endowed with an orientation inducing a total order $\le$ (as in Definiftion \ref{dfn:ord}) and with a (finite) mass measure $\mu$ (satisfying \textbf{Mes 1} and \textbf{Mes 2}). Let $h$ denote the jumping contour process associated with $\le$ and $\mu$. Then $h$ is the unique càdlàg map such that the tree $(\tr_h, d_h, \le_h, \mu_h)$ is isomorphic to $(\tr, d, \le, \mu)$.
\end{thm}

\subsection{A few words on topology}

Real trees are metric spaces. The `space' of real trees only makes sense if one can imbed all trees into the same metric space, say $(\xc,\delta)$, and if two compact real trees, seen as closed subsets of $\xc$, are identified when there is a root-preserving isometry mapping one tree onto the other. 

Rigorously, the space of real trees is then the set of isometry classes of trees which are closed subsets of $\xc$, and it can then be endowed with the usual Hausdorff metric $\delta_H$ associated with $\delta$ (i.e., $\delta_H(F_1,F_2)$ is the smallest $\varepsilon$ such that the $\varepsilon$-enlargement of $F_i$ is contained in $F_j$, $i\not=j$), called on this occasion the  \emph{Gromov--Hausdorff distance} and denoted $d_{GH}$, see e.g., \citet{BBIBook, Pau89}. 

In other words, this distance is defined for any two $\R$-trees $\tr_1$ and $\tr_2$ as
$$
d_{GH}(\tr_1,\tr_2)=\inf_{f_i:\tr_i\to\xc} \delta_H(f_1(\tr_1), f_2(\tr_2))\vee \delta(f_1(\rho_1), f_2(\rho_2)),
$$
where the infimum is taken over all isometries imbedding $\tr_1$ and $\tr_2$ into $\xc$.
\begin{thm}[\citealt{EPW05}]
The Gromov--Hausdorff distance makes the space of compact real trees a complete, separable space.
\end{thm}
\noindent
Actually, one can avoid resorting to the abstract space $\xc$ and directly deal with correspondences between $\tr_1$ and $\tr_2$. 
\begin{dfn}
A correspondence between $(\tr_1,d_1)$ and $(\tr_2,d_2)$ is a subset $\Rc$ of $\tr_1\times \tr_2$ such that
$$
\forall x_1\in \tr_1\ \ \exists x_2\in \tr_2\ \ (x_1, x_2)\in\Rc,
$$ 
$$
\forall y_2\in \tr_2\ \ \exists y_1\in \tr_1\ \ (y_1, y_2)\in\Rc.
$$ 
The distortion $\text{\rm dis}(\Rc)$ of the correspondence $\Rc$ is defined as
$$
\text{\rm dis} (\Rc) = \sup\{|d_1(x_1,y_1) -d_2(x_2, y_2)|: (x_1,x_2)\in\Rc, (y_1,y_2)\in \Rc\}.
$$
\end{dfn}
\noindent
Then we have the following useful equality (\citealt{BBIBook})
$$
d_{GH}(\tr_1,\tr_2) = \frac{1}{2} \inf_{\Rc}\text{\rm dis}(\Rc),
$$
where the infimum is taken over all distortions $\Rc$ between $\tr_1$ and $\tr_2$. This equality has the following consequence (which is a slight improvement of Lemma 2.4 in \citealt{LeG05}).
\begin{prop}
Let $h_1,h_2:\R_+\rightarrow\R_+$ be two càdlàg functions with compact support, and let $\tr_1:=\tr_{h_1}$ and $\tr_2:=\tr_{h_2}$ denote the real trees associated with $h_1$ and $h_2$ respectively.
Then
$$
d_{GH}(\tr_1, \tr_2)\le 2 d_{S}(h_1, h_2),
$$
where $d_S$ denotes the Skorokhod distance. 
\end{prop}
\begin{rem}
A very important consequence of this proposition is that whenever a sequence $(X_n)$ of càdlàg, non-negative stochastic processes with compact support converges weakly in Skorokhod space to $X$, the trees coded by $X_n$ converge weakly in the Gromov--Hausdorff sense to the tree coded by $X$.

Also note that the separability of the Gromov--Hausdorff tree space stems from the separability of the Skorokhod space, thanks again to the last statement.
\end{rem}
\begin{proof}
Let $\varepsilon>0$ and let $\lambda$ be a perturbation such that $\|h_1\circ\lambda-h_2\|\le d_S(h_1,h_2)+\varepsilon$, %and $\|\lambda-id\|\le d_S(h_1,h_2)+\varepsilon$,
where $\|\cdot\|$ denotes the supremum norm. Then let $\Rc$ be the correspondence defined by
$$
\Rc=\{(x_1,x_2)\in \tr_1\times \tr_2: \exists t\ge 0, p_{h_1}(\lambda(t))=x_1, p_{h_2}(t) =x_2\}.
$$
Then for any $(x_1, x_2)\in \Rc$ and $(y_1,y_2)\in\Rc$,  there are $s,t\ge 0$ such that
$$
x_1=p_{h_1}(\lambda(s)), x_2 = p_{h_2}(s)\mbox{ and }y_1=p_{h_1}(\lambda(t)), y_2 = p_{h_2}(t).
$$
Now
$$
d_1(x_1,y_1)=h_1(\lambda(s))+h_1(\lambda(t)) -2m_{h_1}(\lambda(s),\lambda(t))
$$ 
and 
$$
d_2(x_2,y_2)=h_2(s)+h_2(t) -2m_{h_2}(s,t),
$$
so that
\debeq
|d_1(x_1,y_1)-d_2(x_2,y_2)|& \le& |h_1\circ\lambda-h_2|(s) + |h_1\circ\lambda-h_2|(t) +2\,|\inf_{[s,t]}h_1\circ\lambda-\inf_{[s,t]}h_2|\\
 &\le& 4\|h_1\circ\lambda-h_2\|.
\fineq
Then by definition of the distortion, $\text{\rm dis}(\Rc)\le 4\|h_1\circ\lambda-h_2\|$, so the inequality stated before the proposition yields 
$$
d_{GH}(\tr_1,\tr_2) = \frac{1}{2} \inf_{\Rc}\text{\rm dis}(\Rc)\le 2\|h_1\circ\lambda-h_2\|\le 2d_S(h_1,h_2)+2\varepsilon,
$$
which yields the result. Note that we have not needed to control the difference between $\lambda$ and the identity. 
\end{proof}

\section{Random $\R$-trees}

\subsection{Splitting trees}

Splitting trees (\citealt{Gei96, GK97, Lam10}) are random chronological trees satisfying the branching property. More specifically, let $\Lambda$ be a positive measure on $(0,\infty]$, called the \emph{lifespan measure}, such that $\int_{(0,\infty]} (r\wedge 1)\,\Lambda(dr)<\infty$. A splitting tree is a random chronological tree, where individuals live and reproduce independently and conditional on the life span $(\alpha(u),\omega(u)]$ of a given individual $u$, pairs of birth times and lifetimes of the newborns of $u$ form a Poisson point process on $(\alpha(u),\omega(u)]\times (0,\infty]$ with intensity $\text{Leb}\otimes\Lambda$. If $\Lambda$ is finite with mass $b=\Lambda((0,\infty])$, then individuals give birth at rate $b$ to individuals with lifetime distribution $b^{-1}\,\Lambda$.

Observe that the width process of a splitting tree is not necessarily Markovian. When $\Lambda$ is finite, it is a binary, homogeneous Crump--Mode--Jagers (CMJ) process, and it is not Markovian unless the lifetime distribution is exponential (or a Dirac mass at $\{\infty\}$ in the pure-birth case). For modeling purposes, note that splitting trees with absolutely continuous lifetimes can equivalently be defined via a `death rate' that can be age-dependent.
\begin{rem}
There are two branching processes hidden in a splitting tree other than its width process. The first one is the process tracking the number of individuals alive at each given (discrete) generation, and the second one is the process tracking the total sum of lifetimes of individuals of each given generation. Both are Markovian branching processes in discrete time, the first one with integer values (a Galton--Watson process), and the second one with non-negative real values (a Jirina process). Note that the Jirina process can take finite values even when $\Lambda$ is not finite, which is not the case of the Galton--Watson process.
\end{rem}
\noindent
Recall that chronological trees are naturally endowed with the orientation associated with the rule that `daughters sprout to the right of their mother', so they are given a natural order $\le$ associated with this orientation. 

In addition, thanks to our assumption on $\Lambda$, the length measure is a.s. locally finite. So it is possible to use the length measure to define the exploration process and the jumping contour process for the tree truncated under some fixed, finite height. The law of the jumping contour process is particularly appealing in this setting (see also Fig \ref{fig:jcp}).
\begin{thm}[\citealt{Lam10}]
\label{thm:jccp}
Let $X_t=Y_t-t$, where $Y$ is the subordinator with Lévy measure $\Lambda$. Conditional on the lifetime $x$ of the root individual, the jumping contour process of the splitting tree with lifespan measure $\Lambda$ truncated below height $a$ is distributed like the process $X$ started at $x$, reflected below $a$ and killed upon hitting 0.
\end{thm}

\subsection{The continuum random tree}

Recall how in Section \ref{subsec:segments} we have constructed a (random) real tree $\tr$ by connecting together segments colinear to the elements  $(\gamma_n)$ of a linearly independent family of a complete vector space $\mathcal X$. To define the \emph{Continuum Random Tree} (CRT) discovered by \cite{Ald91, Ald93}, we need the lengths of the segments to be random ($\mathcal X=\ell_1$ in the original paper).
\begin{dfn}
The CRT is the tree $\tr$ obtained by connecting segments whose lengths $(\|\gamma_n\|)$ are distributed as the successive distances between consecutive atoms of a Poisson point process on the half line with inhomogeneous intensity $t \, dt$. 
\end{dfn}
\begin{thm}[\citealt{Ald91, Ald93}]
 Let $\mathbf{e}$ stand for the normalized Brownian excursion, i.e., the positive Brownian excursion conditioned to have lifetime 1. The $\R$-tree $\tr_{\mathbf{e}}$ coded by $\mathbf{e}$ is isometric in law to the CRT.
\end{thm}
\noindent
It is standard that the contour process (in some appropriate meaning) of a Galton--Watson tree with finite offspring variance conditioned to have $n$ vertices, rescaled by a factor $\sqrt{n}$, converges weakly in the Skorokhod space to the normalized excursion. The results of the last section then imply that the CRT is the scaling limit in the Gromov--Hausdorff topology of conditioned Galton-Watson trees with finite offspring variance (see e.g., \citealt{Ald93, LeG93}). 

Also, since binary trees with $n$ tips have $2n-1$ vertices, binary Galton--Watson trees conditioned to have $n$ tips, which follow $\Pnpda$, converge to the tree coded by the Brownian excursion with length normalized to 2.\\
\\
This conditioning (on number of vertices or tips) can sometimes be awkward in some practical situations since it is evanescent in the limit. It can be more convenient to consider a forest of $n$ independent critical Galton--Watson trees, whose contour process (in the same appropriate meaning) is the concatenation of $n$ independent contour processes (\citealt{LeG05}). In the limit, we should get a contour process which is the concatenation of a fixed amount of excursions. This fixed amount is measured by the local time at 0 of this process.

Actually, the contour processes of Galton--Watson processes are not in general Markovian (\citealt{DLG02, LeG05}), so it will be more convenient to display the same kind of result starting with splitting trees, whose contour process is Markovian, thanks to Theorem \ref{thm:jccp}. For the sake of generality, we will also allow the trees to be subcritical ($\alpha>0$ in the following statement). 

Let $\zeta>0$.
Consider a forest of $[A\zeta]$ i.i.d. splitting trees, where $A>0$ is a scaling parameter, characterized by a finite birth rate $b_A$ and a lifetime distribution given by a random variable $V_A$. We make three assumptions.
\begin{itemize}
\item[(H1)] $\displaystyle b_A\,\E(V_A)= 1-\frac\alpha A +o\left(\frac 1A\right)$
\item[(H2)] $\displaystyle \lim_{A\to\infty} \frac{A\,\EE\big(V_A^2\big)}{2\,\EE(V_A)}=\beta>0$
\item[(H3)] $\displaystyle \lim_{A\to\infty} A\,b_A\,\EE\left(V_A^3\wedge 1\right)=0$
\end{itemize}
\begin{rem}
Note that if $V_A$ is exponentially distributed with parameter $d_A$, then the assumptions $(H1-H3)$ hold as soon as $d_A=A/\beta=b_A+\alpha/\beta$.
\end{rem}
\begin{thm}[\citealt{LSZ13}]
\label{thm:LSZ}
Let $X^A$ denote the jumping contour process of this forest and $Z^A_t$ denote the total number of individuals alive at time $t$, rescaled by $A$. Then the pair $(X^A, Z^A)$ converges weakly in Skorokhod space. 

First, the limit of the sequence $(X^A)$ is the process $X-\underline{X}$ killed when $\underline{X}$ hits $-\zeta$, where $\underline{X}_t=\inf_{s\in[0,t]}X_s$, and 
$$
X_t =-\alpha t +\sqrt{2\beta} B_t,
$$
where $B$ is the standard Brownian motion started at 0. Second, the limit of $(Z^A)$ is a diffusion process $Z$ started at $Z_0=\zeta$ and solution to an SDE of the form
$$
dZ_t=-\frac\alpha\beta\, Z_t\, dt +\sqrt{\frac{2\,Z_t}{\beta}}\, dW_t.
$$
%where $W$ is a(nother) standard Brownian motion.
\end{thm}
Note that $-\underline{X}$ is a local time at 0 for $X-\underline{X}$, so the limiting contour process $X-\underline{X}$ is indeed killed when its local time hits $\zeta$. 

The fact that the width processes $Z^A$ converge cannot be deduced from the convergences of the contours $X^A$, since the local time functional (mapping the tree to its width process) is not continuous. In this direction, notice that the theorem does not specify that $Z$ is the local time process of $X-\underline{X}$. 

Last, the convergence in Skorokhod space of the contours ensures that the splitting trees themselves converge in the Gromov--Hausdorff sense to what we could call Brownian tree, in a wider sense than the CRT (no normalization of the contour excursion interval, possible subcriticality).

\begin{rem}
Actually, the previous theorem only holds if the lifetimes of the $[A\zeta]$ progenitors are (i.i.d. and) distributed as the forward recurrence time $V_A^\star$ of $V_A$, otherwise the width process is not continuous at 0.
$$
\PP(V_A^\star\in dx) =\frac{\PP(V_A\ge x)}{\EE(V_A)}\, dx\qquad x>0.
$$
Note that $V_A^\star$ is distributed like $V_A$ if (and only if) $V_A$ is exponentially distributed. 
\end{rem}

\chapter{Reduced Trees}
\label{chap:reduced}

For a real tree $\tr$ and a fixed real number $T>0$, the so-called \emph{reduced tree} at height $T$ is the tree spanned by points at distance $T$ from the root
$$
\{y\in \tr: \exists x\in \tr, y\preceq x, d(\rho,x)=T\}.
$$
It is usually called \emph{reconstructed tree} in phylogenetics and \emph{coalescent tree} in population genetics. Its topology can be understood from the topology of the sphere of $\tr$ with center $\rho$ and radius $T>0$
$$
\trT:= \{x\in\tr : d(\rho,x)=T\},
$$
which will thus be the focus of the present chapter.

\section{The comb metric}

Most of this section is taken from \cite{LUB16}.

\subsection{Definition and examples}

Let $I$ be a compact interval and $f:I\to \R_+$ such that for any $\varepsilon >0$, $\{f\ge \varepsilon\}$ is finite. 

For any $s,t\in I$, define $\dtf$ by
$$
\dtf(s,t) = 2\sup_{(s\wedge t, s\vee t)} f.
$$
It is clear that $\dtf$ is a pseudo-distance on $\{f=0\}$ and that it is ultrametric, in the sense that
$$\dtf(r,t)\le \max(\dtf(r,s),\dtf(s,t)) \qquad r,s,t\in I.
$$
It is a distance on $\{f=0\}$ whenever $\{f\not=0\}$ is dense in $I$ for the usual distance. This may not be the case in general, so we need to consider $\dot{I}$ the quotient space $\{f=0\}|_\sim$ where $\sim$ is the equivalence relation 
 $$
 s\sim t \Leftrightarrow \dtf(s,t)=0 \Leftrightarrow f=0 \mbox{ on }[s\wedge t,s\vee t].
 $$

\begin{dfn}
\label{dfn:comb}
We call $f$ a \emph{comb-like function} or \emph{comb}, and $\dtf$ the \emph{comb metric} on $\dot{I}$.
\end{dfn}
\noindent
Let us give the canonical example of a comb. 
Let $\tr$ be an oriented $\R$-tree with finite length, total order $\le$ and jumping contour process $h$ associated (for example) with its length measure (i.e., the mass measure $\mu$ is taken equal to the length measure $\ell$). Let $T>0$ such that the sphere
$$
\trT= \{x\in\tr : d(\rho,x)=T\}
$$
has finite cardinality $N_T\ge 2$. Let $x_1\le \cdots \le x_{N_T}$ denote its elements labelled in the order $\le$. Then for any $1\le i <j\le N_T$, writing $s_i:=\inf p_h^{-1}(\{x_i\})$ (each set $p_h^{-1}(\{x\})$ is actually a singleton %pas évident
in the case when the sphere is finite),
$$
d(x_i,x_{j})= h(s_i) + h(s_{j})- 2 \inf_{[s_i,s_{j}]}h = 2(T- \inf_{[s_i,s_{j}]}h)= 2\max(h_i,\ldots, h_{j-1}),
$$
 where
$$
h_i:=T- \inf_{[s_i,s_{i+1}]}h.
$$
In conclusion, the metric on $\trT$ is isomorphic to the comb metric $\dtf$ on $\dot I$, where $I=[1,N_T]$ and $f:=\sum_{i=1}^{N_T-1}h_i\indicbis{i}$. We will extend this description to all locally compact $\R$-trees in Section \ref{sec:spheres}.
\\
\\
Fig \ref{fig:comb} shows a comb and how an ultrametric tree can be embedded into it. As in Fig \ref{fig:combleftright} the same metric space can be represented by identifying points of $\dot I$ to their left-neighbour in $I\setminus \dot I$ (or to their right-neighbour) and reporting the distances accordingly.\\
\begin{figure}[!ht]
\includegraphics[width=\textwidth]{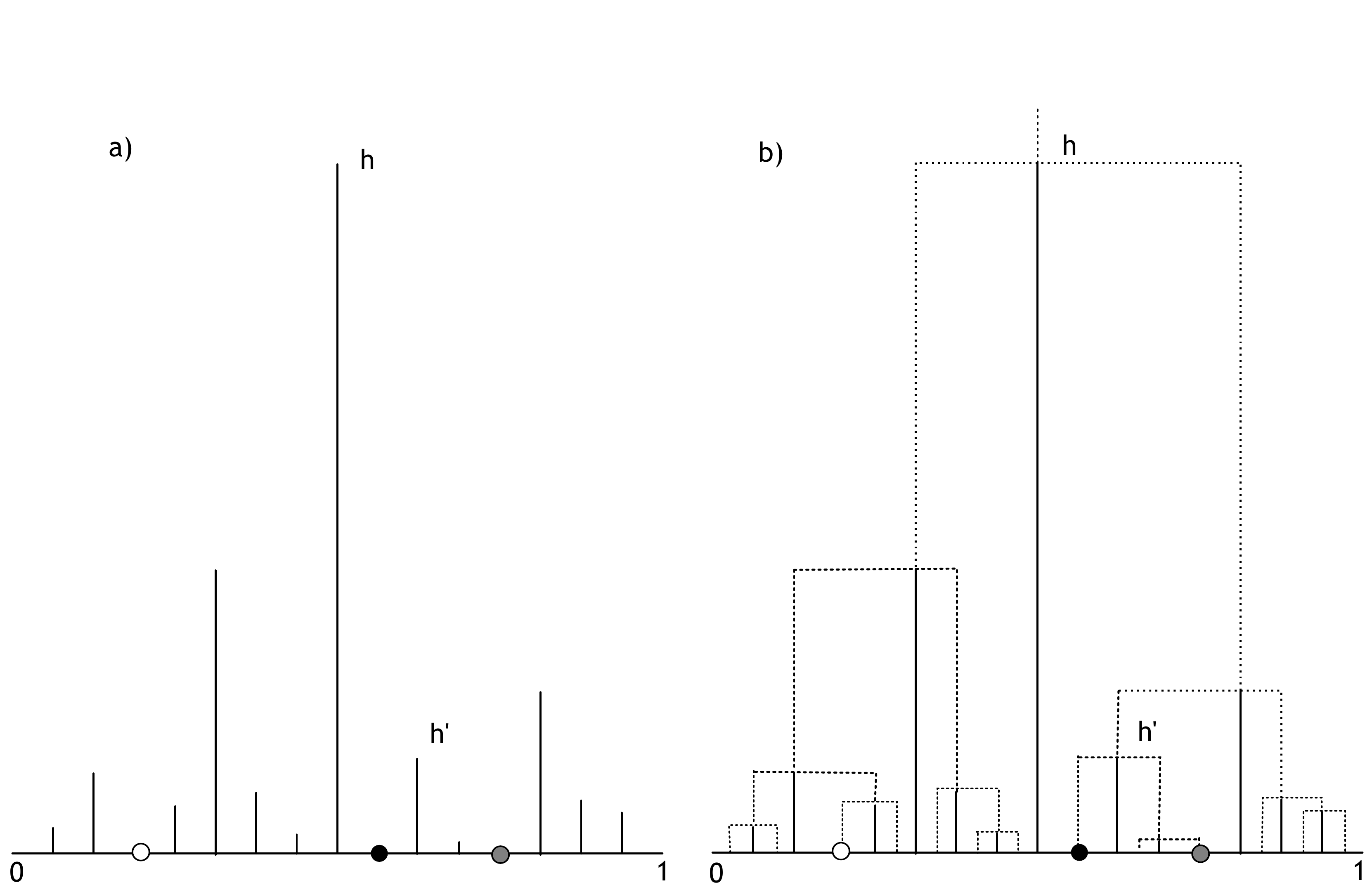}
\caption{a) A comb-like function with finite support on $[0,1]$. The distance between the black dot and the grey dot is $h'$, whereas $h$ is the distance from either of these dots to the white dot; b) In dashed lines, the ultrametric tree associated to the comb shown in a).}
\label{fig:comb}
\end{figure}
\begin{figure}[!ht]
\includegraphics[width=\textwidth]{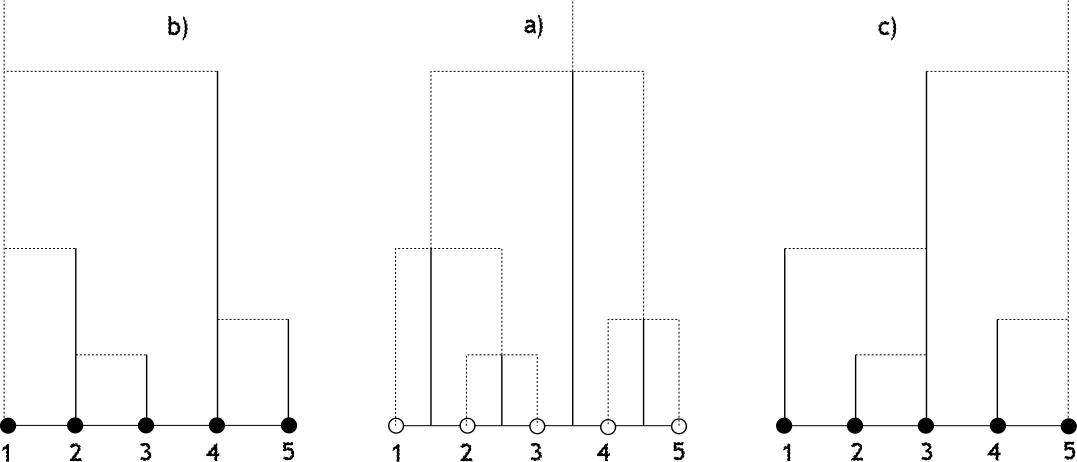}
\caption{a) A comb with finite support, and the associated ultrametric tree in dashed lines; equivalent representations of this ultrametric space can be obtained by reporting all tips of the comb to the left (b) or to the right (c).}
\label{fig:combleftright}
\end{figure}
\begin{rem}
The space $(\dot{I}, \dtf)$ is not complete in general. To make it complete, one has to distinguish for each point $t\in I$ between its \emph{left face} $(t,l)$ and its \emph{right face} $(t,r)$. The distance $\dtf$ is extended to the space $I\times\{l,r\}$ by the following definitions for $s< t\in I$
$$
\dtf((s,r),(t,l)) = 2\sup_{(s, t)} f, \qquad \dtf((s,l),(t,l)) = 2\sup_{[s, t)} f, 
$$
$$
\dtf((s,r),(t,r)) =2\sup_{(s, t]} f,  \qquad \dtf((s,l),(t,r)) = 2\sup_{[s, t]} f,
$$
and the symmetrized definitions for $s> t$. If $s=t$, the four last quantities are respectively defined as $f(t)$, 0, 0, $f(t)$. This extension of $\dtf$ is a pseudo-distance and it can be shown \citep{LUB16} that the associated quotient space $\bar{I}$ is a compact, ultrametric space called \emph{comb metric space}. Actually the converse is also true, as states the next theorem.
\end{rem}

\begin{thm}[\citealt{LUB16}]
\label{thm:second sense}
Any compact ultrametric space is isometric to a comb metric space.
\end{thm}

\subsection{Spheres of $\R$-trees}
\label{sec:spheres}
Here, we consider a locally compact $\R$-tree $\tr$ and we assume that the sphere
$$
\trT= \{x\in\tr : d(\rho,x)=T\}
$$
is not empty. Note that by the four-points condition, for any $x,y,z\in\trT$,
\begin{multline*}
T+d(x,z)=\\d(\rho,y)+d(x,z)\le \max\left[d(\rho,x)+d(y,z), d(\rho,z)+d(y,x)\right]\\ =\max \left[T+d(y,z), T+d(y,x)\right],
\end{multline*}
which yields $d(x,z)\le \max \left[d(y,z), d(y,x)\right]$, that is the metric induced by $d$ on $\trT$ is ultrametric. 

Since $\tr$ is locally compact, $\trT$ is a compact ultrametric space and Theorem \ref{thm:second sense} ensures that provided it has no isolated point, it is isometric to a comb metric space. There is actually an isometry between $\trT$ and a comb metric space preserving the order on $\trT$ inherited from the order $\le$ on $\tr$.
Let $h$ denote the jumping contour process of the tree truncated at height $T$, which is the closed ball with center $\rho$ and radius $T$.
\begin{exo}
Assume that $\trT$ has no isolated point. Prove that $\{h=T\}$ has no isolated point and empty interior. Also prove that $\trT$ has empty interior.
\end{exo}
Recall from the paragraph p.\pageref{subsec:local time} on local time, that since $\{h=T\}$ is perfect, we can construct a local time at level $T$ for $h$, that is a nondecreasing, continuous map $L^T:[0,\infty)\to [0,\infty)$ such that $L^T(0)=0$ and for any $s<t$
$$
L_t^T>L_s^T \Leftrightarrow (s,t)\cap\{h=T\}\not=\varnothing.
$$
Let $I=[0,L^T_{\sigma_h}]$, and set
$$
\dottrT:= \{x\in\trT : \exists (x_n^+)\uparrow, (x_n^-) \downarrow\in\trT, \lim_n\uparrow x_n^+=\lim_n\downarrow x_n^-=x\},
$$
where the sequences in the previous definition are requested to be strictly monotonic.
\begin{thm}[\citealt{LUB16}] 
\label{thm:comb-reducedtree}
Assume as previously that $\tr$ is a locally compact $\R$-tree such that $\trT$ is not empty and has no isolated point. Then there is a comb-like function $f$ on $I$ and two global isometries $\dot\theta:(\dot I, \dtf)\to(\dottrT,d)$ and $\bar\theta:(\bar I, \dtf)\to(\trT,d)$ preserving the order and mapping the Lebesgue measure to the push forward  $\mu^T$ of the measure $dL^T$ by $p_h$.
\end{thm}

\section{Coalescent Point Processes}

\subsection{The reduced tree of splitting trees, of the Brownian tree}

Consider a splitting tree $\tr$ with lifespan measure $\Lambda$ satisfying $\int_{(0,\infty]} (r\wedge 1)\,\Lambda(dr)<\infty$. We have seen in the last chapter that $\tr$ is an oriented $\R$-tree naturally endowed with the associated total order $\le$, and that it has locally finite length. Taking the mass measure equal to the length measure $\ell$, the jumping contour process $X$ of the tree $\tr$ truncated below $T$  is well-defined and thanks to Theorem \ref{thm:jccp}, it has the law of the process $(Y_t-t;t\ge0)$, where $Y$ is the subordinator with Lévy measure $\Lambda$, reflected below $T$ and killed upon hitting 0.

So $\tr$ falls into the category of canonical examples given after Definition \ref{dfn:comb} and conditional on $N_T\ge 1$ (where $N_T=\#\trT$), it is isometric to the comb metric space $(\dot I, \dtf)$, where $I=[1,N_T]$ and $f:=\sum_{i=1}^{N_T-1}H_i\indicbis{i}$, with
$$
H_i:=T- \inf_{[\sigma_i,\sigma_{i+1}]}X,
$$
and $\sigma_i$ is the $i$-th visit of $T$ by $X$. 

Note that $\PP_x(N_T\not=0)=\PP_x(\tau_T^+<\tau_0)$, where the subscript $x$ records the lifetime $x$ of the progenitor, which is also the starting point of the contour process, and $\tau_z^+$  (resp. $\tau_y$) denotes the first hitting time by $X$ of $[z,+\infty)$ (resp. of $\{y\}$). 

Also note, thanks to the strong Markov property of $X$, that conditional  on $N_T\ge 1$, $N_T$ is geometric with failure probability $\PP_T(\tau_T^+>\tau_0)$. Furthermore, conditional on $N_T=n$, the $H_i$'s are $n$ i.i.d. random variables, all distributed as the depth of the excursion of $X$ away from $T$ conditioned to be smaller than $T$. It is then elementary to get the following result.

\begin{prop}
The sphere $\tr^T$ of the splitting tree $\tr$ is non empty with probability $\PP_x(\tau_T^+<\tau_0)$. Conditional on being non-empty, it is isometric to the random comb $\sum_{i\ge 1}H_i\indicbis{i}$, where the $H_i$'s form a sequence of i.i.d. random variables killed at its first value larger than $T$, and whose common distribution is given by 
\begin{equation}
\label{eqn:formula_H}
\PP(H_1>t)=\PP_T(\tau_{T-t}<\tau_T^+).
\end{equation}
We say that $\trT$ is (isometric to) a \emph{coalescent point process} (CPP).
\end{prop}
\noindent
Fig \ref{fig:treecoal} shows a coalescent point process and how it codes for an ultrametric tree.
\begin{figure}[ht]
\unitlength 2mm % = 8.536pt
\linethickness{0.2pt}
\centering
\begin{picture}(69,37)(-4,3)
\put(6,5){\vector(0,1){35}}
%\dashline{1}(5,6)(65,6)
\put(4.961,5.961){\line(1,0){.9836}}
\put(6.928,5.961){\line(1,0){.9836}}
\put(8.895,5.961){\line(1,0){.9836}}
\put(10.863,5.961){\line(1,0){.9836}}
\put(12.83,5.961){\line(1,0){.9836}}
\put(14.797,5.961){\line(1,0){.9836}}
\put(16.764,5.961){\line(1,0){.9836}}
\put(18.731,5.961){\line(1,0){.9836}}
\put(20.699,5.961){\line(1,0){.9836}}
\put(22.666,5.961){\line(1,0){.9836}}
\put(24.633,5.961){\line(1,0){.9836}}
\put(26.6,5.961){\line(1,0){.9836}}
\put(28.568,5.961){\line(1,0){.9836}}
\put(30.535,5.961){\line(1,0){.9836}}
\put(32.502,5.961){\line(1,0){.9836}}
\put(34.469,5.961){\line(1,0){.9836}}
\put(36.436,5.961){\line(1,0){.9836}}
\put(38.404,5.961){\line(1,0){.9836}}
\put(40.371,5.961){\line(1,0){.9836}}
\put(42.338,5.961){\line(1,0){.9836}}
\put(44.305,5.961){\line(1,0){.9836}}
\put(46.272,5.961){\line(1,0){.9836}}
\put(48.24,5.961){\line(1,0){.9836}}
\put(50.207,5.961){\line(1,0){.9836}}
\put(52.174,5.961){\line(1,0){.9836}}
\put(54.141,5.961){\line(1,0){.9836}}
\put(56.108,5.961){\line(1,0){.9836}}
\put(58.076,5.961){\line(1,0){.9836}}
\put(60.043,5.961){\line(1,0){.9836}}
\put(62.01,5.961){\line(1,0){.9836}}
\put(63.977,5.961){\line(1,0){.9836}}
%\end
%\dottedline(5,33)(64,33)
\multiput(4.961,32.961)(.983333,0){61}{{\rule{.4pt}{.4pt}}}
%\end
\put(4,33){\makebox(0,0)[cc]{\scriptsize $T$}}
%\dashline{1}(6,23)(14,23)
\put(5.961,22.961){\line(1,0){1}}
\put(7.961,22.961){\line(1,0){1}}
\put(9.961,22.961){\line(1,0){1}}
\put(11.961,22.961){\line(1,0){1}}
%\end
%\dashline{1}(22,21)(30,21)
\put(21.961,20.961){\line(1,0){1}}
\put(23.961,20.961){\line(1,0){1}}
\put(25.961,20.961){\line(1,0){1}}
\put(27.961,20.961){\line(1,0){1}}
%\end
%\dashline{1}(30,28)(38,28.25)
\put(29.961,27.961){\line(1,0){1}}
\put(31.961,28.023){\line(1,0){1}}
\put(33.961,28.086){\line(1,0){1}}
\put(35.961,28.148){\line(1,0){1}}
%\end
%\dashline{1}(46,26)(54,26)
\put(45.961,25.961){\line(1,0){1}}
\put(47.961,25.961){\line(1,0){1}}
\put(49.961,25.961){\line(1,0){1}}
\put(51.961,25.961){\line(1,0){1}}
%\end
%\dashline{1}(46,12)(6,12)
\put(45.961,11.961){\line(-1,0){1}}
\put(43.961,11.961){\line(-1,0){1}}
\put(41.961,11.961){\line(-1,0){1}}
\put(39.961,11.961){\line(-1,0){1}}
\put(37.961,11.961){\line(-1,0){1}}
\put(35.961,11.961){\line(-1,0){1}}
\put(33.961,11.961){\line(-1,0){1}}
\put(31.961,11.961){\line(-1,0){1}}
\put(29.961,11.961){\line(-1,0){1}}
\put(27.961,11.961){\line(-1,0){1}}
\put(25.961,11.961){\line(-1,0){1}}
\put(23.961,11.961){\line(-1,0){1}}
\put(21.961,11.961){\line(-1,0){1}}
\put(19.961,11.961){\line(-1,0){1}}
\put(17.961,11.961){\line(-1,0){1}}
\put(15.961,11.961){\line(-1,0){1}}
\put(13.961,11.961){\line(-1,0){1}}
\put(11.961,11.961){\line(-1,0){1}}
\put(9.961,11.961){\line(-1,0){1}}
\put(7.961,11.961){\line(-1,0){1}}
%\end
%\dashline{1}(22,17)(6,17)
\put(21.961,16.961){\line(-1,0){1}}
\put(19.961,16.961){\line(-1,0){1}}
\put(17.961,16.961){\line(-1,0){1}}
\put(15.961,16.961){\line(-1,0){1}}
\put(13.961,16.961){\line(-1,0){1}}
\put(11.961,16.961){\line(-1,0){1}}
\put(9.961,16.961){\line(-1,0){1}}
\put(7.961,16.961){\line(-1,0){1}}
%\end
\put(14,23){\line(0,1){10}}
\put(22,17){\line(0,1){16}}
\put(30,21){\line(0,1){12}}
\put(38,28){\line(0,1){5}}
\put(46,12){\line(0,1){21}}
\put(54,26){\line(0,1){7}}
\put(62,33){\line(0,-1){29}}
\put(14,34.5){\makebox(0,0)[cc]{\scriptsize 1}}
\put(22,34.5){\makebox(0,0)[cc]{\scriptsize 2}}
\put(30,34.5){\makebox(0,0)[cc]{\scriptsize 3}}
\put(38,34.5){\makebox(0,0)[cc]{\scriptsize 4}}
\put(46,34.5){\makebox(0,0)[cc]{\scriptsize 5}}
\put(54,34.5){\makebox(0,0)[cc]{\scriptsize 6}}
\thicklines
\put(9,33){\vector(0,-1){10}}
\put(17,33){\vector(0,-1){16}}
\put(25,33){\vector(0,-1){12}}
\put(33,33){\vector(0,-1){5}}
\put(41,33){\vector(0,-1){21}}
\put(49,33){\vector(0,-1){7}}
\put(11,28.75){\makebox(0,0)[cc]{\scriptsize $H_1$}}
\put(19,25.875){\makebox(0,0)[cc]{\scriptsize $H_2$}}
\put(27,28.125){\makebox(0,0)[cc]{\scriptsize $H_3$}}
\put(35,30.625){\makebox(0,0)[cc]{\scriptsize $H_4$}}
\put(43,20.625){\makebox(0,0)[cc]{\scriptsize $H_5$}}
\put(51,29){\makebox(0,0)[cc]{\scriptsize $H_6$}}
\end{picture}
\caption{A coalescent point process (upside down compared to previous pictures of combs) with 6 nonzero values (the 7th one is the first one larger than $T$). To recover an oriented ultrametric tree with 7 tips, draw horizontal lines from each tip left to the rightmost vertical line.}
\label{fig:treecoal}
\end{figure}
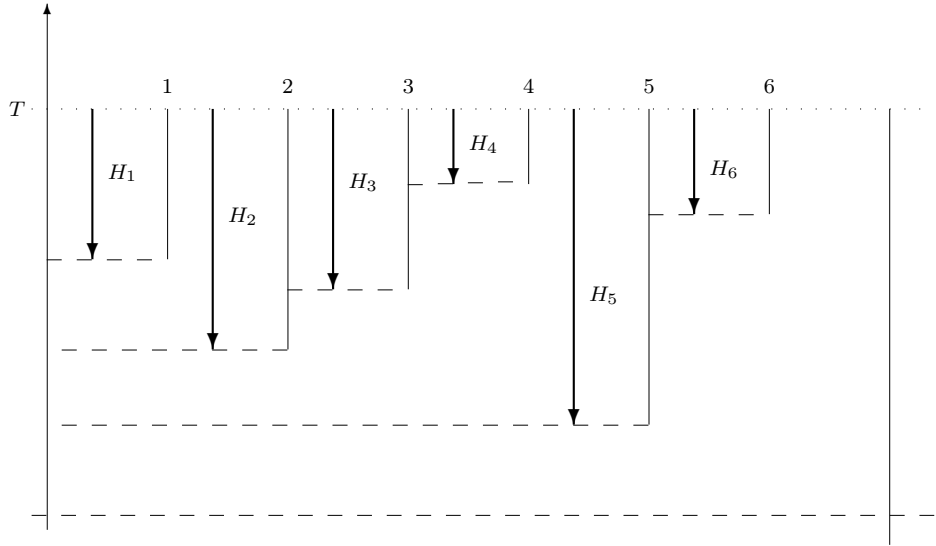
\noindent
By extension, we make the following definition.
\begin{dfn}
Let $\nu$ be a $\sigma$-finite measure on $(0,\infty)$ such that $\nu([\varepsilon,\infty))<\infty$ for all $\varepsilon>0$. Let $\mc$ be a Poisson point process on $(0,\infty)^2$ with intensity $\mbox{Leb}\,\otimes\,\nu$ and denote by $(S_i, H_i)_i$ its atoms. Finally, let $(D,H)$ denote the first (i.e., smallest in the first dimension) atom  such that $H>T$.
We will say that the random comb metric space associated with the comb $\sum_{i:S_i<D}H_i\indicbis{S_i}$ is a \emph{coalescent point process} with height $T$ and intensity measure $\nu$.
\end{dfn}
\noindent
The term `coalescent point process' has first been coined in the setting of the Brownian tree by \citet{Pop04}.
Recall that we called Brownian tree the tree coded by a positive Brownian excursion (CRT when the excursion length is normalized). 

The Brownian excursion has a local time at level $T$, which allows one to construct as in Theorem \ref{thm:comb-reducedtree} the comb giving the metric of the reduced tree at level $T$.  This comb is a `list', in the plane order, of the depths of excursions of the contour away from $T$. 
\begin{thm}[\citealt{Pop04, AP05}]
Conditional on being non-empty,  the sphere $\trT$ of the Brownian tree $\tr$ is a coalescent point process with height $T$ and intensity measure $\nu$, where $\nu$ is the push forward of the Brownian excursion measure by the depth mapping, i.e.
\begin{equation}
\label{eqn:formula_0}
\nu_0(dh) = \frac{dx}{2 x^2} 
\end{equation}
\end{thm}
\noindent
More generally, as a follow up to Theorem \ref{thm:LSZ}, it can be shown that under suitable scaling, the coalescent point processes of (sub)critical splitting trees conditioned on reaching height $T$ converge to the Poisson point process of excursion depths of a Brownian motion (with negative drift). 

More specifically, under the assumptions (H1--H3) of Theorem  \ref{thm:LSZ}, if $H_1^A, H_2^A,\ldots$ denote the coalescence times of the splitting tree conditioned on reaching height $T$, then the point processes $\sum_{i\ge 1} \delta_{(\frac{i}{A}, H_i^A)}$ converge as $A\to\infty$ to a CPP with height $T$ and intensity measure $\nu$, where 
\begin{equation}
\label{eqn:formula_alpha}
\nu_\alpha ((x,\infty)) = \frac{\alpha}{1-e^{-\alpha x/\beta}}.
\end{equation}
\noindent
Formulae \eqref{eqn:formula_H}, \eqref{eqn:formula_0} and \eqref{eqn:formula_alpha} can actually all be rephrased in terms of the scale function $W$ of the relevant contour process, as we will now see. We refer the reader to \citet{BerBook1, KypBook} for more information about what follows.

Recall that a Lévy process $X$ with no negative jumps is characterized by the Laplace transform of its one-dimensional marginals
$$
\EE(\exp (-\lambda X_t)) =\exp(t\psi(\lambda)) \qquad t,\lambda \ge 0,
$$
where $\psi$ is called the \emph{Laplace exponent} of $X$. If $X_t=Y_t-t$ (where $Y$ is the subordinator with Lévy measure $\Lambda$), the Lévy--Khinchin formula gives
$$
\psi(\lambda) = \lambda-\int_{(0,\infty]} \left(1-e^{-\lambda r}\right)\,\Lambda(dr)\qquad \lambda\ge 0.
$$
If $X_t=-\alpha t +\sqrt{2\beta} B_t$, then $\psi(\lambda) = \alpha\lambda+\beta\lambda^2$.

Notice that $\psi(0)=0$, that $\psi$ is convex and has at most one positive root, here denoted $\eta$. 
Then there is a unique non-negative, increasing function $W$ called the \emph{scale function} such that
\begin{equation}
\label{eqn:LTW}
\int_0^\infty W(x)\,e^{-\lambda x}\,dx= \frac{1}{\psi(\lambda)}  \qquad \lambda\ge \eta.
\end{equation}
\noindent
In the case when $X_t=-\alpha t +\sqrt{2\beta} B_t$, it is not difficult to compute $W(x) = x/\beta$ when $\alpha=0$ and when $\alpha\not=0$,
\begin{equation}
\label{eqn:WMB}
W(x) = \frac{1-e^{-\alpha x/\beta}}{\alpha}\qquad x\ge 0.
\end{equation}
In general, one can prove (see \citealt{BerBook1}) that
$$
W(x) = \exp\left\{ \int_0^x \underline{N}(\sup>s)\, ds\right\} \qquad x\ge 0,
$$
where $\underline{N}$ is the excursion measure of $X-\underline{X}$ away from 0. When $X$ has finite variation, like when $X_t=Y_t-t$, $\underline{N}$ is merely the birth rate $b=\Lambda((0,\infty))$ times the law of $X$ started from a random jump with law $b^{-1}\Lambda$.

 In all cases, it is a consequence of the exponential formula for the Poisson point process of excursions of $X-\underline{X}$ away from 0 that
\begin{equation}
\label{eqn:proba-scale}
\PP_x(\tau_0<\tau_a^+)=\frac{W(a-x)}{W(a)}\qquad 0\le x\le a.
\end{equation}
In the cases with infinite variation, such as Brownian motion (with or without drift), $W(0)=0$ and
\begin{equation}
\label{eqn:depthmeas}
N(-\inf >x)=\frac{1}{W(x)}\qquad x\ge 0,
\end{equation}
where $N$ is the excursion measure of $X$ away from 0, so that Eqs \eqref{eqn:formula_0} and \eqref{eqn:formula_alpha} are a consequence of \eqref{eqn:WMB} and \eqref{eqn:depthmeas}.
In the case when $X_t= Y_t -t$, $W(0)=1$, and we can be more accurate in Eq \eqref{eqn:formula_H} using Eq \eqref{eqn:proba-scale}
$$
\PP(H_1>t)=\PP_T(\tau_{T-t}<\tau_T^+) = \frac{1}{W(t)},
$$
where $W$ can be identified inverting the Laplace transform \eqref{eqn:LTW}.
\begin{rem}
In \citet{LP13}, we have expressed the distribution of the coalescent point process of non binary branching trees, including Galton--Watson processes and continuous-state branching processes. One of the main difficulties is to cope with the existence of points with arbitrarily large degree in the tree.
\end{rem}

\subsection{A more general class of models}

In this section, we seek to investigate population processes which generate trees whose spheres are (isometric to) CPPs.\\

Consider a population where all individuals live and reproduce independently, and each individual is endowed with a trait (some random character living in $\R$ for simplicity) that evolves through time according to independent copies of the same, possibly time-inhomogeneous, Markov process $K$ with  generator $L_t=L(t,\cdot)$. Further assume what follows.
\begin{itemize}
\item This trait is non-heritable, in the sense that any individual born at time $t$ draws the value of her trait at birth from the same distribution $\nu_t$, independently of her mother's history;
\item All individuals give birth at the same, possibly time-inhomogeneous rate $b(t)$;
\item An individual holding trait $x$ at time $t$ dies at rate $d(t,x)$. 
\end{itemize}
\begin{thm}[\citealt{LS13}]
\label{thm:generalmodel}
Under the previously defined population model, starting with one individual at time 0 and conditional on having at least one alive individual at time $T$, the reduced genealogical tree at level $T$ is given by a coalescent point process with typical depth $H$ whose inverse tail distribution is given by
$$
W(t):=\frac{1}{\PP(H>t)} =\exp\left( \int_{T-t}^ T b(s)\, (1-q(s))\,ds\right)\qquad t\in[0,T],
$$
where $q(t)$ denotes the probability that an individual born at time $t$ has no descendants alive by time $T$.
\end{thm}
We will see later why the function $W$ defined in the previous statement becomes the scale function of the last section when there is no time-inhomogeneity and the inheritable trait is the age.

In addition, $W$ can be computed from the knowledge of $g$, where $g(t,\cdot)$ denotes the density of the death time of an individual born at time $t$.  
\begin{prop}[\citealt{LS13}]
\label{prop:NEW}
The function $W$ is solution to the following integro-differential equation
\begin{equation}
\label{eqn:hidden convolution}
W'(t) = b(T-t)\,\left( W(t) - \int_0^t \ W(s)\,g(T-t,T-s) ds \right)\qquad t\ge 0,
\end{equation}
with initial condition $W(0)=1$.
\end{prop}
\begin{proof}
First observe that 
$$
q(t) = \int_t^T \, g(t,s) \,e^{-\int_t^s  b(u)\,(1-q(u))du}ds.
$$
Recalling that 
$$
W(t) = \frac{1}{\PP(H>t)} = \exp\left( \int_{T-t}^T b(s)\, (1-q(s))\,ds\right),
$$
we get
$$
q(t)= \int_t^T \, g(t,s) \, \frac{W(T-s)}{W(T-t)} ds,
$$
or equivalently
$$
q(T-t) = \int_0^t \, g(T-t,T-s) \, \frac{W(s)}{W(t)} ds.
$$
Now check that 
$$
W'(t) = b(T-t)\, (1-q(T-t))\, W(t).
$$
Equation \eqref{eqn:hidden convolution} is a consequence of the last two equations.
\end{proof}
\begin{rem}
In general, $g$ is not given directly in terms of the model ingredients $b$, $d$ and the generator $L_t$ of the trait process $K$. To compute $g$ and then $W$, one can proceed as follows. 
Recall that $g(t,\cdot)$ is the density of the death time of an individual born at time $t$, so that 
$$
g(t,s) = \int_\RR \nu_t (dx)\ u_s(t,x) \qquad s\ge t,
$$
where $u_s(t,x)$ is that density conditional on the value $x$ of the trait at birth ($K_t=x$), that is,
\begin{equation}
\label{eqn:ustx}
u_s(t,x):=\EE_{t,x} \left(d(s, K_{s})\ e^{-\int_t^s \, d(r,K_{r})dr}\right)\qquad s\ge t,
\end{equation}
where $\EE_{t,x}$ denotes the expectation associated to the distribution of $K$ started at time $t$ in state $x$.  
Now the Feynman-Kac formula ensures that $u_s$ is solution to 
\begin{equation}
\label{eqn:Feynman-Kac}
\frac{\partial u_s}{\partial t} (t,x) + L_t u_s (t,x) =   d(t,x)\, u_s(t,x),
\end{equation}
with terminal condition $u_s(s,x)= d(s,x)$. Specifically, when $K$ is the age, the initial trait value is $x=0$ and the age at $s$ of a species born at $t$ is $K_s= s-t$ so that 
\begin{equation}
\label{eqn:FK age}
g(t,s) = d(s, s-t)\ e^{-\int_t^s \, d(r,r-t)dr} \qquad s\ge t.
\end{equation}
\end{rem}

\begin{proof}[Proof of Theorem \ref{thm:generalmodel}]
Let $n\ge 1$ be an integer, and  let $h_1,\ldots, h_{n-1}$ be elements of $(0,T)$. Assume $N_T\ge n$, and condition on $H_i = h_i$ for $i=1,\ldots, n-1$. We are going to prove that the conditional law of $H_n$ is given by 
\begin{equation}
\label{to prove}
\PP(H_n>t) = \exp\left(- \int_{T-t}^ T b(s)\, (1-q(s))\,ds\right)\qquad t\in[0,T],
\end{equation}
which will show that $H_n$ is independent of $H_1, \ldots, H_{n-1}$ and has $W$ as inverse tail distribution. This result yields the theorem by induction. Note that conditonal on $N_T\ge n$, $N_T$ exactly equals $n$ iff $H_n >T$.

Label by $0,1,\ldots, n-1$ the individuals alive  at time $T$ in the order induced by the plane orientation of the tree, where daughters sprout to the right of their mother. In particular, $h_i$ is the coalescence time between individuals $i-1$ and $i$ ($1\le i\le n-1$).

We denote by $k$ the number of generations separating individuals $n-1$ from the progenitor. We let $x_k$ denote the time when individual $n-1$ was born, $x_{k-1}<x_k$ the time when her mother was born, and so on, until %$x_{1}$ the time when her most remote ancestor (except the progenitor) was born, and finally 
$x_0=0$ the birth time of the progenitor. By the orientation of the tree, there are (conditionally) deterministic indices $0=i_0  <\cdots <i_k=n-1$, such that $x_j= T-h_{i_j}$ (and $h_v< h_{i_j}$ for all $v\in\{i_{j-1}+1,\ldots, i_j -1\}$), so that conditional on $H_i = h_i$ for $i=1,\ldots, n-1$, the times $x_0,\ldots, x_k$ are deterministic. 

By the orientation of the tree again, apart from the individuals already labelled, individuals alive at $T$ descend from births occurring during one of the time intervals $I_j:=[x_j, x_{j+1})$, where $x_{k+1}:=T$ for convenience. On each of these time intervals, births occur at rate $b(t)$, and so by thinning, successful births, i.e., births with alive descendance at time $T$, occur at rate $b(t)\, (1-q(t))$. But conditional on the $(x_j)$, all the ancestors of individual $n-1$ (including her) independently give birth on their corresponding interval $I_j$. Then if $A$ denotes any subset of $[0,T)$, the number $N(A)$ of successful births occurring during $A$ is the sum 
$$
N(A) = \sum_{j=0}^k N(A\cap I_j),
$$
where the random numbers $N(A\cap I_j)$ are independent. Now $N(A\cap I_j)$ is a Poisson random variable with parameter $\int_{A\cap I_j} b(t)\, (1-q(t))\, dt$. As a consequence, $N(A)$ is a Poisson random variable with parameter
$$
\sum_{j=0}^k \int_{A\cap I_j} b(t)\, (1-q(t))\, dt = \int_{A} b(t)\, (1-q(t))\, dt.
$$
The proof finishes noticing that $H_n>t$ iff $N([T-t, T)) =0$, which occurs with the probability stated in \eqref{to prove}.
\end{proof}
\begin{exo}
When there is no trait dependence of the death rate, the process is merely a time-inhomogeneous birth--death process. Prove that in this case $W$ is given by
$$
W(t) =1+\int_{T-t}^T b(s)\,e^{\int_s^T  \,r(u)du}\,ds,
$$
where $r(t):=b(t)-d(t)$.
In particular, when rates do not depend on time, the process is a linear birth--death process with birth rate $b$ and death rate $d$. The last formula then boils down to
\begin{equation}
\label{eqn:Markovian scale}
W(t)=\begin{cases}
1 + \frac{b}{r}\big(e^{rt}- 1\big) & \text{if } r\not=0 \\
1+b t & \text{if }r=0.
\end{cases}
\end{equation}
\end{exo}
\noindent
Fig \ref{fig:coalescentdensity} shows the density $W'/W^2$ (recall $W(t)=1/\PP(H\ge t)$) of coalescence times in the pure-birth case  ($d=0$) and in the critical case ($b=d$).
\begin{figure}[!ht]
\centering
\includegraphics[width=9cm]{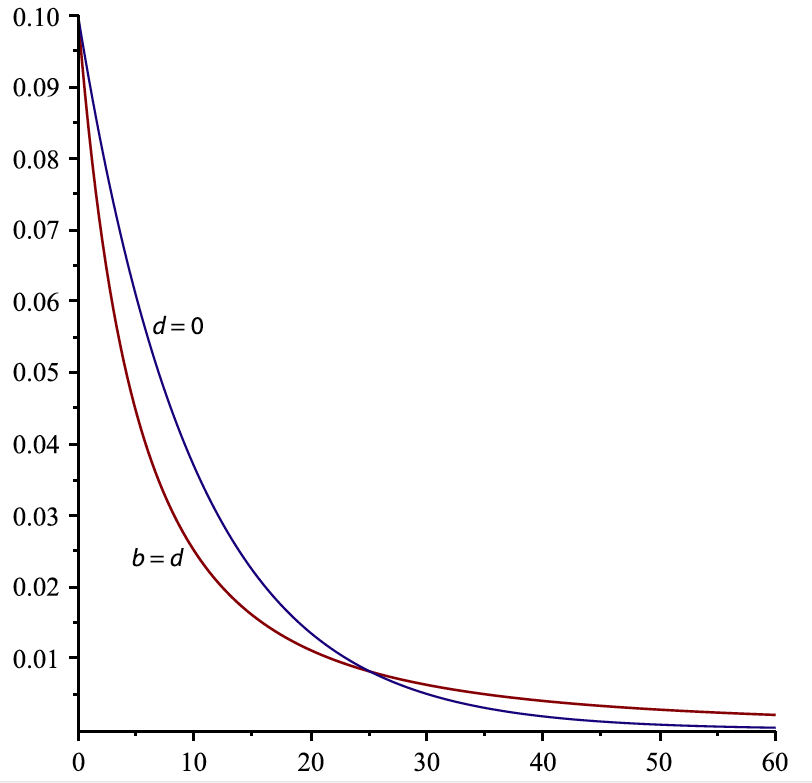}
\caption{The common density function $f=W'/W^2$ of the node depths of the reduced tree for a birth--death process with  constant rates $b$ and $d$. In blue, the pure birth case $f(t)=b\,e^{-bt}$ ($b=0.1$ in the figure); in red, the critical case $f(t) = bt/(1+bt)^2$ ($b=d=0.1$ in the figure). The critical process gives a density with a faster decay initially, but has a heavier tail than for the pure-birth model. This figure is taken from  \citet{LSt13}.}
\label{fig:coalescentdensity}
\end{figure}
\begin{exo}
When there is no time-dependence of the birth and death rates, and the trait is chosen to be the age, we are back to the splitting tree model. Here, $g(t,s)\equiv g(s-t)$ and $\Lambda (dr) = b\,g(r)\, dr$. First prove thanks to \eqref{eqn:hidden convolution} that 
\begin{equation}
\label{eqn:convolution}
W' = b\,\left( W - W\star g\right),
\end{equation}
and then recover that the Laplace transform of $W$ indeed is $1/\psi$ (as in Eq \ref{eqn:LTW}), where here $\psi$ can be written as $\psi(\lambda)=\lambda-b\int_0^\infty\left(1-e^{-\lambda r}\right)\,g(r)\,(dr)$. 

In passing, Eq \eqref{eqn:convolution} offers to compute the pair $(W,W')$ by solving numerically a 2D integro-differential equation (Eq \ref{eqn:convolution} along with $W(t)=1+\int_0^t W'(s)\, ds$) instead of inverting the Laplace transform of $W$, which can be computationally tricky.
\end{exo}
\begin{exo}
Prove that the shape of an ultrametric tree associated with a coalescent point process conditioned to have $n$ tips is always $\Pnerm$.
\end{exo}

\section{Applications}

In this section, we wish to give a taste of some recent applications of contour processes and coalescent point processes in evolutionary biology and in epidemiology. They rely in particular on the property that the density (or likelihood) of a given ultrametric tree $\tr$ with age $T$ and node depths $(h_i)$, seen as the reduced tree of a tree generated under one of the models displayed in the last section, is in product form
$$
\Lc(\tr) = p(T)\,\prod_i f(h_i),
$$
where $p$ and $f$ have to be computed in terms of the model ingredients, in particular $f=W'/W^2$ is the density of a random node depth.

\subsection{Bottlenecks and missing tips}
\label{subsec:bottlenecks}
Let $\tr$ be a rooted $\R$-tree interpreted as the genealogy of some population. We want to model the fact that at some fixed time point $t$, a macroscopic proportion say $p$ of the population, is killed (and its entire descendance as well of course). In population genetics, such events are called \emph{bottlenecks}, whereas in phylogenetics they model \emph{mass extinctions}. When $t=T$ is present time, this procedure is meant to model \emph{incomplete sampling} (or contemporary extinctions, see next section).

%If we are only interested in the reduced tree spanned by the population alive at time $T$, 
There are two reduced trees to consider, the reduced tree \emph{ex ante}, spanned by individuals alive at time $T$ in the absence of bottleneck, and the reduced tree \emph{ex post}, spanned by individuals alive even in the presence of the bottleneck. The second one is of course included in the first one.

Note that the reduced tree \emph{ex post} is not affected by lineages that do not even make it to the present in the absence of bottleneck, so we can equally assume that the bottleneck is only applied to the reduced tree \emph{ex ante}. See Fig \ref{fig:bottlenecks}.

As soon as the reduced tree is compact it is a comb metric space, and if $t<T$ we can model the bottleneck by simply disconnecting each lineage of the reduced tree \emph{ex ante} at distance $t$ from the root, independently with probability $p$. The next statement ensures that this operation preserves the CPP property.

\begin{prop}[\citealt{LS13}]
\label{prop:bottlenecks}
Start with a CPP with inverse tail distribution $W$. Add bottlenecks with survival probabilities $\varepsilon_1,\ldots,\varepsilon_k$ at times $T-s_1>\ldots>T-s_k$ (where $s_1>0$ and $s_k<T$). Then conditional on survival, the reduced tree \emph{ex post} is again a coalescent point process, with inverse tail distribution $W_\varepsilon$ given by
\begin{equation}
\label{eqn:bottlenecks}
W_\varepsilon (t)=
\varepsilon_1\cdots\varepsilon_m\,W(t)+\sum_{j=1}^m (1-\varepsilon_j)\,\varepsilon_1\cdots\varepsilon_{j-1}\,W(s_j)\qquad t\in[s_m, s_{m+1}], 0\le m\le  k,
\end{equation}
where $s_0:=0$ and $s_{k+1}=T$ (empty sum is zero, empty product is 1).
\end{prop}
\noindent
In the case of a finite number of tips, this formula can also include sampling with probability $p$ by adding a bottleneck with $s_0 =0$ and $\varepsilon_0=p$, resulting in 
$$
W_\varepsilon (t)=
\varepsilon_0\cdots\varepsilon_m\,W(t)+\sum_{j=0}^m (1-\varepsilon_j)\,\varepsilon_0\cdots\varepsilon_{j-1}\,W(s_j)\qquad t\in[s_m, s_{m+1}], 0\le m\le  k,
$$
which boils down to $W_\varepsilon=1-p+ pW$ when $k=0$ (since $W(0)=1$).

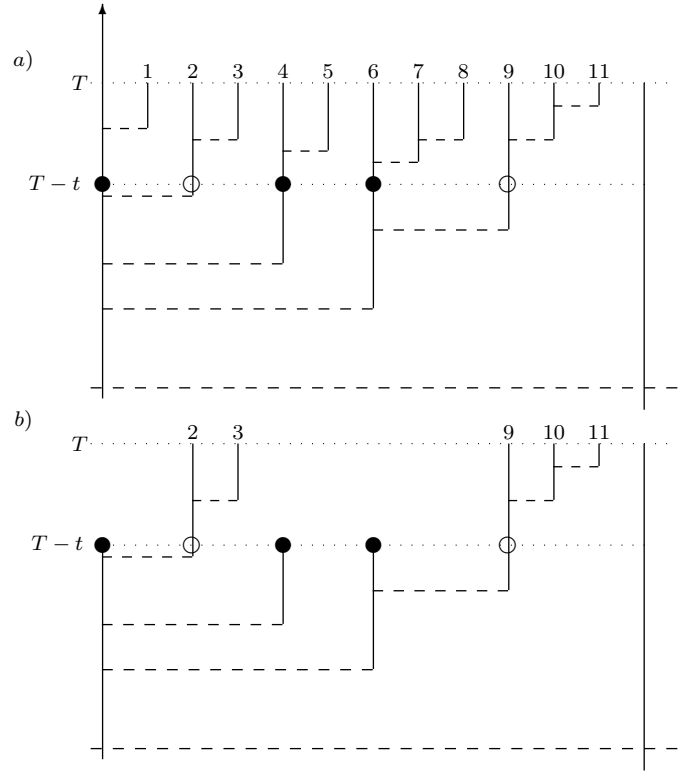
\begin{figure}[ht]
\centering
\unitlength 1.5mm % = 3.414pt
%TeXCAD (http://texcad.sf.net/) Picture. File: [bottlenecks.tex]. Options on following lines.
%\grade{\on}
%\emlines{\off}
%\epic{\off}
%\beziermacro{\on}
%\reduce{\on}
%\snapping{\off}
%\quality{8.000}
%\graddiff{0.005}
%\snapasp{1}
%\zoom{8.0000}
\linethickness{0.5pt}
\ifx\plotpoint\undefined\newsavebox{\plotpoint}\fi % GNUPLOT compatibility
%\begin{picture}(58,78)(0,-10)
\begin{picture}(58,78)(0,5)
\put(7,38){\vector(0,1){35}}
\put(5,34){\makebox(0,0)[cc]{\scriptsize $T$}}
\put(5,66){\makebox(0,0)[cc]{\scriptsize $T$}}
\put(0,36){\makebox(0,0)[cc]{\scriptsize $b)$}}
\put(0,68){\makebox(0,0)[cc]{\scriptsize $a)$}}

%\dashline{1}(23,18)(7,18)
\put(22.941,17.941){\line(-1,0){.9412}}
\put(21.059,17.941){\line(-1,0){.9412}}
\put(19.177,17.941){\line(-1,0){.9412}}
\put(17.294,17.941){\line(-1,0){.9412}}
\put(15.412,17.941){\line(-1,0){.9412}}
\put(13.53,17.941){\line(-1,0){.9412}}
\put(11.647,17.941){\line(-1,0){.9412}}
\put(9.765,17.941){\line(-1,0){.9412}}
\put(7.883,17.941){\line(-1,0){.9412}}
%\end
%\dashline{1}(23,50)(7,50)
\put(22.941,49.941){\line(-1,0){.9412}}
\put(21.059,49.941){\line(-1,0){.9412}}
\put(19.177,49.941){\line(-1,0){.9412}}
\put(17.294,49.941){\line(-1,0){.9412}}
\put(15.412,49.941){\line(-1,0){.9412}}
\put(13.53,49.941){\line(-1,0){.9412}}
\put(11.647,49.941){\line(-1,0){.9412}}
\put(9.765,49.941){\line(-1,0){.9412}}
\put(7.883,49.941){\line(-1,0){.9412}}
%\end
\put(15,24){\line(0,1){10}}
\put(15,56){\line(0,1){10}}
\put(23,50){\line(0,1){16}}
\put(39,61){\line(0,1){5}}
\put(55,34){\line(0,-1){29}}
\put(55,66){\line(0,-1){29}}
\put(11,66){\line(0,-1){4}}
\put(19,34){\line(0,-1){5}}
\put(19,66){\line(0,-1){5}}
\put(27,66){\line(0,-1){6}}
\put(31,66){\line(0,-1){20}}
\put(35,66){\line(0,-1){7}}
\put(43,34){\line(0,-1){13}}
\put(43,66){\line(0,-1){13}}
\put(47,34){\line(0,-1){5}}
\put(47,66){\line(0,-1){5}}
\put(51,34){\line(0,-1){2}}
\put(51,66){\line(0,-1){2}}
%\dashline{1}(11,62)(7,62)
\put(10.941,61.941){\line(-1,0){.8}}
\put(9.341,61.941){\line(-1,0){.8}}
\put(7.741,61.941){\line(-1,0){.8}}
%\end
%\dashline{1}(15,24)(7,24)
\put(14.941,23.941){\line(-1,0){.8889}}
\put(13.164,23.941){\line(-1,0){.8889}}
\put(11.386,23.941){\line(-1,0){.8889}}
\put(9.608,23.941){\line(-1,0){.8889}}
\put(7.83,23.941){\line(-1,0){.8889}}
%\end
%\dashline{1}(15,56)(7,56)
\put(14.941,55.941){\line(-1,0){.8889}}
\put(13.164,55.941){\line(-1,0){.8889}}
\put(11.386,55.941){\line(-1,0){.8889}}
\put(9.608,55.941){\line(-1,0){.8889}}
\put(7.83,55.941){\line(-1,0){.8889}}
%\end
%\dashline{1}(19,29)(15,29)
\put(18.941,28.941){\line(-1,0){.8}}
\put(17.341,28.941){\line(-1,0){.8}}
\put(15.741,28.941){\line(-1,0){.8}}
%\end
%\dashline{1}(19,61)(15,61)
\put(18.941,60.941){\line(-1,0){.8}}
\put(17.341,60.941){\line(-1,0){.8}}
\put(15.741,60.941){\line(-1,0){.8}}
%\end
%\dashline{1}(27,60)(23,60)
\put(26.941,59.941){\line(-1,0){.8}}
\put(25.341,59.941){\line(-1,0){.8}}
\put(23.741,59.941){\line(-1,0){.8}}
%\end
%\dashline{1}(31,14)(7,14)
\put(30.941,13.941){\line(-1,0){.96}}
\put(29.021,13.941){\line(-1,0){.96}}
\put(27.101,13.941){\line(-1,0){.96}}
\put(25.181,13.941){\line(-1,0){.96}}
\put(23.261,13.941){\line(-1,0){.96}}
\put(21.341,13.941){\line(-1,0){.96}}
\put(19.421,13.941){\line(-1,0){.96}}
\put(17.501,13.941){\line(-1,0){.96}}
\put(15.581,13.941){\line(-1,0){.96}}
\put(13.661,13.941){\line(-1,0){.96}}
\put(11.741,13.941){\line(-1,0){.96}}
\put(9.821,13.941){\line(-1,0){.96}}
\put(7.901,13.941){\line(-1,0){.96}}
%\end
%\dashline{1}(31,46)(7,46)
\put(30.941,45.941){\line(-1,0){.96}}
\put(29.021,45.941){\line(-1,0){.96}}
\put(27.101,45.941){\line(-1,0){.96}}
\put(25.181,45.941){\line(-1,0){.96}}
\put(23.261,45.941){\line(-1,0){.96}}
\put(21.341,45.941){\line(-1,0){.96}}
\put(19.421,45.941){\line(-1,0){.96}}
\put(17.501,45.941){\line(-1,0){.96}}
\put(15.581,45.941){\line(-1,0){.96}}
\put(13.661,45.941){\line(-1,0){.96}}
\put(11.741,45.941){\line(-1,0){.96}}
\put(9.821,45.941){\line(-1,0){.96}}
\put(7.901,45.941){\line(-1,0){.96}}
%\end
%\dashline{1}(35,59)(31,59)
\put(34.941,58.941){\line(-1,0){.8}}
\put(33.341,58.941){\line(-1,0){.8}}
\put(31.741,58.941){\line(-1,0){.8}}
%\end
%\dashline{1}(39,61)(35,61)
\put(38.941,60.941){\line(-1,0){.8}}
\put(37.341,60.941){\line(-1,0){.8}}
\put(35.741,60.941){\line(-1,0){.8}}
%\end
%\dashline{1}(43,21)(31,21)
\put(42.941,20.941){\line(-1,0){.9231}}
\put(41.095,20.941){\line(-1,0){.9231}}
\put(39.249,20.941){\line(-1,0){.9231}}
\put(37.403,20.941){\line(-1,0){.9231}}
\put(35.557,20.941){\line(-1,0){.9231}}
\put(33.711,20.941){\line(-1,0){.9231}}
\put(31.865,20.941){\line(-1,0){.9231}}
%\end
%\dashline{1}(43,53)(31,53)
\put(42.941,52.941){\line(-1,0){.9231}}
\put(41.095,52.941){\line(-1,0){.9231}}
\put(39.249,52.941){\line(-1,0){.9231}}
\put(37.403,52.941){\line(-1,0){.9231}}
\put(35.557,52.941){\line(-1,0){.9231}}
\put(33.711,52.941){\line(-1,0){.9231}}
\put(31.865,52.941){\line(-1,0){.9231}}
%\end
%\dashline{1}(47,29)(43,29)
\put(46.941,28.941){\line(-1,0){.8}}
\put(45.341,28.941){\line(-1,0){.8}}
\put(43.741,28.941){\line(-1,0){.8}}
%\end
%\dashline{1}(47,61)(43,61)
\put(46.941,60.941){\line(-1,0){.8}}
\put(45.341,60.941){\line(-1,0){.8}}
\put(43.741,60.941){\line(-1,0){.8}}
%\end
%\dashline{1}(51,32)(47,32)
\put(50.941,31.941){\line(-1,0){.8}}
\put(49.341,31.941){\line(-1,0){.8}}
\put(47.741,31.941){\line(-1,0){.8}}
%\end
%\dashline{1}(51,64)(47,64)
\put(50.941,63.941){\line(-1,0){.8}}
\put(49.341,63.941){\line(-1,0){.8}}
\put(47.741,63.941){\line(-1,0){.8}}
%\end
\put(7,25){\circle*{1.414}}
\put(7,57){\circle*{1.414}}
\put(23,25){\circle*{1.414}}
\put(23,57){\circle*{1.414}}
\put(31,25){\circle*{1.414}}
\put(31,57){\circle*{1.414}}
\put(14.875,25){\circle{1.414}}
\put(14.875,57){\circle{1.414}}
\put(42.875,25){\circle{1.414}}
\put(42.875,57){\circle{1.414}}
\put(11,67){\makebox(0,0)[cc]{\scriptsize $1$}}
\put(15,35){\makebox(0,0)[cc]{\scriptsize $2$}}
\put(15,67){\makebox(0,0)[cc]{\scriptsize $2$}}
\put(19,35){\makebox(0,0)[cc]{\scriptsize $3$}}
\put(19,67){\makebox(0,0)[cc]{\scriptsize $3$}}
\put(23,67){\makebox(0,0)[cc]{\scriptsize $4$}}
\put(27,67){\makebox(0,0)[cc]{\scriptsize $5$}}
\put(31,67){\makebox(0,0)[cc]{\scriptsize $6$}}
\put(35,67){\makebox(0,0)[cc]{\scriptsize $7$}}
\put(39,67){\makebox(0,0)[cc]{\scriptsize $8$}}
\put(43,35){\makebox(0,0)[cc]{\scriptsize $9$}}
\put(43,67){\makebox(0,0)[cc]{\scriptsize $9$}}
\put(47,35){\makebox(0,0)[cc]{\scriptsize $10$}}
\put(47,67){\makebox(0,0)[cc]{\scriptsize $10$}}
\put(51,35){\makebox(0,0)[cc]{\scriptsize $11$}}
\put(51,67){\makebox(0,0)[cc]{\scriptsize $11$}}
%\dashline{1}(6,7)(58,7)
\put(5.941,6.941){\line(1,0){.9811}}
\put(7.904,6.941){\line(1,0){.9811}}
\put(9.866,6.941){\line(1,0){.9811}}
\put(11.828,6.941){\line(1,0){.9811}}
\put(13.79,6.941){\line(1,0){.9811}}
\put(15.753,6.941){\line(1,0){.9811}}
\put(17.715,6.941){\line(1,0){.9811}}
\put(19.677,6.941){\line(1,0){.9811}}
\put(21.64,6.941){\line(1,0){.9811}}
\put(23.602,6.941){\line(1,0){.9811}}
\put(25.564,6.941){\line(1,0){.9811}}
\put(27.526,6.941){\line(1,0){.9811}}
\put(29.489,6.941){\line(1,0){.9811}}
\put(31.451,6.941){\line(1,0){.9811}}
\put(33.413,6.941){\line(1,0){.9811}}
\put(35.375,6.941){\line(1,0){.9811}}
\put(37.338,6.941){\line(1,0){.9811}}
\put(39.3,6.941){\line(1,0){.9811}}
\put(41.262,6.941){\line(1,0){.9811}}
\put(43.224,6.941){\line(1,0){.9811}}
\put(45.187,6.941){\line(1,0){.9811}}
\put(47.149,6.941){\line(1,0){.9811}}
\put(49.111,6.941){\line(1,0){.9811}}
\put(51.073,6.941){\line(1,0){.9811}}
\put(53.036,6.941){\line(1,0){.9811}}
\put(54.998,6.941){\line(1,0){.9811}}
\put(56.96,6.941){\line(1,0){.9811}}
%\end
%\dashline{1}(6,39)(58,39)
\put(5.941,38.941){\line(1,0){.9811}}
\put(7.904,38.941){\line(1,0){.9811}}
\put(9.866,38.941){\line(1,0){.9811}}
\put(11.828,38.941){\line(1,0){.9811}}
\put(13.79,38.941){\line(1,0){.9811}}
\put(15.753,38.941){\line(1,0){.9811}}
\put(17.715,38.941){\line(1,0){.9811}}
\put(19.677,38.941){\line(1,0){.9811}}
\put(21.64,38.941){\line(1,0){.9811}}
\put(23.602,38.941){\line(1,0){.9811}}
\put(25.564,38.941){\line(1,0){.9811}}
\put(27.526,38.941){\line(1,0){.9811}}
\put(29.489,38.941){\line(1,0){.9811}}
\put(31.451,38.941){\line(1,0){.9811}}
\put(33.413,38.941){\line(1,0){.9811}}
\put(35.375,38.941){\line(1,0){.9811}}
\put(37.338,38.941){\line(1,0){.9811}}
\put(39.3,38.941){\line(1,0){.9811}}
\put(41.262,38.941){\line(1,0){.9811}}
\put(43.224,38.941){\line(1,0){.9811}}
\put(45.187,38.941){\line(1,0){.9811}}
\put(47.149,38.941){\line(1,0){.9811}}
\put(49.111,38.941){\line(1,0){.9811}}
\put(51.073,38.941){\line(1,0){.9811}}
\put(53.036,38.941){\line(1,0){.9811}}
\put(54.998,38.941){\line(1,0){.9811}}
\put(56.96,38.941){\line(1,0){.9811}}
%\end
%\dottedline(6,34)(57,34)
\multiput(5.941,33.941)(.980769,0){53}{{\rule{.4pt}{.4pt}}}
%\end
%\dottedline(6,66)(57,66)
\multiput(5.941,65.941)(.980769,0){53}{{\rule{.4pt}{.4pt}}}
%\end
\put(7,6){\line(0,1){19}}
\put(23,18){\line(0,1){7}}
\put(31,14){\line(0,1){11}}
\put(2.75,57){\makebox(0,0)[cc]{\scriptsize $T-t$}}
\put(3,25.125){\makebox(0,0)[cc]{\scriptsize $T-t$}}
%\dottedline(7,57)(55,57)
\multiput(6.941,56.941)(.979592,0){50}{{\rule{.4pt}{.4pt}}}
%\end
%\dottedline(7,25)(55,25)
\multiput(6.941,24.941)(.979592,0){50}{{\rule{.4pt}{.4pt}}}
%\end
\end{picture}
\caption{Bottleneck at time point $T-t$, black dots disconnect lineages. a) Coalescent point process \emph{ex ante}, in the absence of bottleneck; b) Reduced tree \emph{ex post}, after passage of the bottleneck.}
\label{fig:bottlenecks}
\end{figure}

\begin{proof}
We characterize the effect of one bottleneck on a CPP with finitely many individuals at height $T$. 

Assume $k=1$ and $s_1\in(0,T)$. Recall that a CPP is defined thanks to a sequence of independent, identically distributed random variables $(H_i)$. We will see that the tree obtained after thinning is still a coalescent point process, defined from independent random variables, say $(B_i)$. Let $(e_i)$ be the i.i.d. Bernoulli random variables defined by $e_i=1$ if lineage $i$ survives the bottleneck (this has a meaning only if $H_i \ge s_1$; it happens with probability $\varepsilon_1$). By the orientation of the tree, a tip terminating a pendant edge with depth smaller than $s_1$ is kept alive iff the rightmost pending edge on its left with depth larger than $s_1$ survives. As a consequence, if $H_1<s_1$, then the first lineage is trivially still alive and its coalescence time with the left-hand ancestral lineage is $B:=H_1$. Otherwise, define $1=J_1< J_2<\cdots$ the indices of consecutive edges with depths larger than $s_1$. Then the first lineage kept alive after thinning is the least $J_m$ such that $e_{J_m}=1$, and its coalescence time with the ancestral lineage is $B:=\max(H_{J_1},\ldots, H_{J_m})$. By the independence property of coalescent point processes and by the independence of the Bernoulli random variables $(e_i)$, the new genealogy is obtained by a sequence of independent random variables $(B_i)$ all distributed as $B$. 

Let us specify the law of $B$. First, with probability $P(H< s_1)$, $P(B\in \cdot) =P(H\in \cdot \mid H< s_1)$. Second, with probability $P(H\ge s_1)$
$$
B\stackrel{(d)}{=}\max\{A_1,\ldots, A_M\},
$$
where the $A_i$'s are i.i.d. distributed as $H$ conditional on $H\ge s_1$ and $M$ is an independent (modified) geometric random variable, that is, $P(M=j)=\varepsilon_1(1-\varepsilon_1)^{j-1}$. Then for any $s\ge s_1$
$$
\frac{1}{P(B\ge s)}=\frac{1-\varepsilon_1}{P(H \ge s_1)} + \frac{\varepsilon_1}{P(H\ge s)}\qquad s\ge s_1. 
$$
Then if $W_\varepsilon$ denotes the inverse tail distribution of $B$, i.e., $W_\varepsilon(s):= 1/P(B\ge s)$, we have
$$
W_\varepsilon (s) = 
\begin{cases}
W(s) &\text{if } 0\le s\le s_1\\
(1-\varepsilon_1)W(s_1)+\varepsilon_1 W(s)&\text{if }s_1\le s \le t, 
\end{cases}
$$
where $W$ is the inverse tail distribution of $H$.
Iterating this procedure yields the result in Proposition \ref{prop:bottlenecks}.\end{proof}

\subsection{Loss of phylogenetic diversity}
In the context of the contemporary crisis of biodiversity, conservation biologists have proposed to quantify the loss of evolutionary heritage by the sum of branch lengths that disappear from the Tree of Life as new extinctions occur, that is, evolutionary heritage of a clade is quantified by the sum of its branch lengths, called \emph{phylogenetic diversity} (PD).   
Then a natural question to ask is the following. If a random, say 10\% of species from some given clade were to disappear in the next 100 years due to current high rates of extinction, how much evolutionary heritage would be lost?

`Not so much', asserted \citet{NM97} in a very much debated paper, where the tree of life was modeled by Kingman coalescent. `A lot more', replied \citet{MGS12}, in a paper where the tree of life was modeled by a Yule tree.

In \citet{LSt13}, we generalized their calculations to the case of a splitting tree with age $T$ and typical node depth $H$, where tips (contemporary species) are eliminated independently with probability $1-p$ (`field of bullets' model).
Let $G$ denote a geometric random variable with success probability $p$, let $(A_i)$ be a sequence of independent copies of $H$ conditioned on $H\le T$ and set
$$
B:=\max_{i=1,\ldots, G}A_i,
$$
%Then it is easy to see that $\frac{1}{\PP(B>t)}=1-p +p W(t)$ (where $W(t)=\frac{1}{\PP(H>t)}$),  
that is $B$ is the typical node depth of the tree after passage of the field of bullets (see previous paragraph on bottlenecks). 
%Section \ref{subsec:bottlenecks}). 

Conditional on the number of tips $n$ of the tree \emph{ex ante} and on the number $k_n$ of the tree \emph{ex post}, as $n\to\infty$ and $k_n/n\to p$, elementary SLLN-type arguments show that the ratio of the remaining PD to the old PD converges a.s. to
$$
\pi_T(p) = p\, \frac{\EE(B)}{\EE(A)}
$$
The ratio of remaining to old PD is always above the identity, corresponding to the case when the tree is star-like. This obviously holds also for $\pi_T$. In addition, it is not difficult to see that $\pi_T$ is always a concave, increasing function such that $\pi_T(0)=0$ and $\pi_T(1)=1$. 

\begin{figure}[!ht]
\centering
\includegraphics[width=10cm]{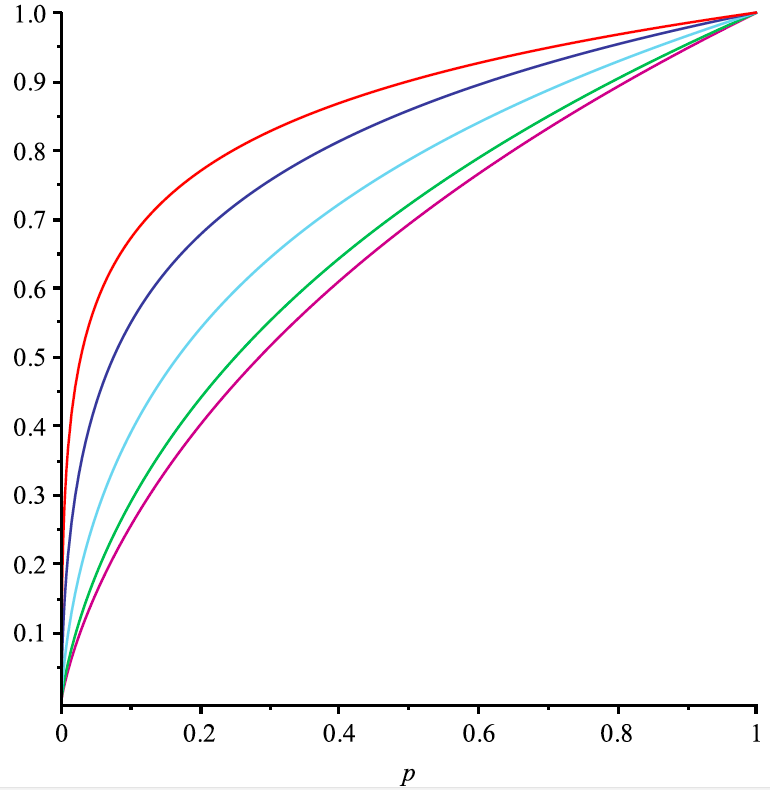}
\caption{Remaining fraction $\pi_\infty$ of phylogenetic diversity (PD) as a function of the probability $p$ of species survival to the mass extinction, for a constant-rate birth--death process. Observe the slow progression towards the unit step function (from pure birth to critical): $d/b=0$ (the lowest curve) and then $d/b=0.5,
0.9, 0.99, 0.999$. This figure is taken from  \citet{LSt13}.}
\label{fig:pd}
\end{figure}
\noindent
If we take $T=\infty$, we get 
$$
\pi_\infty(p) = p\, \frac{\intgen\frac{dt}{1-p+pW(t)}}{\intgen\frac{dt}{W(t)}}
$$
In the case of a birth--death tree with rates $b$ and $d$, with $r:=b-d>0$, we get
$$
\pi_\infty(p)=
\begin{cases}
\frac{dp}{bp-r}\frac{\ln(bp/r)}{\ln (b/r)}, & \mbox{ if }  b>r\not= bp;\\
 -\frac{p\ln (p)}{1-p}, &  \mbox{ if }  b=r> bp;\\
-\frac{1-p}{\ln (p)}, & \mbox{ if }  b>r= bp.  
\end{cases}
$$
\noindent
Fig \ref{fig:pd} shows the graph of $\pi_\infty$ for a range of birth and death rates $b>d$. Note that the more concave the better for evolutionary heritage. Also the larger $d/b<1$, the larger the remaining fraction of phylogenetic diversity, converging, but very slowly, to 1 as $d/b\to1$. %However, observe from the figure that this convergence is extremely slow. In addition, the more $d$ gets close to $b$, the less likely becomes the model...

\subsection{Do species age?}

In \citet{LAS14}, we have developed a framework to test the assumption that the extinction rate of a species remains constant all the way through its lifetime. Specifically, we have applied a maximal likelihood procedure to the recently published bird phylogeny (\citealt{JTJ12}), to infer the lifetime distribution of bird species, assuming they are Gamma distributed. We tested the assumption that the shape parameter $k$ of the Gamma r.v. equals 1 (exponential lifetimes, i.e., age-independent extinction rate) vs $k\not=1$ (age-dependence). This study generalizes previous works on the inference of diversification processes from reconstructed phylogenies (e.g., \citealt{NMH94, Nee06, Sta11}).

Our estimate of $k$ is much larger than 1, indicating that the extinction rate is not constant but increases with age. For the record, our estimate of the speciation rate is $0.11\ My^{-1}$ and our estimate of the mean species lifetime is $15\ My$.

\subsection{How long does speciation take?}
In most models of diversification, like the previous one, species are seen as particles that split instantaneously into two daughter species upon speciation. It is obvious that on the contrary speciation takes time, but the last assumption would still be relevant if the time speciation takes was negligible compared to a species lifetime. There is evidence that this is not the case, and some authors have recently proposed an alternative model of diversification, called \emph{protracted speciation}, meant to take this effect into account.

In this model, a species is described as an ensemble of populations, and as time passes, these populations diverge (genetically) gradually from each other. To not have to record all (phylo)genetic distances between all  populations composing each species, \cite{ER12} have proposed a model where each population passes through $k$ different stages of maturation, after which it becomes a so-called good species. This model produces quite realistic patterns in terms of phylogenetic balance and branching tempo, but an inference framework was missing. 

Since speciation stage is a non-heritable trait, Theorem \ref{thm:generalmodel} ensures that if speciation rate does not depend on speciation stage, the phylogeny produced by this model is a CPP. In \cite{LME14}, we have characterized the common distribution of node depths in this CPP.

In \citet{EML14} we have developed an inference framework that we tested against 46 bird clades. Individual parameters are difficult to infer, but the method is relatively good at inferring a composite parameter of interest, the duration of speciation. The duration of speciation is defined as the time it takes for a novel population to get a good species in its descendance.
The results are shown in Fig \ref{fig:protractedinf}.
 
\begin{figure}[!ht]
\centering
\includegraphics[width=8cm]{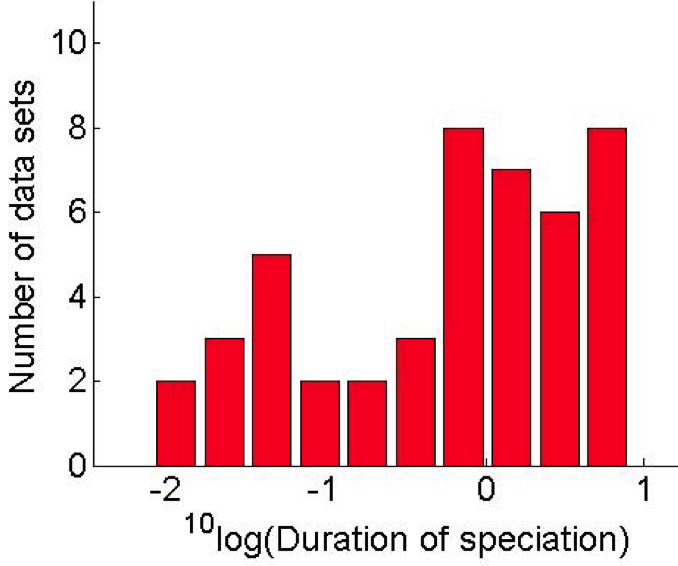}
\caption{Duration of speciation (in My). Distribution inferred from 46 bird clades, ranging from 10,000 years to 10 My. This figure is taken from \citet{EML14}.}
\label{fig:protractedinf}
\end{figure}

\subsection{Trees with random marks}

In this section, trees are endowed with marks to model sampling, detection or mutation. 

\subsubsection{The phylogeny of pathogens}

In this paragraph, we focus on a population of patients carrying (or not) an infectious disease. We stick to the framework of splitting trees, with the interpretation that a birth is a transmission event and a death is either a real death or the end of the infectious period (exit from the infective population). The branching property assumption is justified in the case of a well-mixed population where susceptibles are always in excess.

Our setting where lifetime distributions are not exponential is most attractive for diseases like flu or HIV, where the infectious period is known to be deterministic (flu) or heavy-tailed (HIV). 

In addition, we assume that patients are detected to be infective after some random time $D$, at which they are tagged by a mark (see figures). Upon detection (if $D$ is smaller than the infectious period/lifetime), the patients are assumed to exit the infective population, for example because they change their behaviour to avoid transmission or because they are treated and become non-infectious.

In \citet{LAS14}, we have considered the case where the transmission tree spanned by detected patients has been reconstructed, similarly as \citet{Sta10}. Actually, the data is not directly the transmission tree but the phylogenetic tree of the pathogens carried by patients, reconstructed thanks to biological samples taken from the patients upon detection. Fig \ref{fig:markedjcp} shows the oriented tree of the epidemic, with black dots showing detection events, along with its reduced tree and contour process. Patients can be labelled in the plane order, so we can define $S_i$ the detection time of patient $i$ and $R_i$ the coalescence time between patients $i-1$ and $i$.

By considering the jumping contour process of the epidemics, we have shown that the sequence $(S_i, R_i)$ is a killed Markov chain. The likelihood of a tree $\tr$ with coalescence times $(x_i)_{2\le i\le n}$ and detection times $(y_i)_{1\le i\le n}$ can be written in the form
$$
\Lc(\tr) =  g(y_1)\,k(y_n)\,\prod_{i=2}^n f(y_{i-1}; x_i, y_i),
$$
where $g$, $k$ and $f$ can be semi-explicitly expressed in terms of the model ingredients (law of infectious lifetimes, transmission rate, detection rate).
\begin{figure*}[ht]
\unitlength 2mm % = 8.536pt
\linethickness{0.2pt}
\ifx\plotpoint\undefined\newsavebox{\plotpoint}\fi % GNUPLOT compatibility
\begin{picture}(73,28)(0,0)
\put(9,15){\line(0,1){6}}
\put(13,19){\line(0,1){8}}
\put(17,12){\line(0,1){5}}
\put(21,15){\line(0,1){5}}
\put(25,17){\line(0,1){11}}
\put(29,10){\line(0,1){8}}
\put(33,16){\line(0,1){9}}
\put(41,14){\line(0,1){5}}
\put(45,16){\line(0,1){4}}
\put(17,17){\circle*{1.061}}
\put(9,21){\circle*{1.061}}
\put(54,21){\circle*{1.061}}
\put(62,17){\circle*{1.061}}
\put(70,20){\circle*{1.061}}
\put(37,20){\circle*{1.061}}
%\dashline{1}(9,15)(5,15)
\put(8.982,14.982){\line(-1,0){.8}}
\put(7.382,14.982){\line(-1,0){.8}}
\put(5.782,14.982){\line(-1,0){.8}}
%\end
%\dashline{1}(13,19)(9,19)
\put(12.982,18.982){\line(-1,0){.8}}
\put(11.382,18.982){\line(-1,0){.8}}
\put(9.782,18.982){\line(-1,0){.8}}
%\end
%\dashline{1}(17,12)(5,12)
\put(16.982,11.982){\line(-1,0){.9231}}
\put(15.136,11.982){\line(-1,0){.9231}}
\put(13.29,11.982){\line(-1,0){.9231}}
\put(11.444,11.982){\line(-1,0){.9231}}
\put(9.598,11.982){\line(-1,0){.9231}}
\put(7.752,11.982){\line(-1,0){.9231}}
\put(5.906,11.982){\line(-1,0){.9231}}
%\end
%\dashline{1}(21,15)(17,15)
\put(20.982,14.982){\line(-1,0){.8}}
\put(19.382,14.982){\line(-1,0){.8}}
\put(17.782,14.982){\line(-1,0){.8}}
%\end
%\dashline{1}(25,17)(21,17)
\put(24.982,16.982){\line(-1,0){.8}}
\put(23.382,16.982){\line(-1,0){.8}}
\put(21.782,16.982){\line(-1,0){.8}}
%\end
%\dashline{1}(29,10)(5,10)
\put(28.982,9.982){\line(-1,0){.96}}
\put(27.062,9.982){\line(-1,0){.96}}
\put(25.142,9.982){\line(-1,0){.96}}
\put(23.222,9.982){\line(-1,0){.96}}
\put(21.302,9.982){\line(-1,0){.96}}
\put(19.382,9.982){\line(-1,0){.96}}
\put(17.462,9.982){\line(-1,0){.96}}
\put(15.542,9.982){\line(-1,0){.96}}
\put(13.622,9.982){\line(-1,0){.96}}
\put(11.702,9.982){\line(-1,0){.96}}
\put(9.782,9.982){\line(-1,0){.96}}
\put(7.862,9.982){\line(-1,0){.96}}
\put(5.942,9.982){\line(-1,0){.96}}
%\end
%\dashline{1}(33,16)(29,16)
\put(32.982,15.982){\line(-1,0){.8}}
\put(31.382,15.982){\line(-1,0){.8}}
\put(29.782,15.982){\line(-1,0){.8}}
%\end
%\dashline{1}(37,12)(29,12)
\put(36.982,11.982){\line(-1,0){.8889}}
\put(35.205,11.982){\line(-1,0){.8889}}
\put(33.427,11.982){\line(-1,0){.8889}}
\put(31.649,11.982){\line(-1,0){.8889}}
\put(29.871,11.982){\line(-1,0){.8889}}
%\end
%\dashline{1}(41,14)(37,14)
\put(40.982,13.982){\line(-1,0){.8}}
\put(39.382,13.982){\line(-1,0){.8}}
\put(37.782,13.982){\line(-1,0){.8}}
%\end
%\dashline{1}(45,16)(41,16)
\put(44.982,15.982){\line(-1,0){.8}}
\put(43.382,15.982){\line(-1,0){.8}}
\put(41.782,15.982){\line(-1,0){.8}}
%\end
\put(3,23.375){\makebox(0,0)[cc]{\scriptsize $t$}}
\put(2.875,3.375){\makebox(0,0)[cc]{\scriptsize $0$}}
\put(51.875,3.375){\makebox(0,0)[cc]{\scriptsize $0$}}
\put(7,21){\makebox(0,0)[cc]{\scriptsize $1$}}
\put(52.5,21){\makebox(0,0)[cc]{\scriptsize $1$}}
\put(15,17){\makebox(0,0)[cc]{\scriptsize $2$}}
\put(60.5,17){\makebox(0,0)[cc]{\scriptsize $2$}}
\put(36,21){\makebox(0,0)[cc]{\scriptsize $3$}}
\put(68.625,21.125){\makebox(0,0)[cc]{\scriptsize $3$}}
\put(37,12){\line(0,1){8}}
\put(5,18){\line(0,-1){15}}
%\vector[both](10,21)(10,3)
\put(10,3){\vector(0,-1){.035}}\put(10,21){\vector(0,1){.035}}\put(10,21){\line(0,-1){18}}
%\end
%\vector[both](55.5,21)(55.5,3)
\put(55.5,3){\vector(0,-1){.035}}\put(55.5,21){\vector(0,1){.035}}\put(55.5,21){\line(0,-1){18}}
%\end
%\vector[both](18,17)(18,3)
\put(18,3){\vector(0,-1){.035}}\put(18,17){\vector(0,1){.035}}\put(18,17){\line(0,-1){14}}
%\end
%\vector[both](63.5,17)(63.5,3)
\put(63.5,3){\vector(0,-1){.035}}\put(63.5,17){\vector(0,1){.035}}\put(63.5,17){\line(0,-1){14}}
%\end
\put(38,3){\vector(0,-1){.07}}
\put(71.5,3){\vector(0,-1){.07}}
\put(38,20){\vector(0,1){.07}}
\put(71.5,20){\vector(0,1){.07}}
\put(38,20){\line(0,-1){17}}
\put(71.5,20){\line(0,-1){17}}
% added
\put(11,15){\makebox(0,0)[cc]{\scriptsize $S_1$}}
\put(56.5,15){\makebox(0,0)[cc]{\scriptsize $S_1$}}
\put(19.125,12){\makebox(0,0)[cc]{\scriptsize $S_2$}}
\put(64.625,12){\makebox(0,0)[cc]{\scriptsize $S_2$}}
\put(39,11){\makebox(0,0)[cc]{\scriptsize $S_3$}}
\put(72.5,11){\makebox(0,0)[cc]{\scriptsize $S_3$}}
%added
%\vector[both](12,12)(12,3)
\put(12,3){\vector(0,-1){.035}}\put(12,12){\vector(0,1){.035}}\put(12,12){\line(0,-1){9}}
%\end
%\vector[both](59,12)(59,3)
\put(59,3){\vector(0,-1){.035}}\put(59,12){\vector(0,1){.035}}\put(59,12){\line(0,-1){9}}
%\end
%\vector[both](20,10)(20,3)
\put(20,3){\vector(0,-1){.035}}\put(20,10){\vector(0,1){.035}}\put(20,10){\line(0,-1){7}}
%\end
%\vector[both](67,10)(67,3)
\put(67,3){\vector(0,-1){.035}}\put(67,10){\vector(0,1){.035}}\put(67,10){\line(0,-1){7}}
%\end
\put(13.375,7.375){\makebox(0,0)[cc]{\scriptsize $R_2$}}
\put(60.375,7.375){\makebox(0,0)[cc]{\scriptsize $R_2$}}
\put(21.125,6.75){\makebox(0,0)[cc]{\scriptsize $R_3$}}
\put(68.125,6.75){\makebox(0,0)[cc]{\scriptsize $R_3$}}
\put(4,3){\line(1,0){42}}
%\dottedline(4,23)(43,23)
\multiput(3.982,22.982)(.975,0){41}{{\rule{.2pt}{.2pt}}}
%\end
\put(54,3){\line(0,1){18}}
\put(62,12){\line(0,1){5}}
\put(70,10){\line(0,1){10}}
\put(19.125,26){\makebox(0,0)[cc]{\scriptsize a)}}
\put(62,24){\makebox(0,0)[cc]{\scriptsize b)}}
%\dashline{1}(62,12)(54,12)
\put(61.982,11.982){\line(-1,0){.8889}}
\put(60.205,11.982){\line(-1,0){.8889}}
\put(58.427,11.982){\line(-1,0){.8889}}
\put(56.649,11.982){\line(-1,0){.8889}}
\put(54.871,11.982){\line(-1,0){.8889}}
%\end
%\dashline{1}(70,10)(54,10)
\put(69.982,9.982){\line(-1,0){.9412}}
\put(68.1,9.982){\line(-1,0){.9412}}
\put(66.218,9.982){\line(-1,0){.9412}}
\put(64.335,9.982){\line(-1,0){.9412}}
\put(62.453,9.982){\line(-1,0){.9412}}
\put(60.571,9.982){\line(-1,0){.9412}}
\put(58.688,9.982){\line(-1,0){.9412}}
\put(56.806,9.982){\line(-1,0){.9412}}
\put(54.924,9.982){\line(-1,0){.9412}}
%\end
\put(53,3){\line(1,0){20}}
\end{picture}
\end{figure*}
\begin{figure}[ht]
\unitlength 2mm
\linethickness{0.2pt}
%TeXCAD (http://texcad.sf.net/) Picture. File: [markedjccp.tex]. Options on following lines.
%\grade{\on}
%\emlines{\off}
%\epic{\off}
%\beziermacro{\on}
%\reduce{\on}
%\snapping{\off}
%\quality{8.000}
%\graddiff{0.005}
%\snapasp{1}
%\zoom{8.0000}
\ifx\plotpoint\undefined\newsavebox{\plotpoint}\fi % GNUPLOT compatibility
\begin{picture}(79,27.625)(0,0)
\put(6,21){\circle*{1.061}}
\put(19,17){\circle*{1.061}}
\put(50,20){\circle*{1.061}}
% added
%added
%\emline(3,18)(6,15)
\multiput(3,18)(.03370787,-.03370787){89}{\line(0,-1){.03370787}}
%\end
\put(6,15){\line(0,1){6}}
%\emline(6,21)(8,19)
\multiput(6,21)(.03333333,-.03333333){60}{\line(0,-1){.03333333}}
%\end
\put(8,19){\line(0,1){4}}
\put(8,23){\line(1,-1){11}}
\put(19,12){\line(0,1){5}}
%\emline(19,17)(21,15)
\multiput(19,17)(.03333333,-.03333333){60}{\line(0,-1){.03333333}}
%\end
\put(21,15){\line(0,1){5}}
%\emline(21,20)(24,17)
\multiput(21,20)(.03370787,-.03370787){89}{\line(0,-1){.03370787}}
%\end
\put(24,17){\line(0,1){6}}
\put(24,23){\line(1,-1){13}}
\put(37,10){\line(0,1){8}}
%\emline(37,18)(39,16)
\multiput(37,18)(.03333333,-.03333333){60}{\line(0,-1){.03333333}}
%\end
\put(39,16){\line(0,1){7}}
\put(39,23){\line(1,-1){11}}
\put(50,12){\line(0,1){8}}
\put(50,20){\line(1,-1){6}}
\put(56,14){\line(0,1){5}}
%\emline(56,19)(59,16)
\multiput(56,19)(.03370787,-.03370787){89}{\line(0,-1){.03370787}}
%\end
\put(59,16){\line(0,1){4}}
\put(59,20){\line(1,-1){17}}
\put(3,3){\vector(0,1){23}}
%\dottedline(3,23)(65,23)
\multiput(2.93,22.93)(.984127,0){64}{{\rule{.4pt}{.4pt}}}
%\end
\put(1.625,23.25){\makebox(0,0)[cc]{\scriptsize $t$}}
\put(2,3){\vector(1,0){77}}
\put(6,2.5){\line(0,1){1}}
\put(8,2.5){\line(0,1){1}}
\put(19,2.5){\line(0,1){1}}
\put(24,2.5){\line(0,1){1}}
\put(39,2.5){\line(0,1){1}}
\put(50,2.5){\line(0,1){1}}
\put(76,2.5){\line(0,1){1}}
\put(3,1){\makebox(0,0)[cc]{\scriptsize $0$}}
\put(7.625,27.625){\makebox(0,0)[cc]{\scriptsize c)}}
\end{picture}
\caption{The transmission tree. a) Binary tree with marks (detection events); b) its reduced tree and c) its contour process.}
\label{fig:markedjcp}
\end{figure}
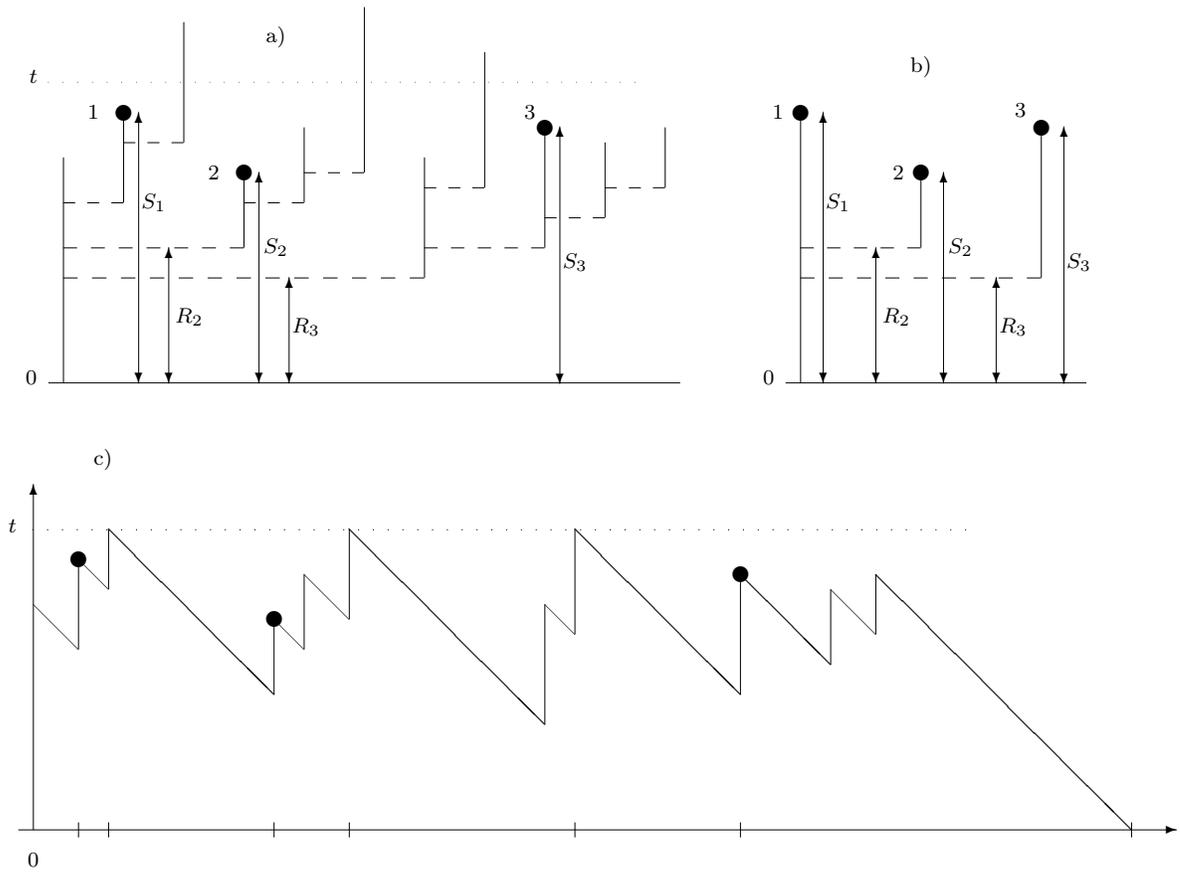
\subsubsection{The state of the epidemics at first detection}

We have considered  the same setting in \citet{LT13}, but for a different question, namely the structure of the epidemic at the first detection time. This question arises in the case of hospital-borne diseases due for example to bacterial antibiotic resistance. In this situation, everything can be known about all carriers of the disease, but only after the first detection of a case. At this random time, denoted $T$, everybody in the hospital is scanned and infected individuals are identified. Here the phylogeny of pathogens is not assumed to be known but patients' data like entrance dates or durations of stays are precisely known.

We assume that patients have i.i.d lengths of stay in the hospital, all distributed as some r.v. $K$. Conditional on infection, the length of stay of a patient is supposed to be a size-biased version of $K$. Finally, the transmission rate is $b$ and the detection rate per patient is denoted $\delta$. 

For individual $i$, set 
\begin{itemize}
\item $U_i:=$ time elapsed from entrance of the hospital up to infection
\item $A_i:=$ time elapsed from infection up to $T$
\item $R_i:=$ residual lifetime in the hospital after $T$.
\end{itemize}
See Fig \ref{fig:uar} for an example.
Set $m:=\EE(K)$ and let $\phi$ denote the inverse of the convex function 
$$
x\mapsto x - \frac{b}{m}\int_{(0,\infty]}(1-e^{-xy})\,\PP(K>y)\, dy.
$$
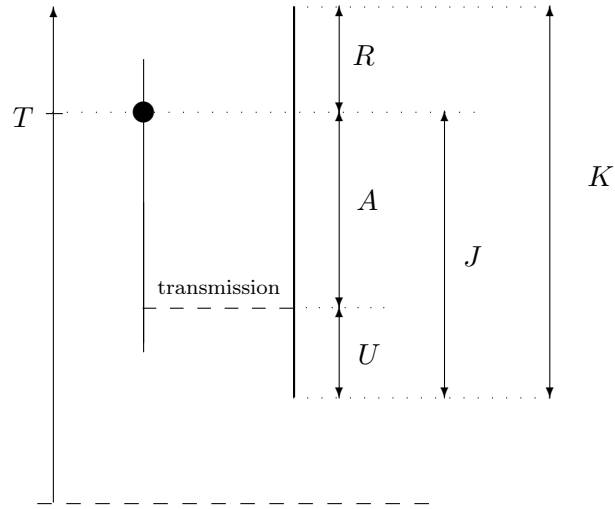
\begin{figure}[ht]
\centering
\unitlength 2mm % = 5.406pt
\linethickness{0.4pt}
%TeXCAD (http://texcad.sf.net/) Picture. File: [UAR.tex]. Options on following lines.
%\grade{\on}
%\emlines{\off}
%\epic{\off}
%\beziermacro{\on}
%\reduce{\on}
%\snapping{\off}
%\pvinsert{% Your \input, \def, etc. here}
%\quality{8.000}
%\graddiff{0.005}
%\snapasp{1}
%\zoom{8.0000}
\ifx\plotpoint\undefined\newsavebox{\plotpoint}\fi % GNUPLOT compatibility
\begin{picture}(38.5,35)(0,2)
\put(1,27.625){\makebox(0,0)[cc]{$T$}}
\put(2.5,27.875){\line(1,0){1.125}}
\put(3,2){\vector(0,1){33}}
%\dashline{1}(2,2)(29,2)
\put(1.93,1.93){\line(1,0){.9643}}
\put(3.858,1.93){\line(1,0){.9643}}
\put(5.787,1.93){\line(1,0){.9643}}
\put(7.715,1.93){\line(1,0){.9643}}
\put(9.644,1.93){\line(1,0){.9643}}
\put(11.573,1.93){\line(1,0){.9643}}
\put(13.501,1.93){\line(1,0){.9643}}
\put(15.43,1.93){\line(1,0){.9643}}
\put(17.358,1.93){\line(1,0){.9643}}
\put(19.287,1.93){\line(1,0){.9643}}
\put(21.215,1.93){\line(1,0){.9643}}
\put(23.144,1.93){\line(1,0){.9643}}
\put(25.073,1.93){\line(1,0){.9643}}
\put(27.001,1.93){\line(1,0){.9643}}
%\end
%\dashline{1}(9,15)(19,15)
\put(8.93,14.93){\line(1,0){.9091}}
\put(10.748,14.93){\line(1,0){.9091}}
\put(12.566,14.93){\line(1,0){.9091}}
\put(14.384,14.93){\line(1,0){.9091}}
\put(16.202,14.93){\line(1,0){.9091}}
\put(18.021,14.93){\line(1,0){.9091}}
%\end
\put(19,9){\line(0,1){26}}
\put(9,22){\line(0,-1){9.375}}
%\vector[both](22,35)(22,28)
\put(22,28){\vector(0,-1){.07}}\put(22,35){\vector(0,1){.07}}\put(22,35){\line(0,-1){7}}
%\end
%\vector[both](22,28)(22,15)
\put(22,15){\vector(0,-1){.07}}\put(22,28){\vector(0,1){.07}}\put(22,28){\line(0,-1){13}}
%\end
%\vector[both](22,15)(22,9)
\put(22,9){\vector(0,-1){.07}}\put(22,15){\vector(0,1){.07}}\put(22,15){\line(0,-1){6}}
%\end
\put(14,16.375){\makebox(0,0)[cc]{\scriptsize transmission}}
\put(24,11.875){\makebox(0,0)[cc]{$U$}}
\put(23.875,22.125){\makebox(0,0)[cc]{$A$}}
\put(23.625,31.75){\makebox(0,0)[cc]{$R$}}
\put(9,12){\line(0,1){19.5}}
\put(9,28){\circle*{1.414}}
\thicklines
\put(19,35){\line(0,-1){26}}
\thinlines
%\dottedline(19,35)(36,35)
\multiput(18.93,34.93)(.94444,0){19}{{\rule{.4pt}{.4pt}}}
%\end
%\dottedline(19,9)(36,9)
\multiput(18.93,8.93)(.94444,0){19}{{\rule{.4pt}{.4pt}}}
%\end
%\vector[both](29,28)(29,9)
\put(29,9){\vector(0,-1){.07}}\put(29,28){\vector(0,1){.07}}\put(29,28){\line(0,-1){19}}
%\end
%\vector[both](36,35)(36,9)
\put(36,9){\vector(0,-1){.07}}\put(36,35){\vector(0,1){.07}}\put(36,35){\line(0,-1){26}}
%\end
%\dottedline(3,28)(31,28)
\multiput(2.93,27.93)(.96552,0){30}{{\rule{.4pt}{.4pt}}}
%\end
%\dottedline(19,15)(25,15)
\multiput(18.93,14.93)(.85714,0){8}{{\rule{.4pt}{.4pt}}}
%\end
\put(30.25,18.375){\makebox(0,0)[lc]{$J$}}
\put(38.5,23.75){\makebox(0,0)[lc]{$K$}}
\end{picture}
\caption{The structure of the stay in the hospital of an infected patient. Some (other) patient is detected at time $T$ (random). The focal patient has total duration of stay $K=U+A+R$ (see text). }
\label{fig:uar}
\end{figure}
\noindent
Using Vervaat's transform applied to the path of the contour process, we were able to show that conditional on $N_T=n$, the triples $(U_i, A_i, R_i)$ of the $n$ infectives at time $T$ are i.i.d. distributed as 
\begin{multline*}
\EE(f(U,A,R)) =\\
\frac{b}{m}\,\frac{\phi(\delta)}{\phi(\delta)-\delta} \ \int_{u=0}^\infty du \int_{a=0}^\infty da \int_{z=u+a}^\infty \PP(K\in dz) \,e^{-\phi(\delta)a}\, f(u,a,z-u-a),
\end{multline*}
In particular, the times $J_i=U_i+A_i$ spent in the hospital up to time $T$ are i.i.d. distributed as the r.v.\ $J$
$$
\PP(J\in dy) = \frac{b/m}{\phi(\delta)-\delta}\  \PP(K>y) \,\big(1-e^{-\phi(\delta)y}\big)\, dy.
$$
The last formulae will allow us to infer the dynamical characteristics of hospital-borne disease epidemics from hospital data. This contrasts with the fact that inference is impossible from the sole numbers of cases found upon detection  (\citealt{TB09}), due to the geometric distribution of $N_T$ (recall however that here $T$ is random).

\subsubsection{Mutations}

Marks on a tree can also be used to model mutations. In population genetics, it is standard to assume that each new mutation occurs at a new site of the DNA sequence, the infinitely-many-sites assumption. The list of mutated sites of a sequence is called allele. One of the fundamental questions in population genetics is to interpret genetic data such as the number of individuals in a population carrying a specific mutation or a specific allele. Reciprocally, the number of mutations, or of alleles, carried by $k$ individuals in the population is called frequency spectrum by sites, or by alleles.

The frequency spectrum of neutral mutations (that is, mutations with no influence on the population dynamics) has been extensively studied for random genealogies arising from models with constant population size, culminating in so-called Ewens' sampling formula (\citealt{Ewe72}).
In a series of recent papers relying heavily on contour techniques, we have studied the frequency spectrum in branching genealogies (\citealt{Lam09, Lam11, CL12, CL13, R14, DAL16}).

\section{Perspectives}

For who has understood how to identify which forward-in-time processes generate trees whose reduced tree is a coalescent point process, and has learnt the procedure of characterizing $W$ and the coalescent density from the model ingredients, coalescent point processes are a very convenient tool:
\begin{itemize}
\item
They arise in a wide class of models;
\item
They generate robust patterns, that in particular are invariant under incomplete sampling and under the action of bottlenecks;
\item
The reconstructed tree has a particularly simple distribution, which is extremely fast to simulate, in contrast with the entire forward-in-time process, that may even not be Markovian;
\item
The inference of model parameters from the knowledge of the reconstructed tree can be done using low-tech statistical methods.
\end{itemize}
On the other hand, CPPs also have a number of shortcomings:
\begin{itemize}
\item The models in which CPP arise exclude some interesting features from the modeling point of view, in particular the trait/age-dependence of birth rates;
\item Among the robust patterns they generate, the shape of the reduced tree is always ERM, which is certainly not the rule in empirical genealogies/phylogenies;
%\item The models that we have considered so far to model phylogenies are lineage-based, %(to the exception of the protracted speciation model, where a species can be split out in a random number of populations),
 %in the sense that they ignore the dynamics of individuals inside the species, and they are neutral, in the sense that all particles are exchangeable.
\end{itemize} 
Currently, one of our main lines of research (mine but more generally that of the SMILE group \emph{-- Stochastic models for the inference of life evolution}, UPMC \& Collège de France) is to produce and study models that 
\begin{enumerate}
\item are grounded on the microscopic description of individuals, either at the ecological scale or at the genetic scale;
\item feature a small number of parameters, which can nevertheless be tuned so as to generate a wide range of different patterns, when the corresponding empirical patterns vary across datasets (e.g., species abundance distributions); 
\item generate robust patterns when the corresponding empirical patterns are well conserved across datasets (e.g., the MLE of $\beta$ in empirical phylogenies, that revolves around $-1$);
\item produce observable statistics (e.g., reconstructed trees) with computable likelihoods.  
\end{enumerate}
\noindent
%Criteria 2 and 4 (sometimes 1) are generally satisfied by models inducing coalescent point processes. %Markov branching models also satisfy 2 and 4. In the recent papers (\citealt{LM15, MLM15, SDLU}), we have proposed three novel models satisfying criteria 1 and 2. 

Criterion 3 is in general difficult to satisfy, especially simultaneously with 2.

% of coalescent in the context of peripatric speciation, based on the description of dynamical metapopulations, ending up in a structured coalescent where lineages can be either inner lineages (i.e., inside the mother population) or outer lineages (i.e., in one of a multitude of surrounding populations), can change types but only pairs of inner lineages can coalesce. This model satisfies properties 

\bibliographystyle{apalike}
%\bibliography{RefSaoPaulo} 
\bibliography{FromZotero}

\end{document}